\documentclass[11pt]{amsart}
\usepackage{amsmath, amssymb, amsthm, bbm}
\usepackage{enumitem}
\usepackage{tikz}
\usepackage{graphicx}
\usepackage{xcolor}
\usepackage[UKenglish]{isodate}
\usepackage{graphicx}
\usepackage[font=small,labelfont=it]{caption}
\usepackage{subcaption}
\usepackage{hyperref}
\usepackage{enumitem}
\usepackage{multirow}
\usepackage{fancyvrb}
\usepackage{geometry}
\usepackage{algorithm}
\usepackage{algpseudocode}
\usepackage{algcompatible}
\usepackage{mathtools}

\mathtoolsset{showonlyrefs}

\renewcommand{\phi}{\varphi}

\newcommand{\ud}{\mathrm{d}}
\newcommand{\eqdef}{\ensuremath{\stackrel{\mbox{\upshape\tiny def.}}{=}}}

\def \ud{\mathrm{d}}

\newtheorem*{theorem*}{Theorem}
\newtheorem{theorem}{Theorem}[section]

\newtheorem{proposition}[theorem]{Proposition}
\newtheorem{corollary}[theorem]{Corollary}
\newtheorem{remark}[theorem]{Remark}
\newtheorem{lemma}[theorem]{Lemma}

\newtheorem{definition}[theorem]{Definition}
\newtheorem{assumption}[theorem]{Assumption}

\makeatletter
\@namedef{subjclassname@2020}{\textup{2020}Mathematics Subject Classification} 
\makeatother

\title[NeuralChaos]{NeuralChaos: Optimal Adapted Approximation of Square Integrable Predictable Processes}


\author[Kratsios]{Anastasis Kratsios} 
\author[Livieri]{Giulia Livieri}
\author[Schmocker]{Philipp Schmocker} 
\address{A.\ Kratsios: Department of Mathematics, McMaster University and Vector Institute, 1280 Main Street West, Hamilton, Ontario, Canada L8S 4K1}
\email{\href{mailto:kratsioa@mcmaster.ca}{kratsioa@mcmaster.ca}}
\address{G.\ Livieri: The London School of Economics and Political Science, Dept.\ Statistics, Columbia House, 69 Aldwych, London WC2B 4RR.}
\email{\href{mailto:g.livieri@lse.ac.uk}{g.livieri@lse.ac.uk}}
\address{P.\ Schmocker: ETH Zurich, Department of Mathematics, R\"{a}mistrasse 101, 8092 Z\"{u}rich, Switzerland.}
\email{\href{mailto:philipp.schmocker@math.ethz.ch}{philipp.schmocker@math.ethz.ch}}
\date{\today}

\begin{document}
\begin{abstract}
We address fundamental challenges in representing and computing $\mathbb{R}^{d}$-valued predictable square-integrable processes over $[0,T]$, collected in the space $\mathcal{H}^2_T(\mathbb{R}^{d})$. These processes are central to continuous-time stochastic control, reinforcement learning, and mathematical finance. Although Wiener-chaos expansions offer strong theoretical tools, traditional computational methods are hindered by the need for large chaos dictionaries and high-order iterated integrals. To overcome these obstacles, we introduce NeuralChaos -- a neural operator architecture that produces elements of $\mathcal{H}^2_T(\mathbb{R}^{d})$ using only finitely many evaluations of the driving Brownian motion, while preserving predictability and square-integrability. We prove that NeuralChaos is dense in $\mathcal{H}^2_T(\mathbb{R}^{d})$ and achieves the best $N$-term chaoslet approximation rates for compressible and Malliavin--Sobolev regular processes. Moreover, compressibility is shown to be typical for processes from $\mathcal{H}^2_T(\mathbb{R}^{d})$ under non-degenerate sub-Gaussian sampling. In contrast, we show that finite-dimensional Markovian neural SDE models constitute a meagre and Gaussian-null subset in $\mathcal{H}^2_T(\mathbb{R}^{d})$, regardless of discretization, whereas compressible processes are generic. Numerical experiments on a stochastic optimal control problem and dynamic hedging highlight the practical effectiveness of our approach. Our results enable more efficient and expressive modeling in stochastic analysis and mathematical finance.
\end{abstract}

\maketitle
\section{Introduction}\label{sec::introduction}
Many operations in stochastic analysis, stochastic optimal control, and mathematical finance -- e.g., dynamic hedging, portfolio choice, volatility calibration, optimal execution, and stochastic control with path-dependent objectives -- require the approximation or optimization of predictable processes satisfying basic integrability requirements. It turns out that a natural ambient space for such problems is the Hilbert space of $\mathbb{R}^{d}$-valued predictable square-integrable processes considered, {\color{black} e.g.,} over the interval $[0,T]$, denoted by  $\mathcal{H}^2_T(\mathbb{R}^{d})$. Importantly, predictability encodes the causal information constraint faced by, e.g., a trading or control strategy, while square-integrability guarantees that the resulting stochastic integrals and objective functionals are well-defined. Therefore,  a computationally useful parametrization of strategies (or controls) should approximate rich classes of elements in $\mathcal{H}^2_T(\mathbb{R}^{d})$ without violating either of the just-mentioned two structural requirements.

Currently, there are effectively two approaches for bringing general theoretical solutions into practice via computationally viable methods. The first approach consists in focusing on stylized models -- e.g., one imposes Markovian or affine model structures that lead to finite-dimensional state variables -- which (often) admit analytically solvable closed-form expressions; see, e.g.,~\cite{black1973pricing,merton1973theory,heston1993closed,bertsimas1998execution,avellaneda2008limitorderbook,kallsen2010utility,gueant2013inventory,gerhold2014transaction,moreau2017trading,cartea2016closedformvwap,cartea2016orderflow,muhle2025dynamic} for a non-exhaustive list of references. Instead, the second approach uses non-parametric techniques that can address problems under more general and realistic sets of assumptions, e.g., path-dependent claims and non-Markovian volatility. At the forefront of the second approach, there is the recent boom in deep-learning approaches due to their compatibility with efficient (stochastic) optimization algorithms (e.g.,~\cite{polyak1992acceleration,kingma2014adam}); see, e.g.,~\cite{kratsios2020deep,zhang2019deeplob,yu2019barrierbsde_ssrn,horvath2021roughhedging,cao2021deephedgingrl,ning2021double,horvath2021deep,gonon2021deep,neufeld2022chaotic,gonon2023random,lillo2023analysis,benth2024pricing,
gonon2024deep,reppen2023deep,gauthier2025deep,gonon2025leveraging,wiedemann2025operator,alvarez2025neural,abi2025volatility,yang2026synchronizing,buehler2026sanos,lozano2026}, for, again, a non-exhaustive list of references. However, there is still a lack of such deep-learning models in the space $\mathcal{H}^2_T(\mathbb{R}^{d})$, and the state of the art perhaps still consists only of extensions of the Wiener chaos; e.g.~\cite[Lemma B.2]{alvarez2025neural}.  This creates a gap between the expressive Hilbert-space representation naturally suggested by stochastic analysis and the finite, trainable parameterisations used in modern numerical methods. 
\subsection{The current state of the art} 
Several existing results have established (quantitative) approximation guarantees for the solution of stochastic differential equations (SDE, henceforth). More precisely, let\footnote{cf.~Subsection \ref{subsec::notation} for the notation.} $d_{X} \in \mathbb{N}_{+}$ be the dimension of the state process, $D \in \mathbb{N}_{+}$ be the dimension of the driving Brownian motion, and $\varepsilon>0$ be the target approximation accuracy. Then, results such as \cite[Proposition 4.12 and Lemma 4.6]{gonon2023deep} imply that -- under the assumptions stated there -- a Markovian SDE solution can be approximated by a Neural-SDE type model with a number of parameters of order $O(\varepsilon^{-\max\{q,1/2\}})$, where $q>0$ depends on the dimension and regularity of the problem. In the worst-case estimates inherited from~\cite{yarotsky2018optimal} one obtains $q=d_{X}/(1+s)$, where $s>0$ is a smoothness parameter. On the other hand, on arbitrary time-horizons one may instead obtain a value of $q=1$ (which is dimension-independent); see \cite{kratsios2025generative}, and, also,~\cite{gierjatowicz2020,cohen2023}. Related deep-learning approximation results for discrete-time stochastic systems and traditional non-machine learning methods are available in~\cite{grigoryeva18,gonon2019reservoir,gonon2023approximation,gonon2021fading,arabpour2024low,galimberti2026designing} and~\cite{mullergronbach2002uniform,mullergronbach2004pointwise}, respectively. However, these results have as their natural object Markovian SDEs without a diverging state dimension and, therefore, are not typical in the space $\mathcal{H}^2_T(\mathbb{R}^{d})$. Indeed, we show, in Proposition \ref{prop:badnews}, that the class of generalized predictable Euler-Maruyama-type processes with finite, non-exploding state dimension is meagre {\color{black}{(i.e.\ membership to this class is unstable under small perturbations since the class has empty interior)}} in $\mathcal{H}^2_T(\mathbb{R}^{d})$ and this class has $\mu$-(almost surely) measure zero for every non-degenerate Gaussian measure on $\mathcal{H}^2_T(\mathbb{R}^{d})$; importantly,  this class contains the usual discrete-time approximations of Markovian SDEs (see Proposition \ref{prop:goodnews}). Moreover, Markovian (neural) SDEs are not always the natural objects for sequential decision problems arising in stochastic control, reinforcement learning, and mathematical finance. For example, linear backward SDE (BSDEs) with path-dependent terminal conditions (e.g.,~\cite[Proposition 2.2]{ElKarouiPEngQuenez1997Closedformprop2pt2}), variance-optimal hedging and recursive utility problems (e.g.,~\cite{duffie1992stochastic}), or time-consistent risk evaluation through (so-called) $g$-expectations (e.g.,~\cite{di2024fully}) lead to non-Markovian objects.\\
\noindent Finally, it is worth calling attention to the developments in kernelized continuous-time machine learning theory, which builds on (expected) path-signatures~\cite{kiraly2019kernels,morrill2021neural,salvi2021signature,chevyrev2022signature,cuchiero2025signature,philippuniversal2026,philippderivatives2026}. This theory is, however, primarily path-wise or distributional and is usually not formulated as a finite-dimensional trainable parameterizations of elements of $\mathcal{H}^2_T(\mathbb{R}^{d})$. The current paper, instead, builds on the stochastic predecessor of path-signatures, namely the Wiener chaos and its related Malliavin--Sobolev calculus, which provide an orthogonal graded description of Brownian square-integrable objects; it is, at this point, thoroughly developed (e.g., \cite{privault1994chaotic,nualart06,diNunno2009malliavin,privault2009stochastic,peccati2011wiener}).

Our result can be viewed as an \textit{operator-learning} result; i.e.\ a result on approximating non-linear operators between infinite-dimensional topological vector spaces; cf.~\cite{li2021fourier,de2022deep,benth2023neural,furuya2023globally,lanthaler2023operator,neufeld2023universal,adcock2024optimal,furuya2024can,lanthaler2025nonlocality,galimberti2026designing,lanthaler2026parametric,schwab2026deep,philippuniversal2026,philippderivatives2026} which have several applications in mathematical finance~\cite{benth2024pricing,wiedemann2025operator,neufeld2024solving,kratsios2025generative} and game-theory/control~\cite{alvarez2025neural,furuya2026polynomialscalingpossibleneural,furuya2025one}.
Indeed, for any fixed predictable process $u_{\cdot}$, the objective is to learn the non-linear operator mapping a Brownian path $W_{\cdot}(\omega)$ to the corresponding realization of $u_{\cdot}(\omega)$ with the approximation criteria being approximation minimizing the time-averaged variance across most paths; i.e.\ the $\mathcal{H}_T^2(\mathbb{R}^d)$-norm.  

\subsection{Limitations}
The goal of the current paper is not simply to prove a universal approximation theorem, which is rather easily obtained, but, instead, to focus on the more delicate property of \emph{minimax-optimal approximation rates} in the best $N$-term sense of non-linear approximation of~\cite{devore1998nonlinear}, for well-behaved non-Markovian processes in $\mathcal{H}^2_T(\mathbb{R}^{d})$, while respecting the predictable and square-integrable structure of $\mathcal{H}^2_T(\mathbb{R}^{d})$. Critically, we seek a computationally \emph{lightweight} solution that: 1) avoids costly operations, such as higher-order quadrature of iterated integrals; 2) avoids the poor rates of kernel methods identified in~\cite[Theorem 2]{yarotsky2018optimal} in the finite-dimensional case; 3) embraces the non-linear representation capacity of deep-learning technologies; and, finally, 4) remains compatible with modern stochastic optimization pipelines (e.g., \cite{duchi2009efficient,zeiler2012adadelta}). Ideally, such a solution should remain as close as possible to out-of-the-box multilayer perceptrons (MLPs, in short) since any deep-learning-based solution would then be straightforward to implement in a modern Artificial Intelligence (AI, in short) software (e.g., \texttt{Tensorflow} or \texttt{PyTorch}). Our solution has all of these properties, thereby creating new pathways for efficient and expressive modeling in mathematical finance and stochastic analysis.

\subsection{Our results}\label{subsec::our_results} In the current paper, we introduce $\operatorname{NeuralChaos}$, a universal class of deep-learning powered predictable and square-integrable processes in $\mathcal{H}^2_T(\mathbb{R}^{d})$, which preserves by construction the mandatory adaptedness to the information structure encoded by predictability, and the square-integrability required for downstream usage (e.g., when parameterising SDEs by predictable controls).

Our construction is based on three basic operations (cf.~Figure \ref{fig:model_workflow}). First, we sample the driving Brownian motion at finitely many deterministic times $0=t_0<t_1<\dots<t_M = T$. Second, these sampled values are processed through lower-triangular linear lifts and row-wise neural-network heads. The lower-triangular structure guarantees that the $m^{th}$ row of the lifted feature matrix depends only on Brownian samples available up to the previous time instant. Finally, the resulting adapted random variables are assembled into a continuous-time process through deterministic causal time masks. Consequently, the value of the output process at time $t$ depends only on Brownian information available up to that time. Importantly, no iterated stochastic integral is computed explicitly.

\begin{figure}[ht!]
    \centering
    \includegraphics[width=.75\linewidth]{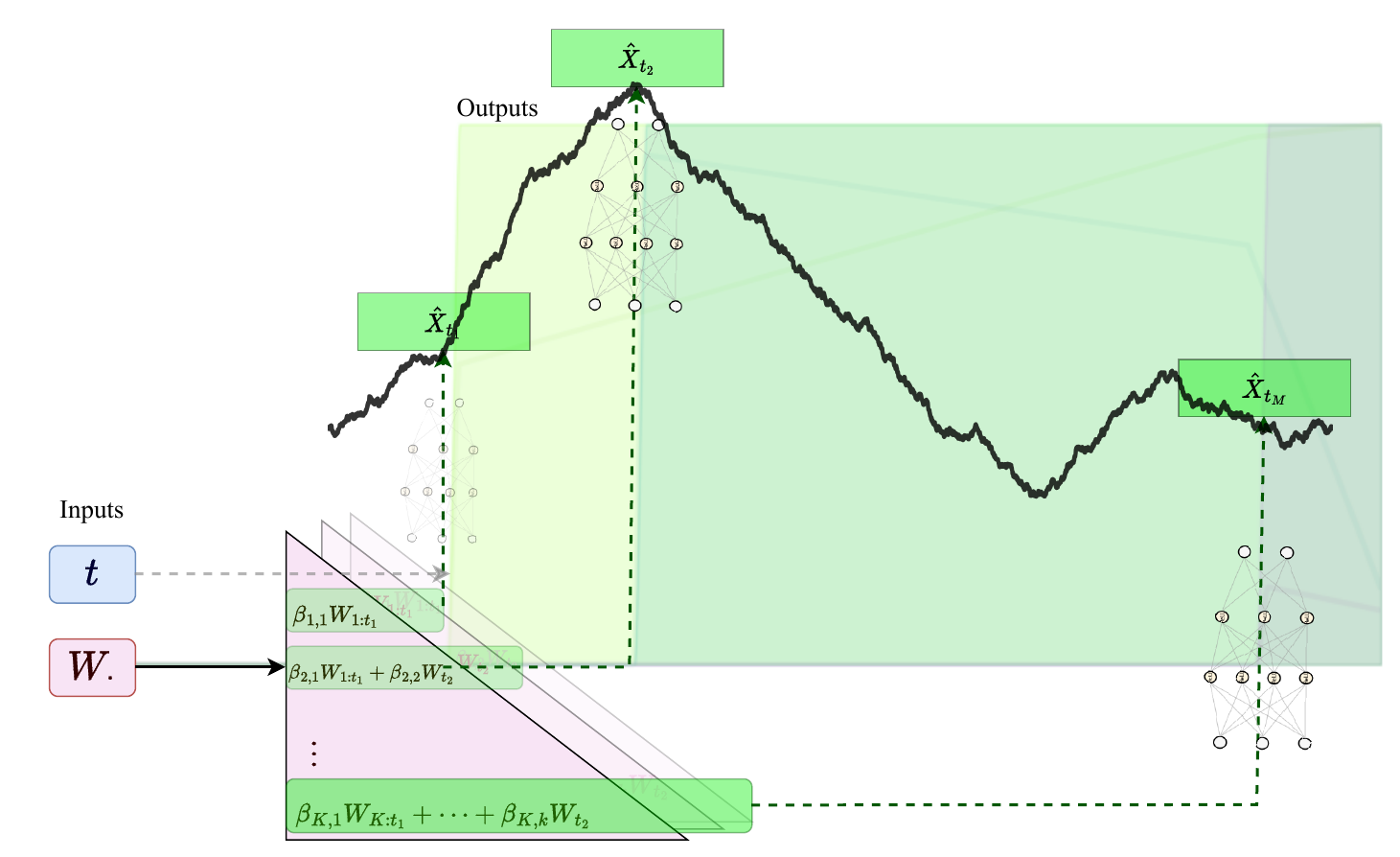}  
    \\ 
    \caption{Our \textbf{NeuralChaos} architecture (for details cf.\ Section~\ref{sec::section_3}):
    The input Brownian path is first sampled at the time-points $0=t_0<t_1<\cdots<t_M=T.$ (cf.~Figure~\ref{fig:bm_path_evals}), producing the path-wise evaluations $W_{t_0},\dots, W_{t_{M-1}}$; notice that no iterated integrals are explicitly computed, only implicitly.
    \hfill\\
    \textbf{Lifting Channels:} 
    These evaluations are mapped, via a sequence of $K$ lower-triangular matrices, into a random feature matrix whose $m^{th}$ (lifted) row is adapted to $\mathcal{F}_{t_{m-1}}$.
    \hfill\\
    \textbf{Prediction Heads:} 
    The lifted feature matrix is processed row by row by $M$ neural heads $\mathfrak{h}_m:\mathbb R^{KD}\to\mathbb R^d$. The $m$-th head acts only on the $m$-th lifted row and produces an adapted random vector $Y_m\eqdef \mathfrak{h}_m\left(L_{\mathrm{lift}}(\boldsymbol{W}^{\boldsymbol{t}})_{m:}\right) \in L^2(\mathcal F_{t_{m-1}};\mathbb R^d)$, where $\boldsymbol{W}^{\boldsymbol{t}}\eqdef[W_{t_0}^{\top},W_{t_1}^{\top},\ldots,W_{t_{M-1}}^{\top}]^{\top} \in \mathbb{R}^{M \times D}$. Thus $Y_m$ depends only on Brownian samples available up to $t_{m-1}$, and it does not contain future Brownian information.
    \hfill\\
    \textbf{Continuous-Time Assembly Via Masks:} 
    The sequence of random vectors $Y_1,\dots,Y_M$ is then assembled into a continuous-time output through the $M$ temporal masks in Figure~\ref{fig:temporal_masks}. Each mask is parameterized by a ReLU MLP, is $0$ before time $t_k$, and transitions to nearly $1$ by time $t_k+\delta$ ($\delta>0$), where the sharpness is controlled by a temperature parameter. Therefore the $m$-th row output can contribute only after the time at which the information used to compute it is available. This enforces causality by construction.}
    \label{fig:model_workflow}
\end{figure}

Our first main (positive) result is a universal approximation theorem in $\mathcal{H}^2_T(\mathbb{R}^{d})$; see Theorem~\ref{thrm:happytimes}. The proof hinges on an (adapted) chaoslet expansion of predictable processes. In particular, each chaoslet is a product of a deterministic causal time atom and a Hermite polynomial evaluated at finitely many Gaussian Haar coordinates of the Brownian path. Importantly, since these Gaussian coordinates are computable from finitely
many Brownian samples, each finite chaoslet expansion can be implemented, up to arbitrary accuracy, by our $\operatorname{NeuralChaos}$ architecture. Remarkably, our approximation theorem also has a quantitative form. If we denote by $\mathbb{D}_{T}^{s,2:d}$ the Malliavin--Sobolev space of order $s$ with $L^{2}$-integrability for processes in $\mathcal{H}^2_{T}(\mathbb{R}^d)$, Theorem~\ref{thrm:happytimes} shows that our model realizes finite chaoslet approximations and inherits the corresponding best $N$-term approximation rate, thus implying both universality and optimal approximability for sparse and smooth stochastic processes, namely, processes belonging to a Malliavin--Sobolev space and exhibiting some compressibility (cf.~Definition \ref{defn:compressibility}). To prove the previous result, we import ideas from random matrix theory (e.g.,~\cite{candes2006robust,candes2006stable,candes2006near}), non-linear (optimal) constructive approximation (e.g, ~\cite{cohen2009compressed,devore1998nonlinear}), and compressive sensing (e.g., ~\cite{iwen2010combinatorial}) into mathematical finance. Also, we demonstrate, in Proposition~\ref{prop:genericity_of_compressibility}, that our compressibility assumption is probabilistically generic: almost surely, a randomly selected predictable process in $\mathcal{H}_T^2(\mathbb{R}^{d})$ will be compressible, as long as the selection is not made in a pathological way.

Our second (although negative) main result is a comparison with finite-dimensional Markovian neural-SDE-type parametrizations and \textit{bounded} latent state dimension; i.e.\ not signature-augmented neural SDEs e.g.~\cite{kidger2021neural,issa2023non,doi:10.1137/22M1512338}. In particular, we consider generalized predictable Euler--Maruyama-type processes with non-exploding state dimension. Such processes are natural finite-dimensional surrogates for Markovian SDE models and are sufficient to approximate classical Markovian SDEs under standard Lipschitz assumptions. However, Proposition~\ref{prop:badnews} shows that this class
is meagre in $\mathcal{H}_T^2(\mathbb{R}^{d})$ and has measure zero under every centred non-singular Gaussian law. This does not imply that such SDEs are ineffective as interpretable models of stochastic phenomena. Rather, it suggests that most stochastic phenomena are unlikely to admit a representation of this form, in contrast with compressible processes, which are typical.

Finally, we illustrate the relevance of our framework in a simple stochastic optimal control problem as well as an example of mathematical finance, namely dynamic hedging. In both cases, NeuralChaos provides a class of trainable, predictable, square-integrable strategies that serve as control processes and trading strategies, respectively.
These examples are not meant to exhaust the possible applications, but they show how a finite-sampling, causal, trainable parametrization of $\mathcal{H}_T^2(\mathbb{R}^{d})$ can be used in financial optimization problems where the unknown object is itself a predictable process.

\section{Preliminaries}\label{sec::preliminaries}
This section recalls some preparatory material for the derivations of the main results of this paper.  
\subsection{Notation}\label{subsec::notation}
For the sake of the reader, we collect and define the notations and conventions we will use in the rest of the paper. 

\indent{\emph{Sets and spaces.}}\, $\mathbb{R}_{+} \subseteq \mathbb{R}$ denotes the subset of real numbers that are strictly greater than zero. Similarly, $\mathbb{N}_{+} \subseteq \mathbb{N}$ denotes the subset of natural numbers that are strictly greater than zero. For a generic $N \in \mathbb{N}_{+}$, we define $[N]_{+}\eqdef\{1,\ldots,N\}$ and $[N] \eqdef\{0,\ldots,N\}$. Moreover, for a countable index set $I$ and $0<p<\infty$, we write $\ell^p(I)\eqdef\big\{a=(a_i)_{i\in I}:\| a \|_{\ell^p}^p \eqdef \sum_{i\in I}|a_i|^p<\infty\big\}$. When $I=\mathbb N_+$, we simply write $\ell^p$. Note that $\|\cdot\|_{\ell^p}$ is for  a norm, whereas for $0<p<1$ it is a quasi-norm. For vector-valued families $a=(a_i)_{i\in I} \subseteq \mathbb{R}^d$,  we write $a\in\ell^p(I;\mathbb R^d)$ if $\Vert a \Vert_{\ell^p(I;\mathbb{R}^d)}^p \eqdef \sum_{i\in I}\vert a_i \vert_{\mathbb{R}^d}^p<\infty$. For a countable set $I$, we set $\ell_0(I;\mathbb N) \eqdef \big\{a=(a_i)_{i\in I} \subseteq \mathbb{N}:
\operatorname{supp}(a)\eqdef\{i\in I:a_i\neq0\}\text{ is finite}\big\}$.

\vspace{0.3cm}
\indent\emph{Linear algebra.}\, For a generic $d \in \mathbb{N}_{+}$, we denote by $\mathsf{1}_{d} \in \mathbb{R}^{d}$ the $d$-dimensional vector with entries all equal to one. For any generic $N \times M$ matrix $A$, we denote its $i^{th}$ row vector by $A_{i:}$ and its $j^{th}$ column vector by $A_{:j}$, where $i \in [N]_{+}$ and $j \in [M]_{+}$.

\vspace{0.3cm}
\indent \emph{Machine learning.}\, We denote by $\sigma(\cdot)$ any continuous threshold function $\sigma:\mathbb{R}\rightarrow[0,1]$ satisfying $\sigma(t)=0$ for all $t \leq 0$ and $\sigma(t)=1$ for all $t \geq 1$; e.g., $\sigma(t)\eqdef \max\{0,\min\{1,t\}\}$ or $\sigma(t)\eqdef0 \cdot \mathsf{1}_{\{t\leq 0\}}+(3 t^2-2 t^3)\cdot\mathsf{1}_{\{0<t<1\}}+1 \cdot\mathsf{1}_{\{t \geq 1\}}$. Moreover, we denote $\operatorname{ReLU}(t)\eqdef \max\{0,t\}$, for $t\in \mathbb{R}$.

\vspace{0.3cm}
\indent \emph{Stochastic Analysis.}\, Throughout the paper, we fix $T \in \mathbb{R}_{+}$ and a $D$-dimensional ($D \in \mathbb{N}_{+}$) Brownian motion on a complete probability space $(\Omega, \mathcal{F}, \mathbb{P})$. We denote by $\mathcal{F}_{\cdot}=(\mathcal{F}_{t})_{0 \leq t\leq T}$ the usual $\mathbb{P}$-augmentation of the Brownian filtration; unless otherwise stated $\mathcal{F}_{0}$ is the $\mathbb{P}$-trivial filtration. For $d \in \mathbb{N}$, we denote by $\mathcal{H}^{2}_{T}(\mathbb{R}^d)\eqdef L^2_{\rm pred}(\Omega \times [0,T], \mathcal{P}, \mathbb{P} \otimes \ud t;\mathbb{R}^d)$ the space of $\mathbb{R}^d$-valued predictable processes $u_{\cdot}$ with $\|u_{\cdot}\|^2_{\mathcal{H}^{2}_{T}(\mathbb{R}^d)}\eqdef \mathbb{E}\left[\int_{0}^{T} \vert u_t \vert_{\mathbb{R}^d}^2\,\ud t\right] < \infty$, where $\mathcal{P}$ denotes the predictable $\sigma$-field on $\Omega \times [0,T]$. Sometimes, we also use the norm $\|u_{\cdot}\|^2_{\mathcal{S}^{2}_{T}(\mathbb{R}^d)}\eqdef \mathbb{E}\left[\sup_{0 \leq t \leq T} \vert u_t \vert_{\mathbb{R}^d}^2 \right]$. Moreover, for $0 \leq t\leq T$, we denote by $L^{2}(\mathcal{F}_{t};\mathbb{R}^d)\eqdef L^2(\Omega,\mathcal{F}_t,\mathbb{P};\mathbb{R}^d)$ the space of $\mathbb{R}^d$-valued $\mathcal{F}_t$-measurable random variables with finite second moment. In addition, for $\Delta_n(t)\eqdef\{(s_1,\ldots,s_n) \in [0,t]^{n}\,:\,0<s_1<\ldots<s_n<t\}$ and $D \in \mathbb{N}_+$, we define $\mathcal{K}_n(t;D)\eqdef L^2\!\left(\Delta_n(t);\mathbb R^{D^n}\right)$. Furthermore, for $0 \leq s < \infty$, we denote by $\mathbb{D}^{s,2:d} \subseteq L^2(\Omega,\mathcal{F}_T,\mathbb{P};\mathbb{R}^d)$ the Malliavin--Sobolev space of order $s$ (see \eqref{eq:Malliavin_Sobolev_norm}), and by $\mathbb{D}_{T}^{s,2:d}$ the analogous space for $\mathbb{R}^{d}$-valued processes in $\mathcal{H}^2_{T}(\mathbb{R}^d)$. Finally, for $d=1$, we abbreviate $\mathcal{H}^{2}_{T} \eqdef \mathcal{H}^{2}_{T}(\mathbb{R})$, $L^2(\Omega,\mathcal{F}_t,\mathbb{P}) \eqdef L^2(\Omega,\mathcal{F}_t,\mathbb{P};\mathbb{R})$, $\mathbb{D}^{s,2} \eqdef \mathbb{D}^{s,2:1}$, and $\mathbb{D}_{T}^{s,2} \eqdef \mathbb{D}_{T}^{s,2:1}$.

\subsection{Patterns in the Wiener Chaos: smoothness and compressibility.}
We first work in the scalar-valued case; vector-valued processes are treated component-wise.

Multivariate ordinary polynomials on $\mathbb{R}^{m}$ are among the basic building blocks of functions on Euclidean space. For example, on a compact domain, such as the $m$-dimensional unit hypercube, every square-integrable function can be approximated by polynomials of sufficiently large degree. An analogous role over $L^2(\mathcal{F}_T)$, where $\mathcal{F}_{\cdot}=(\mathcal{F}_{t})_{0 \leq t\leq T}$ is the usual augmentation of the Brownian filtration (unless otherwise stated $\mathcal{F}_{0}$ is the $\mathbb{P}$-trivial filtration), is played by the Wiener chaos, namely an orthogonal graded decomposition of $L^2(\mathcal{F}_T)$ into ``stochastic monomials'' of increasing degree. In this correspondence, degree-$n$ monomials are replaced by iterated It\^{o} integrals of order $n \in \mathbb{N}_{+}$. More precisely, for $\ell_1,\ldots,\ell_n \in [D]_{+}$ and a family $f_n=(f_n^{\ell_1,\ldots,\ell_n})_{(\ell_1,\ldots,\ell_n)\in [D]^{n}} \in \mathcal{K}_n(t; D)$ (cf. Subsection \ref{subsec::notation}), we denote by $I_n^{t}:\mathcal{K}_n(t; D) \to L^2(\mathcal{F}_t)$ the operator returning the iterated It\^o integral of order $n$, i.e.,
\begin{equation}
\label{eq:iterated_ito_term}
I_n^{t}(f_n) \eqdef \sum_{\ell_1,\ldots,\ell_n=1}^D
\int_0^t\int_0^{s_n}\cdots\int_0^{s_2}
f_n^{\ell_1,\ldots,\ell_n}(s_1,\ldots,s_n)
\,\ud W_{s_1}^{\ell_1}\cdots \ud W_{s_n}^{\ell_n};
\end{equation}
we use the chronological convention, in the sense that the variable $s_j$ is paired with the Brownian component $W^{\ell_j}$; whence, the kernel is evaluated as $f_n^{\ell_1,\ldots,\ell_n}(s_1,\ldots,s_n)$ on $\Delta_n(t)$. Therefore,  just as ordinary polynomials are graded by degree, Wiener chaos is graded by the number of stochastic integrations: constants form degree 0, single stochastic integrals form degree 1, double iterated integrals form degree 2, and so on. As in the polynomial setting, higher degrees correspond to greater complexity and allow one to encode progressively finer features of the underlying object. 

The above analogy can be pushed one step further.  In $\mathbb{R}^m$, passing from degree $k$ to degree $k+1$ amounts to multiplying by one additional coordinate function $x_i$, for some $i\in [m]_+$.  In the Wiener chaos expansion, the corresponding operation is one additional It\^{o}-integration against one component $W^{\ell}_{\cdot}$, for some $\ell \in [D]_{+}$, of the driving Brownian motion. In this sense, It\^{o} integration plays the role of multiplication by a coordinate, and iterated stochastic integrals play the role of higher-order monomials.

However, unlike polynomials on $\mathbb{R}^{m}$, where each homogeneous degree contains only finitely many monomials, each non-zero level of the Wiener chaos contains infinitely many orthogonal ``monomials'', indexed by the infinite-dimensional kernel space $\mathcal{K}_n(T;D)$. Thus, the Wiener chaos is \textit{infinite} in \textit{two distinct directions}: first, as for ordinary polynomials, the degree grading is unbounded, since $n$ ranges over all orders of stochastic integration; second, unlike for ordinary polynomials on $\mathbb{R}^{m}$, the space spanned by the ``monomials'' of each fixed positive degree is itself infinite-dimensional. As we will see shortly, the first source of infinitude is largely controlled by Malliavin smoothness, while the second is controlled by compressibility.

Much as polynomials yield an asymptotic expansion of functions on Euclidean domains, iterated It\^{o} integrals yield an asymptotic expansion of square-integrable random variables measurable with respect to Brownian motion. Indeed, since the $\sigma$-algebra $\mathcal{F}_t$ is generated by $\{W_s\}_{0\le s\le t}$, we obtain for every $0\le t\le T$ and $u\in L^2(\mathcal{F}_t)$ the Wiener chaos expansion 
\begin{equation}
\label{eq:WienerChaos}
        u 
    = 
        \mathbb{E}[u] 
        + 
        \sum_{n=1}^{\infty}\,
        I_n^{t}(f_n),
\end{equation}
for some $f_n \in \mathcal{K}(t;D)$, $n \in \mathbb{N}_+$, where $I_n^{t}(f_n)$ is defined in Eq.~\eqref{eq:iterated_ito_term}.

The trouble with the Wiener chaos expansion in Eq.~\eqref{eq:WienerChaos}, and with similar objects such as path signatures, is not that higher-order objects are necessarily complicated. Instead, the main difficulty is that a high-order sparse structure is expensive to find if one first exhibits a large linear dictionary, and only afterwards performs model selection 
which is computationally costly, cf.~\cite{guo2025consistency} for the signature version of this idea using the LASSO of~\cite{tibshirani1996regression}; that is, if the Wiener chaos is used non-adaptively.  
For the sake of clarity, we illustrate this point via a simple example of a ``higher-order but sparse perturbation of a single Brownian motion path'', as represented in the following Figure~\ref{fig:chaos_so_much_chaos}. 

\begin{figure}[h!]
    \centering
    \includegraphics[width=0.85\linewidth]{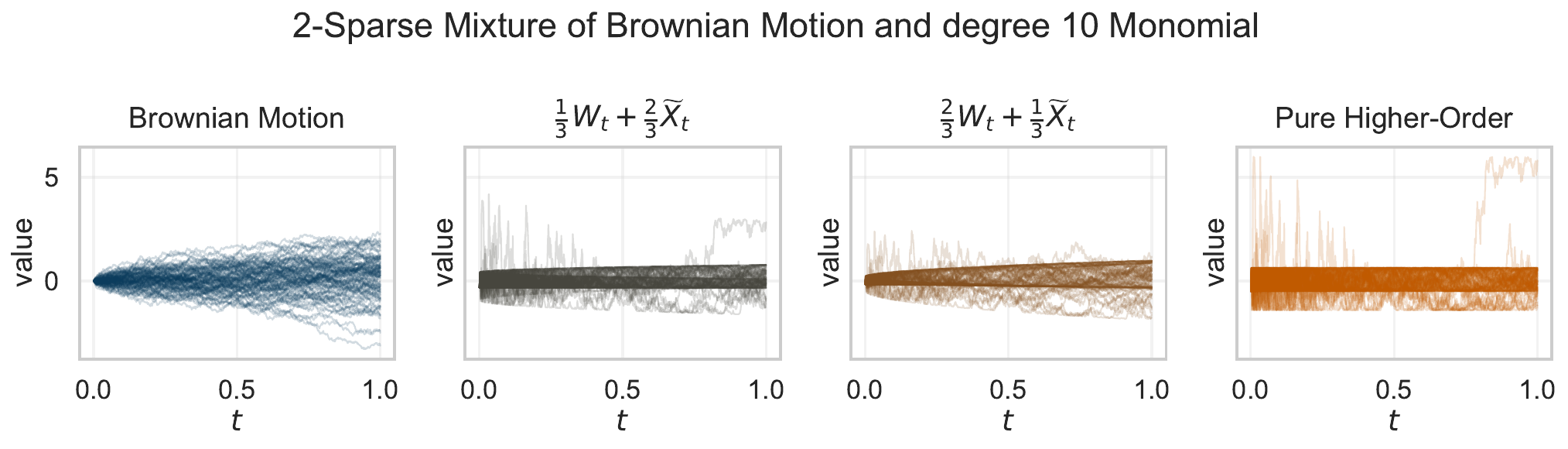}
    \caption{
    \textbf{A process which is easy for \textit{adaptive} approximation methods to approximate but hard for non-adaptive methods.}
    Evidently, at most two parameters are needed to exactly represent any of the illustrated processes: one parameter for the Brownian factor $W_{\cdot}$ and one parameter for the higher-order chaos factor $\tilde{X}*{\cdot}$; equivalently, each process is a two-term degree-$10$ polynomial in the Wiener-chaos sense.
    \hfill\\
    A \textit{non-adaptive} method first constructs a large graded dictionary of chaos monomials until the relevant factor $\tilde{X}*{\cdot}$ is included, and only then sub-selects the relevant factors. Its computational cost is bottlenecked by this first phase, which requires constructing sufficiently many factors before sub-selection can begin. 
    \hfill\\
    In contrast, an \textit{adaptive} method directly \textit{searches} for the important factors, without wasting computation on constructing the irrelevant ones. Its computational benefit arises precisely from avoiding the construction of a large dictionary of candidate factors.
    }
    \label{fig:chaos_so_much_chaos}
\end{figure}

Let $W_{\cdot}=(W_t)_{0 \leq t \leq T}$ be a one-dimensional Brownian motion and $\mathrm{He}_{10}(\cdot)$ be the probabilists' Hermite polynomial of degree ten. Then, we set $h_{10}(x)\eqdef\mathrm{He}_{10}(x)/\sqrt{10!}$ and define for $t > 0$,
\begin{equation}\label{eq:hermite}
\widetilde{X}_t \eqdef \sqrt{t}
h_{10}\left(\frac{W_t}{\sqrt t}\right),
\end{equation}
where $\widetilde X_0\eqdef 0$. The quantity $\widetilde X_t$ belongs to the tenth Wiener chaos of $L^2(\mathcal F_t)$; one has $\mathbb E[\widetilde X_t]=0$, $\mathbb E[\widetilde X_t^2]=t$, and $\mathbb E[W_t\widetilde X_t]=0$. Therefore, for $0<\eta<1$, the two-term chaos polynomial
\begin{equation}\label{eq:lin_reg__badtarget}
\eta W_t+(1-\eta)\widetilde X_t = \eta W_t+(1-\eta) \sqrt t
h_{10}\left(\frac{W_t}{\sqrt t}\right)
\end{equation}
has mean zero and variance $\operatorname{Var}\big(\eta W_t+(1-\eta)\widetilde X_t\big)=\big(\eta^2+(1-\eta)^2\big)t$. From the computational side, the issue is not that the specific scalar
example above is difficult once the active coordinate is known. Indeed, after the degree-ten coordinate has been identified, it can be evaluated from $W_t/\sqrt t$ using the Hermite polynomial of order ten. The bottleneck appears when one tries to discover such sparse high-order structure by linear selection over a pre-generated chaos dictionary. If one retains $m_{\mathrm{feat}}$ scalar Gaussian coordinates, the number of homogeneous Hermite monomials of degree ten is of order $m_{\mathrm{feat}}^{10}$. In particular, if these coordinates arise from
$J$ temporal modes and a $D$-dimensional Brownian motion, then the degree-ten dictionary has size of order $(JD)^{10}$. Hence, in an ordered iterated-integral representation, one obtains the same type of combinatorial growth in the number of degree-ten words. Thus, sparse regression methods such as LASSO, group LASSO, or elastic-net
regularization (see, e.g., \cite{guo2025consistency}) must first generate or scan a large high-order dictionary before they can select the few active chaos coordinates. In contrast, the active structure in Eq.~\eqref{eq:lin_reg__badtarget} has constant description length:
two active chaos coordinates, of degree one and ten, together with two coefficients. A nonlinear parametrization should, therefore, aim to synthesize the active high-order coordinate directly, with a cost depending on the desired approximation accuracy rather than on the size of the full degree-ten dictionary and thus take advantage of sparse non-linear approximation ideas of~\cite{devore1998nonlinear}. 

This paper provides such an approximation guarantee using a \emph{causal} deep-learning model, whose parameters are trainable and can ``cycle through'' different grades of the Wiener chaos without having to generate the entire graded structure before sub-selecting. In particular, in the next subsection, we will consider an alternative description of the Wiener chaos expansion of any random variable $u\in L^2(\mathcal{F}_t)$, for a given $0\le t \le T$, which extends more easily to multiple dimensions and does not require lengthy computations of multiple iterated stochastic integrals. The idea hinges on the replacement of deterministic kernels and iterated stochastic integrals by Hermite polynomials evaluated at finitely many Gaussian coordinates of the Brownian path. In particular, these Gaussian coordinates are chosen to be Haar coordinates, so each active coordinate can be computed from finitely many Brownian samples.

\subsubsection[%
A ``special'' orthonormal basis for the Wiener chaos: chaoslets -- sampling not integration
]{A ``special'' orthonormal basis for the Wiener chaos: chaoslets -- sampling\protect\footnote{Here, sampling is meant in the de-facto sense of compressed sensing~\cite{donoho2006compressed,candes2006robust,iwen2010combinatorial}, approximation theory~\cite{krieg2021functionvalues,siegel2025nearly,krieg2026approximation}, and harmonic analysis~\cite{landau1967necessary,fuhr2017density,BARANOV20151358}; that is, point evaluations at a given point in time $t \in [0,T]$, not sampling in the sense of statistics, i.e.\ i.i.d.\ copies of a process or drawing some $\omega \in \Omega$.} not integration}
\label{subsec::special_chaos}

For every $i\in \mathbb{N}$, we first consider the probabilists' Hermite polynomial $\mathrm{He}_{i}(x)=(-1)^{i} e^{x^2/2}\frac{\ud^i}{\ud x^{i}}e^{-x^2/2}$, and then the corresponding normalized Hermite polynomials $h_i(x) \eqdef \frac{\mathrm{He}_{i}(x)}{\sqrt{i!}}$, with $h_0(x)=1$, which are the eigenfunctions of the generator $\frac{\ud^2}{\ud x^2}-x\frac{\ud}{\ud x}$ of the Ornstein-Uhlenbeck process $\ud X_t = -X_t\,\ud t +\sqrt{2}\,\ud W_t$. For a cylindrical, finite-rank element of the $n^{th}$ homogeneous Wiener chaos, or for a finite-dimensional truncation of a general chaos element, one has the following representation (general elements are obtained as $L^2$-limits of such cylindrical truncations):
\begin{equation}
\label{eq:ith_WienerChaos}
    \sum_{|\boldsymbol\alpha|=n}
    \,
        \beta_{n,\boldsymbol\alpha}
        \,
            \prod_{r=1}^{D}
            \prod_{j=1}^{J_n}
            \,
                h_{\alpha_{j,r}}\Big(
                    \int_0^t\,
                        \psi_{n,j}(s)
                    \,
                    \ud W_s^r
                \Big)
\end{equation}
where $\boldsymbol{\alpha} \eqdef (\alpha_{j,r})_{j\in [J_n]_{+},\,r\in [D]_{+}}\in \mathbb{N}^{J_n\times D}$ is a multi-index with $|\boldsymbol{\alpha}| \eqdef \sum_{r=1}^{D}\sum_{j=1}^{J_n}\, \alpha_{j,r} = n$, and $\beta_{n,\boldsymbol\alpha}\in \mathbb{R}$ are some coefficients. In the previous formula $J_n$ denotes the number of (retained) Gaussian coordinates that are employed to represent the $n^{th}$ homogeneous chaos. In particular, the index $j$ is only an enumeration index for the finitely many deterministic coordinates $\psi_{n,j}$ retained at level $n$; below these coordinates will be chosen from a Haar system.  More precisely, we consider the \textit{Haar (wavelet) system} given for $0 \leq t \leq T$ by
\begin{equation}
    \psi_{i,k}(t) \eqdef 
 \tfrac{2^{i/2}}{\sqrt{T}}\left(
        -\,\mathbf{1}_{[T\frac{k}{2^i},T\frac{1+2k}{2^{i+1}})}(t)
        +
            \mathbf{1}_{[T\frac{1+2k}{2^{i+1}},T\frac{k+1}{2^i})}(t)
    \right).
\end{equation}
The main motivation for this choice is that the stochastic integrals of the previous quantity (cf.~Eq.~\eqref{eq:ith_WienerChaos}) can, not only be computed in closed form, but also, as remarked in~\cite[Equation 52]{alvarez2025neural}, transform computationally heavy stochastic integral computations into lightning-fast \textit{sampling operations}; whereby a process is simply \textit{sampled} at \textit{three Brownian samples} per basic wavelet. In particular, for every $r \in [D]_{+}$, we define the Gaussian Haar coordinate as  
\begin{equation}
\label{eq:stochastic_haar_coordinate}
Z_{i,k}^r\eqdef\tfrac{2^{i/2}}{\sqrt{T}}\,
    W^r_{\frac{Tk}{2^i}}
    -\tfrac{2^{i/2+1}}{\sqrt{T}}\,W^r_{\frac{T(1+2k)}{2^{i+1}}}
    +\tfrac{2^{i/2}}{\sqrt{T}}\,W^r_{\frac{T(k+1)}{2^i}}.
\end{equation}
We notice that the notation in Eq.\eqref{eq:ith_WienerChaos} is linked to the Haar (wavelet) system in the following way. For every $n \in \mathbb{N}_{+}$ one first chooses $\Lambda_n
=\{(i_1,k_1),\ldots,(i_{J_n},k_{J_n})\}$, and then set $\psi_{n,j}\eqdef\psi_{i_j,k_j}$ and $Z_{j,r} \eqdef Z^{r}_{i_j,k_j}$, $j \in [J_n]_{+}$ and $r \in [D]_{+}$. Whence, $\int_0^t\psi_{n,j}(s)\,\ud W_s^r=Z_{j,r}$, and Eq.\eqref{eq:ith_WienerChaos} can be written as a tensorized Hermite polynomial evaluated at the finite Gaussian vector $\mathbf{Z}=(Z_{j,r})_{j \in [J_n]_{+}, r \in [D]_{+}} \in \mathbb{R}^{J_n\times D}$. We now pass from random variables to predictable square-integrable processes $u_{\cdot} \in \mathcal{H}^2_T(\mathbb{R}^{d})$. First, we observe that when $p \in \mathbb{N}$ and $0 \leq q \leq 2^{p-1}$, $\psi_{p,q}(\cdot)$ is supported on $[\frac{qT}{2^p},\frac{(q+1)T}{2^p})$. To preserve adaptedness, we, therefore, define $\mathcal{I}_{p,q}\eqdef
    \Big\{
        (i,k)\in \mathbb{N}^2:
        \,
        0\le k<2^i
        \,\mbox{ and }\,
        \frac{(k+1)}{2^i}\le \frac{q}{2^{p}}
    \Big\}
$. At this point, we let $\boldsymbol\alpha= (\alpha^r_{i,k})_{(i,k)\in \mathcal \mathcal{I}_{p,q},\ r\in[D]_{+}}$ be a finitely supported family with values in $\mathbb{N}$, and define $|\boldsymbol\alpha|\eqdef \sum_{(i,k)\in \mathcal{I}_{p,q}} \sum_{r=1}^{D} \alpha^r_{i,k}$. Then, the stochastic Hermite--Haar factor associated with $(p,q,\boldsymbol\alpha)$ is $u^{(T)}_{\boldsymbol\alpha;p,q}(\omega) \eqdef \prod_{(i,k)\in \mathcal{I}_{p,q}} \prod_{r=1}^{D} h_{\alpha^r_{i,k}} \left(Z^r_{i,k}(\omega)\right) \in L^2(\mathcal{F}_\frac{q T}{2^p})$; in particular, the product is finite because $\boldsymbol\alpha$ is finitely supported, and, importantly, the stochastic integrals (slow) have been reduced to ordered time-samples (fast). The corresponding \textit{chaoslet} is 
\begin{equation}
\label{eq:reshuffled}
    \mathfrak{C}_{(p,q)}^{\boldsymbol{\alpha}}
    (t,\omega)
\eqdef
    \,
    \underbrace{
        \psi_{p,q}(t)
    }_{\text{Time-Haar wavelet}}
    \,
    \underbrace{
        \prod_{(i,k)\in \mathcal{I}_{p,q}}
        \,
        \prod_{r=1}^{D}
        \,
        h_{\alpha_r^{i,k}}\big(
            Z_{i,k}^r(\omega)
        \big)
    }_{\text{Stochastic Hermite-Haar factor:\,\,} u^{(T)}_{\boldsymbol\alpha;p,q}(\omega)}
,
\end{equation}
We emphasize that the indices $(i,k)$ in the stochastic tensorial term $\prod_{(i,k)\in \mathcal{I}_{p,q}}
        \,
        \prod_{r=1}^{D}
        \,
        h_{\alpha_r^{i,k}}\big(
            Z_{i,k}^r
        \big)
$ are tied to the second endpoint, namely $q$, of the temporal wavelet $\psi_{p,q}$. We notice that \cite[Lemma~19]{alvarez2025neural} obtains the following hybrid polynomial/wavelet basis for the space $\mathcal{H}^2_T(\mathbb{R}^{d})$, whereby every process $u_{\cdot} \in \mathcal{H}^2_T(\mathbb{R}^{d})$ admits an asymptotic orthogonal expansion of the form
\begin{equation*}
    u_t(\omega) = c + \sum_{p=0}^{\infty}\sum_{q=0}^{2^p-1}\sum_{\boldsymbol\alpha\in \ell_0(\mathcal{I}_{p,q}\times[D];\mathbb{N})}\beta_{p,q,\boldsymbol\alpha}\mathfrak{C}_{(p,q)}^{\boldsymbol{\alpha}}
    (t,\omega) 
\end{equation*}

In this way, the pair $(p,q)$ determines the deterministic time support of the chaoslet,
whereas $|\boldsymbol{\alpha}|$ is its Wiener-chaos degree. The Wiener-chaos level of $\mathfrak{C}_{(p,q)}^{\boldsymbol{\alpha}}$ is precisely $|\boldsymbol{\alpha}|$.
As shown in~\cite[Lemma 19]{alvarez2025neural}, the family 
\begin{equation}
\label{eq:chaosless_disorganized_ungraded}
\mathfrak{C}^{\rm all}  \eqdef \Big\{
                \tfrac{1}{\sqrt{T}}
        \Big\}
    \cup
        \Big\{
            \mathfrak{C}_{(p,q)}^{\boldsymbol{\alpha}}
            :
            \,
            p,q\in \mathbb{N},
            \,
            q \le 2^{p}-1,
            \,
            \boldsymbol{\alpha} \in \ell_0(\mathcal{I}_{p,q}\times[D]_{+};\mathbb N)
        \Big\}
\end{equation}
is an orthonormal basis of $\mathcal{H}_T^2$, where $1/\sqrt{T}$ denotes the constant process $(t,\omega)\mapsto 1/\sqrt{T}$. 

We review, in the next subsection, the classical notion of regularity in stochastic analysis using our chaoslets, namely Sobolev--Malliavin differentiability.  Our main tool will be precisely this chaoslet expansion, which is a very specific form of a Wiener chaos expansion.

\subsubsection{The Malliavin--Sobolev Space}\label{subsubsec::MalliavinSobolevSpace}
First, we recall the Malliavin--Sobolev space for random variables. Then, we define the corresponding process-level space used in the current paper. More precisely, for any $\mathbb{R}$-valued random variable $X \in L^2(\mathcal{F}_T)$, we abbreviate its Wiener chaos decomposition by $X=\sum_{n=0}^{\infty}\bar{J}_n X$, where $\bar{J}_n$ denotes the orthogonal projection onto the $n^{th}$ homogeneous Wiener chaos. Then, we define, on its natural domain, the Ornstein-Uhlenbeck generator $L$ by
\begin{equation*}
    L X\eqdef -\sum_{n=0}^\infty n \bar{J}_n X.
\end{equation*}
For $s \in \mathbb{R}_+$, the Malliavin--Sobolev space $\mathbb{D}^{s,2}$ is the completion of smooth cylindrical functionals with respect to the norm
\begin{equation}
\label{eq:Malliavin_Sobolev_norm}
    \|X\|_{\mathbb{D}^{s,2}}
\eqdef
    \|(I-L)^{s/2}X\|_{L^2(\mathcal{F}_T)}
=
    \left(
        \sum_{n=0}^\infty (1+n)^s\|\bar{J}_n X\|_{L^2(\mathcal{F}_T)}^2
    \right)^{1/2}.
\end{equation}
For $s \in \mathbb{N}$, $\|X\|_{\mathbb{D}^{s,2}}$ is equivalent to the norm $\Vert X \Vert_{s,2} \eqdef \big( \|X\|_{L^2(\mathcal{F}_T)}^2 + \sum_{j=1}^s \|D^j X\|_{L^2(\mathcal{F}_T)}^2 \big)^{1/2}$ introduced in \cite[Equation~1.32]{nualart06} as $\|D^j X\|_{L^2(\mathcal{F}_T)}^2 = \sum_{n=j}^\infty n (n-1) \cdots (n-j+1) \|\bar{J}_n X\|_{L^2(\mathcal{F}_T)}^2$. For an $\mathbb{R}^{d}$-valued random variable, the Malliavin--Sobolev space is denoted by $\mathbb{D}^{s,2:d}$, with norm $\|X\|^2_{\mathbb{D}^{s,2:d}}\eqdef\sum_{j=1}^{d}\|X^{j}\|_{\mathbb{D}^{s,2}}^2$ (cf.~Eq.~\eqref{eq:Malliavin_Sobolev_norm}).

We now define the corresponding space for square-integrable predictable processes. We let $\mathfrak{C} = \bigcup_{k=0}^{\infty}\,\mathfrak{C}_k$ be the adapted chaoslet orthonormal basis of $\mathcal{H}^2_T$ graded by Wiener-chaos degree; in particular, $\mathfrak{C}_k$ consists of the chaoslets whose stochastic Hermite degree is $k$. For $s \geq 0$, we define
\begin{equation*}
    \mathbb{D}_T^{s,2}\eqdef\left\{X \in \mathcal{H}^2_T\,:\, \|X\|_{\mathbb{D}_T^{s,2}}^2 \eqdef \sum_{n=0}^{\infty}(1+n)^s\|\bar{J}_n^{\mathcal{H}} X\|^2_{\mathcal{H}^2_T}<+\infty\right\}
\end{equation*}
where $\bar{J}_n^{\mathcal{H}} X$ denotes the orthogonal projection in $\mathcal{H}^2_T(\mathbb{R}^{d})$ onto the closed linear span of $\mathfrak{C}_n$. Analogously to random variables, we define $\mathbb{D}_T^{s,2:d}$ for square-integrable predictable processes component-wise. We are now ready to give the following definition.

\subsection{Processes with Few Relevant Monomial Terms - Compressible Processes}
\label{s:Compressibility}
We now introduce a second notion of regularity for stochastic processes, distinct from Malliavin smoothness. These two notions should be understood as controlling the two ``directions of infinite-dimensionality'' of the Wiener chaos; cf.\ the discussion above~\eqref{eq:WienerChaos}. Malliavin smoothness ensures that only few levels of the Wiener chaos are needed, i.e. that only low-degree monomials need to be combined for approximation. Compressibility ensures that, within each level, only few monomials are relevant, despite the fact that the Wiener chaos contains infinitely many orthogonal monomials per level. When combined, these two properties imply that the resulting process is well approximated by few low-degree monomials. Such a structure is conducive to optimal nonlinear approximation guarantees and, as we will show, can be captured by our proposed \textit{NeuralChaos} model using few factors.
\hfill\\
Our notion is a stochastic-process analogue of \textit{compressible vectors} in compressed sensing~\cite{candes2006robust,candes2006stable,candes2006near}.

\begin{figure}[htbp!]
    \centering
    \begin{subfigure}[t]{0.28\linewidth}
        \centering
        \includegraphics[width=.85\linewidth]{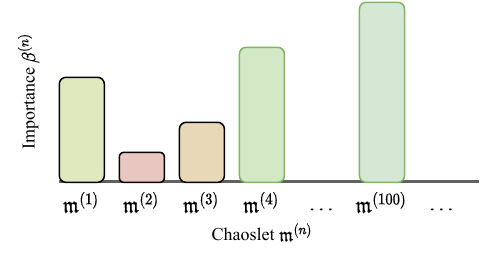}
        \caption{Chaoslet expansion of a \textbf{compressible} process, ordered according to the default ordering in $\mathfrak{C}^{\mathrm{all}}$.}
        \label{fig:Compressibility__Unsorted}
    \end{subfigure}
    \hfill
    \begin{subfigure}[t]{0.28\linewidth}
        \centering
        \includegraphics[width=.85\linewidth]{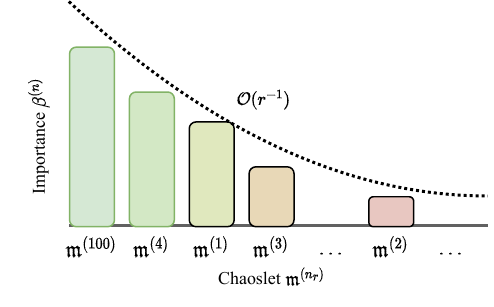}
        \caption{The same chaoslets $\mathfrak{m}^{(n)}$ sorted by \textit{decreasing} coefficient magnitude $|\beta^{(n)}|$.}
        \label{fig:Compressibility__Sorted}
    \end{subfigure}
    \hfill
    \begin{subfigure}[t]{0.28\linewidth}
        \centering
        \includegraphics[width=.85\linewidth]{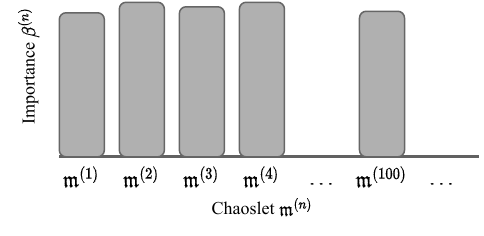}
        \caption{For an \textbf{incompressible} process, sorting the chaoslets $\mathfrak{m}^{(n)}$ by decreasing coefficient magnitude $|\beta^{(n)}|$ does not reveal advantageous decay.}
        \label{fig:Compressibility__Incompressible}
    \end{subfigure}
    \caption{\textbf{Compressibility vs. Incompressibility.} A compressible process is one whose chaoslet expansion, cf.~\eqref{eq:compressibility__expansion}, can be rearranged to exhibit at least power-law coefficient decay, cf.~\eqref{eq:compressibility}. The left subfigure illustrates the expansion in the default chaoslet ordering of $\mathfrak{C}^{\mathrm{all}}$, cf.~\eqref{eq:chaosless_disorganized_ungraded}, while the middle subfigure illustrates the same expansion after sorting the chaoslets by coefficient magnitude; i.e.\ in the spirit of greedy bases~\cite{albiac2021quasi}.
    \hfill\\
    In contrast, for an \textit{incompressible} process, illustrated in the right subfigure, no permutation of the chaoslets yields an advantageously decaying coefficient sequence. Thus, there is no gain from performing a best $N$-term approximation; i.e.\ a basis truncation after reordering, since there is no advantageous permutation.}
    \label{fig:Compressibility}
\end{figure}

\begin{definition}[Compressible predictable process]\label{defn:compressibility}
    For $S>0$, let $X \in \mathcal{H}^2_T(\mathbb{R}^{d})$ be a process with chaoslet expansion
    \begin{equation}
    \label{eq:compressibility__expansion}
        X= \sum_{k=0}^{\infty}\,\sum_{\mathfrak{m}\in \mathfrak{C}_k}\,\beta_{k,\mathfrak{m}}\,\mathfrak{m},
    \end{equation}
    where $(\beta_{k,\mathfrak{m}})_{{k,\mathfrak{m}}} \subseteq \mathbb{R}^d$. Then, $X$ is called \emph{$S$-compressible} if the non-increasing rearrangement of the norms of $(\beta_{k,\mathfrak{m}})_{{k,\mathfrak{m}}}$ decays polynomially with exponent $S$. More precisely, if $|\beta_{k_1,\mathfrak{m}_1}|_{\mathbb{R}^{d}}\ge |\beta_{k_2,\mathfrak{m}_2}|_{\mathbb{R}^{d}} \ge \dots \ge |\beta_{k_r,\mathfrak{m}_r}|_{\mathbb{R}^{d}}
    \ge |\beta_{k_{r+1},\mathfrak{m}_{r+1}}|_{\mathbb{R}^{d}}\ge \dots$ for some enumerations $(k_r)_{r \in \mathbb{N}_+}$ and $(\mathfrak{m}_r)_{r \in \mathbb{N}_+}$ of $\mathbb{N}_+$ and $\mathfrak{C}_k$, respectively, there exists some constant $C_{X}>0$ such that for every $r \in \mathbb{N}_{+}$ it holds that
\begin{equation} 
\label{eq:compressibility}
|\beta_{k_r,\mathfrak{m}_r}|_{\mathbb{R}^{d}} \le C_{X}\,r^{-S}.
\end{equation}
\end{definition}

The previous definition quantifies sparsity of the chaoslet coefficient sequence after all chaos levels have been pooled together. In particular, the previous condition is a non-linear approximation condition, in the sense that if one retains only the $N$ largest chaoslet coefficients, then an $S$-compressible process admits a best $N$-term approximation error of order $N^{-(S-1/2)}$, in the standard sense of non-linear approximation theory (e.g.,~\cite{devore1998nonlinear}). However, this condition alone does not control in which Wiener
chaos levels the largest coefficients occur. The additional Malliavin--Sobolev assumption $X \in \mathbb{D}_T^{s,2:d}$ controls the tail in the chaos degree; this fact is formalized in the next proposition.
\begin{proposition}[Compressibility plus Malliavin regularity]
\label{prop:sparse_and_smooth}
Let $S>\tfrac{1}{2}$ and $s>0$. Assume that $X \in \mathcal{H}_T^2(\mathbb{R}^{d})$ is compressible with compressibility constant $C_{X}$; see Definition \ref{defn:compressibility}. Suppose, in addition, that $X \in \mathbb{D}_T^{s,2:d}$. Then, for every $N \in \mathbb{N}_{+}$ and every $P \in \mathbb{N}$, we have
\begin{equation}
\label{eq:best_N_term_rate}
    \inf_{
        \underset{
            \mathfrak{m}^{(1)},\dots,\mathfrak{m}^{(N)}\in \bigcup_{m=0}^{P} \mathfrak{C}_m
        }{
            \beta_1,\ldots,\beta_N\in \mathbb{R}^d
        }
    }
    \Big\|
        X
        -
        \sum_{n=1}^N
            \beta_n
            \mathfrak{m}^{(n)}
    \Big\|_{\mathcal{H}_T^2(\mathbb{R}^{d})}
    \lesssim_X
    N^{-(S-\frac12)}
    +
    (1+P)^{-s/2},
\end{equation}
where the symbol $``\lesssim_X"$ indicates that the bounding constant depends on $X$, but not on the approximation parameters $N$ and $P$.
\end{proposition}
\begin{proof}
See Appendix \ref{app::section_2}.
\end{proof}

While the process in Figure~\ref{fig:chaos_so_much_chaos} is visibly compressible, it is natural to ask which other processes are compressible as well. We provide a genericity guarantee, reminiscent of the corresponding phenomenon for expander graphs, i.e. families of sparse but highly connected graphs. In combinatorial probability, explicit expanders can be challenging to construct, yet simple random graph models often yield expanders with high probability, asymptotically as the number of vertices diverges; see, e.g.,~\cite{friedman2008proof}.
\hfill\\
Similarly, a sparse process chosen at random in a sufficiently non-degenerate fashion is almost surely compressible, as shown in Proposition~\ref{prop:genericity_of_compressibility}. Thus, compressibility should be viewed not only as a special hand-crafted structure, but also as a generic feature within natural sparse models. This should not be confused with saying that all familiar stochastic models are compressible: for instance, most classical SDE-generated processes need not have the sparse chaos structure exhibited by the processes in Figure~\ref{fig:chaos_so_much_chaos}.

We conclude this section with the following important remark.
\begin{remark}
Importantly, the two assumptions in Proposition \ref{prop:sparse_and_smooth} measure two different features of a predictable process. In particular, compressibility does not imply Malliavin--Sobolev smoothness. Indeed let us consider the one-dimensional (i.e., $d_X=1$), fix $S>\frac{1}{2}$, and for each $r \in \mathbb{N}_{+}$ choose a chaoslet $\mathfrak{m}^{(2^r)} \in \mathfrak{C}_{2^r}$ with $\|\mathfrak{m}^{(2^r)}\|_{\mathcal{H}^2_T}=1$. Define $X \eqdef \sum_{r=1}^{\infty} r^{-S}\mathfrak{m}^{(2r)}$. Then, since $\sum_{r=1}^{\infty} r^{-2 S}<\infty$ we have that $X \in \mathcal{H}^2_T$. In addition, $X$ is $S$-compressible since its non-zero chaoslet coefficients, arranged in decreasing order, are precisely $(r^{-S})_{r \geq 1}$. However, $X \notin \mathbb{D}^{1,2}_T$ since 
\begin{equation*}
    \|X\|^2_{\mathbb{D}^{1,2}_T}=\sum_{r=1}^{\infty}(1+2^r) r^{-2 S}=+\infty.
\end{equation*}
Conversely, Malliavin--Sobolev smoothness does not imply compressibility. To see this, fix $S>\frac{1}{2}$, and consider distinct chaoslets $(\mathfrak{m}_r^{(1)})_{r \in \mathbb{N}_+} \subseteq \mathfrak{C}_1$. Define $X \eqdef \sum_{r=1}^{\infty} r^{-S}\log(1+r)\mathfrak{m}_r^{(1)}$. Then, $\lim\limits_{r \rightarrow \infty}\, r^S|\beta_r|=\log(1+r) = + \infty$, so $X$ is not $S$-compressible. However, as $\sum_{r=1}^\infty r^{-2S}\log(1+r)^2<\infty$, we have $X \in \mathcal{H}_T^2$ and since $X$ lies entirely in the first Wiener chaos, we obtain $ X\in \mathbb{D}^{s,2}_T$, for all $s \geq 0$.
\end{remark}

In fact, we show that there are infinitely many smooth but incompressible processes, and conversely infinitely many compressible processes which are not smooth. Thus, smoothness and compressibility are distinct notions of regularity, controlling the two ``directions of infinite-dimensionality'' of the Wiener chaos: the former controls how many chaos levels are needed, while the latter controls how many monomials are relevant within each level. Together, they yield approximation by few low-degree monomials, the structure captured by our proposed $\operatorname{NeuralChaos}$ model using few factors. We now turn to the architecture of this neural network-based operator.

\section{The NeuralChaos architecture}\label{sec::section_3}
This section defines the $\operatorname{NeuralChaos}$ architecture, which, as said in the introduction, has three main parts: 1) a finite sampling of the driving Brownian path, 2) a causal triangular lifting of these samples, and 3) a continuous-time causal assembly through deterministic time masks. We fix $T>0$, a $D$-dimensional ($D \in \mathbb{N}_{+}$) Brownian motion $W_{\cdot}=(W_t)_{0 \leq t \leq T}$, input and output dimensions $d_{\mathrm{in}}\in\mathbb N_+$ and $d_{\mathrm{out}}\in\mathbb N_+$, a number of (time) grid intervals $M\in\mathbb N_+$, and a number of lifting channels $K\in\mathbb N_+$. 

We first define the discrete-time version of our model on a time-grid $\boldsymbol{t} \eqdef (t_m)_{m=0}^M$ where $0=t_0<t_1<\dots<t_M=T$, which we then extend to a predictable process with almost-surely continuous paths of controlled Lipschitz regularity. The stochastic input to the model is the (causally) ordered Brownian sample matrix $\mathbf{W}^{\boldsymbol{t}}=[W_{t_0}^{\top},W_{t_1}^{\top},\ldots,W_{t_{M-1}}^{\top}]^{\top} \in \mathbb R^{M\times D}$; whence, $d_{\mathrm{in}}=D$ (and, e.g.,  $d_{\rm out}=d$ when approximating $\mathcal{H}^2_T(\mathbb{R}^{d})$). Figure \ref{fig:masks_and_bm_paths} gives a schematic view of the two inputs used by $\operatorname{NeuralChaos}$: finite Brownian samples and the current evaluation time. The left panel illustrates the finite sampling of the driving Brownian path, while the right panel illustrates the causal time masks used to assemble the row-wise outputs into a predictable process. Precisely, the model defines a map
$
\Phi_{\vartriangle}
:
\mathbb{R}^{M\times d_{\mathrm{in}}}
\to
\mathbb{R}^{M \times d_{\mathrm{out}}}
$
and its causal weighted sum of discrete outputs $\mathcal{NC}^{\boldsymbol{t}}:
[0,T]\times \mathbb{R}^{M\times d_{\mathrm{in}}}
\to
\mathbb{R}^{d_{\mathrm{out}}}
$
of the form
\begin{align}
\label{eq:lift}
        \mathcal{NC}^{\boldsymbol{t}}(t,W)
    & \eqdef 
        \Phi_{\operatorname{C-Mask}}(t|\boldsymbol{t})^{\top}
        \,
        \Phi_{\vartriangle}(\mathbf{W}^{\boldsymbol{t}}),
\\
    \Phi_{\vartriangle}
& =
    L_{\mathrm{heads}}
\circ
    L_{\mathrm{lift}},
\end{align}
comprising \textit{lifting channels} $L_{\mathrm{lift}}:\mathbb{R}^{M\times d_{\mathrm{in}}}\to \mathbb{R}^{M\times (K d_{\mathrm{in}})}$ and \textit{causal heads} $L_{\mathrm{heads}}:\mathbb{R}^{M\times (K d_{\mathrm{in}})}\to \mathbb{R}^{M\times d_{\mathrm{out}}}$, which are then causally assembled via the \textit{causal mask} $\Phi_{\operatorname{C-Mask}}(\cdot|\boldsymbol{t}):\mathbb{R}\to \mathbb{R}^M$ mapping the $m^{th}$ row of $\Phi_{\vartriangle}$ to the $m^{th}$ point on the time-grid $\boldsymbol{t}$. At this point, we define the causal triangular lifting ($L_{\mathrm{lift}}$) and the causal heads ($L_{\mathrm{heads}}$). Precisely, the lifting channels $L_{\mathrm{lift}}$ ``develop'' the features of any input while preserving the causal dependence (adaptedness) of the input matrix; that is, the $i^{th}$ row produced by the lifting channel depends only on the first through $i^{th}$ rows of the input matrix, and on no later rows. Formally, 
\begin{equation}
 L_{\mathrm{lift}}(\mathbf{W}^{\boldsymbol{t}})\eqdef
\big[A^{(1)}\mathbf{W}^{\boldsymbol{t}}\,\big|\,A^{(2)}\mathbf{W}^{\boldsymbol{t}}\,\big|\,\cdots\,\big|\,A^{(K)}\mathbf{W}^{\boldsymbol{t}}\big],
\end{equation}
where, for each $k\in [K]_+$,
$
A^{(k)}\in \mathbb{R}^{M\times M}
$
is a lower-triangular matrix, and $[B^{(1)}|B^{(2)}|\dots|B^{(K)}]$ denotes the column-wise concatenation of matrices $B^{(1)},B^{(2)}, \dots,B^{(K)}$ sharing the same number of rows. Instead, the causal heads $L_{\mathrm{heads}}$ act row-wise on any input matrix $B\in \mathbb{R}^{M\times (K d_{\mathrm{in}})}$ by 
\begin{equation}
L_{\mathrm{heads}}(B) \eqdef [\varphi_1(B_{1:})^{\top},\ldots, \varphi_M(B_{M:})^{\top}]^{\top},
\end{equation}
where each $\varphi_1,\dots,\varphi_M:\mathbb{R}^{K d_{\mathrm{in}}}\to \mathbb{R}^{d_{\mathrm{out}}}$ is a MLP with ReLU activation functions. Finally, we discuss the continuous-time causal assembly through deterministic time masks. We write 
\begin{equation*}
\Delta_M^T \eqdef \left\{\boldsymbol{t}=(t_0,\ldots,t_M)\in[0,T]^{M+1}\,:\,0=t_0<t_1<\cdots<t_M=T
\right\}.
\end{equation*}
For every $\boldsymbol{t} \in \Delta_M^T$, the causal mask $\Phi_{\operatorname{C-Mask}}(\cdot|\boldsymbol{t}):[0,T]\to \mathbb{R}^M$ is defined by
\begin{equation}
\label{eq:causal_attention}
\begin{aligned}
\Phi_{\operatorname{C-Mask}}(t|\boldsymbol{t})& \eqdef
\begin{pmatrix}
            f_1(t)\,\sigma\big(
                    \tfrac{t-t_0}{\lambda}\big)
            \\
            \vdots
            \\
            f_M(t)\,\sigma\big(
                    \tfrac{t-t_{M-1}}{\lambda}
                \big)
        \end{pmatrix},
\end{aligned}
\end{equation}
where $f_1,\dots,f_M:[0,T] \to \mathbb{R}$ are MLPs with ReLU activation functions, $0<\lambda<T$ is a temperature parameter determining the model's pathwise regularity, and $\sigma:\mathbb{R}\to [0,1]$ is a continuous non-decreasing function with $\sigma(u)=0$ for $u \leq 0$ and $\sigma(1)=1$; e.g., $\sigma(t)=0 \cdot \mathsf{1}_{\{t\leq 0\}}+(3 t^2-2 t^3)\cdot\mathsf{1}_{\{0<t<1\}}+1 \cdot\mathsf{1}_{\{t \geq 1\}}$. Notice that the role of the lower-triangular lifting and the causal masks is summarized also  in
Figure~\ref{fig:masks_and_bm_paths}. The triangular matrices ensure that the $m$-th lifted feature row depends only on Brownian samples available up to the corresponding grid time, while the mask prevents that row from contributing before its information time.

\begin{figure}[h!]
    \centering
    \begin{subfigure}[t]{0.48\linewidth}
        \centering
        \includegraphics[width=\linewidth]{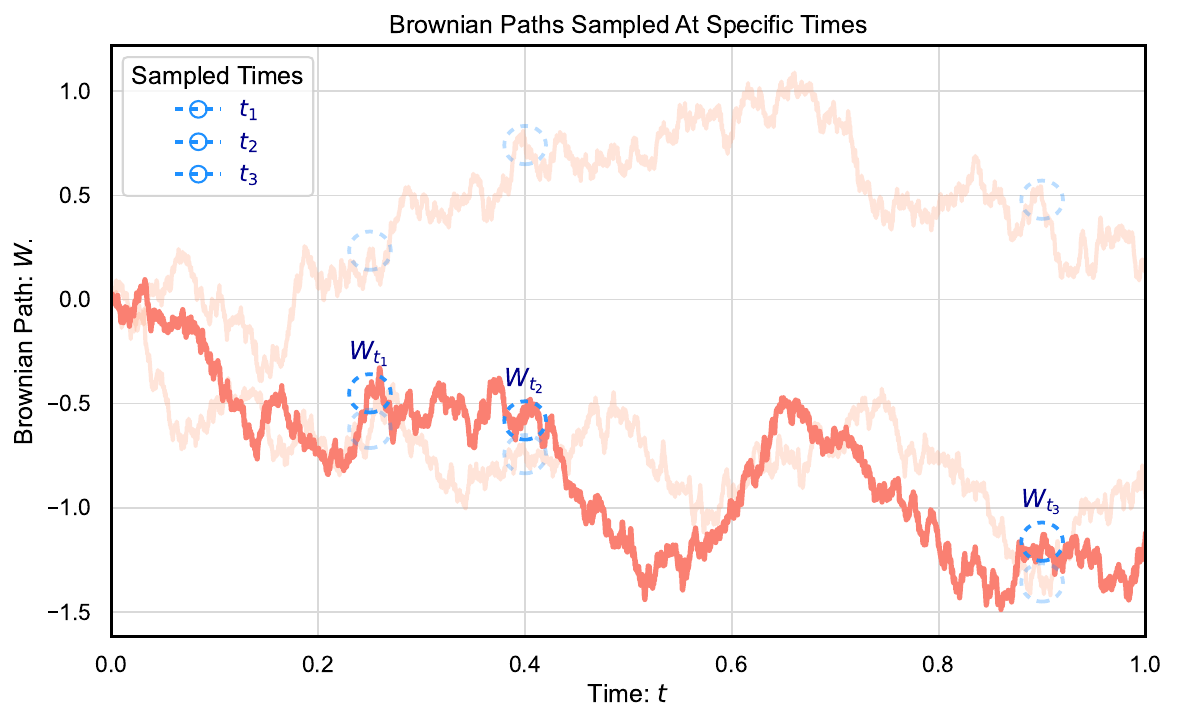}
        \caption{\textbf{Input 1 - Stochastic Signal:} \textit{Brownian motion} sample paths, evaluated at the sampled \textbf{time-points} $0=t_0<t_1<\cdots<t_M=T$, yielding the causally ordered inputs $W_{t_0},\dots,W_{t_{M-1}}$.}
        \label{fig:bm_path_evals}
    \end{subfigure}
    \hfill
    \begin{subfigure}[t]{0.48\linewidth}
        \centering
        \includegraphics[width=\linewidth]{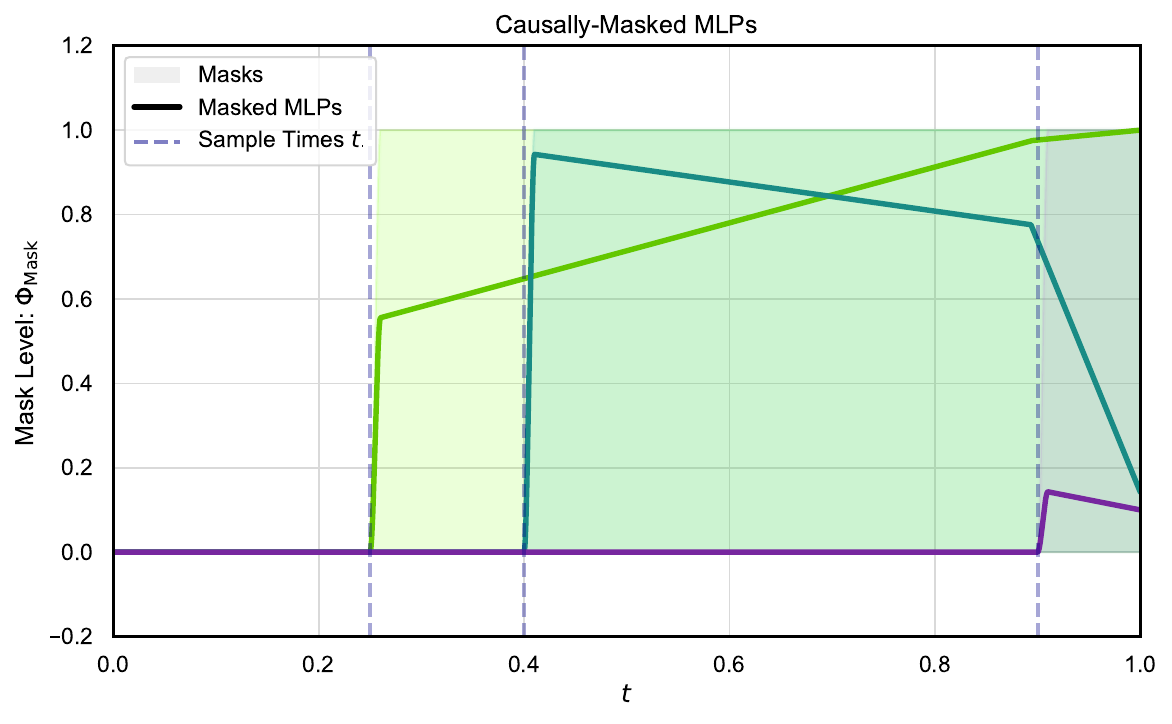}
        \caption{\textbf{Input 2 - Current Time:} The sampled Brownian path values are processed by $M$ row-wise heads. This is achieved by taking causally masked linear combinations of $M$ real-valued ReLU MLPs, each conditioned on the current time input $t\in [0,T]$, with the masks enforcing adaptedness.}
        \label{fig:temporal_masks}
    \end{subfigure}
    \caption{Inputs and causal assembly in the $\operatorname{NeuralChaos}$ architecture.  Panel \ref{fig:bm_path_evals} shows the stochastic input: the Brownian path is sampled at a finite ordered grid $0=t_0<t_1<\cdots<t_M=T$ producing the causally ordered samples These samples are processed by lower-triangular lifting matrices and row-wise neural heads, so that the $m^{th}$ random output is measurable with respect to the information available at the corresponding grid time. Panel \ref{fig:temporal_masks} illustrates the deterministic time masks used to assemble these adapted random outputs into a continuous-time predictable process. Each mask is inactive before its information time and turns on only after that time, thereby enforcing causality. The plot is schematic and is meant to illustrate the masking mechanism, not to represent a unique choice of mask functions.}
    \label{fig:masks_and_bm_paths}
\end{figure}

\subsection{Discussion of the NeuralChaos architecture}
First, we notice that the NeuralChaos architecture is \emph{causal} by construction. Indeed, define $Y\eqdef\Phi_{\vartriangle}(\mathbf{W}^{\boldsymbol{t}})=L_{\mathrm{heads}}(L_{\mathrm{lift}}(\mathbf{W}^{\boldsymbol{t}})) \in \mathbb{R}^{M \times d_{\rm out}}$  (cf.~Eq.~\eqref{eq:lift}). For each lifting channel $k=1,\ldots,K$, the matrix $A^{(k)}$ is lower triangular. Whence, the $m$-th row of $A^{(k)}\mathbf{W}^{\boldsymbol{t}}$ depends only on $W_{t_0},\ldots,W_{t_{m-1}}$. Consequently, the $m^{th}$ lifted row is $\mathcal F_{t_{m-1}}$-measurable. Since the head $\mathfrak{h}_m$ acts only on this $m^{th}$ row, the output row $Y_{m:}=\Phi_{\triangle}(\mathbf{W}^{\boldsymbol{t}})_{m:}$ is also \(\mathcal F_{t_{m-1}}\)-measurable. Moreover, the continuous-time output is $\mathcal{NC}^{\boldsymbol{t}}(t,W)=\Phi_{\mathrm{C\text{-}Mask}}(t\mid \boldsymbol{t})^{\top}
\Phi_{\vartriangle}(\mathbf{W}^{\boldsymbol{t}})=\sum_{m=1}^{M}f_m(t)\sigma\left(\frac{t-t_{m-1}}{\lambda}\right)Y_{m:}$.  The factor $\sigma\left(\frac{t-t_{m-1}}{\lambda}\right)$ vanishes for $t\le t_{m-1}$. Hence, the $m^{th}$ row output can contribute only after the time at which the information used to compute it is available. If $t>t_{m-1}$, then $Y_{m:}\in L^2(\mathcal F_{t_{m-1}};\mathbb R^{d_{\mathrm{out}}}) \subseteq L^2(\mathcal F_t;\mathbb R^{d_{\mathrm{out}}})$. Thus, $\mathcal{NC}^{\boldsymbol{t}}(t,W)$ depends only on Brownian information available up to time $t$. In particular, no future Brownian sample can enter the value of the process at time $t$. Moreover, since the masks are deterministic and continuous in $t$, each term $t\mapsto f_m(t)\sigma\left(\frac{t-t_{m-1}}{\lambda}\right)Y_{m:}$ is \emph{predictable}. Therefore, $\mathcal{NC}^{\boldsymbol{t}}(t,W)$ is predictable as a finite sum of predictable processes. Importantly, for fixed $M, K, d_{\mathrm{in}}$ and $d_{\mathrm{out}}$, and fixed grid $\boldsymbol{t}$, the $\operatorname{NeuralChaos}$ model is determined by \emph{finitely} many parameters: the entries of
the lower-triangular matrices $A^{(1)},\ldots,A^{(K)}\in\mathbb R^{M\times M}$, the weights and biases of the row-wise heads $\varphi_1,\ldots,\varphi_M$, the weights and biases of the time network $f_1,\ldots,f_M$, and the temperature parameter $\lambda>0$. Thus, for each fixed time $t$, the $\mathcal{NC}^{\boldsymbol{t}}(t,W)$ depends only on the finite vector $(\boldsymbol{t},W_{t_0},\ldots,W_{t_{M-1}})$, in addition to a finite parameter vector. Thus, no infinite-dimensional computations are being performed, which is arguably more realistic from a computational perspective. That is, no integral operator, nor any other infinite-rank operator, is used.

Finally, notice that our $\operatorname{NeuralChaos}$ model is not fully-connected. Indeed, each sub-network $\varphi_i$ determining the causal heads $L_{\mathrm{heads}}$ acts only on the corresponding row of the lifted representation $L_{\mathrm{lift}}(\mathbf{W}^{\boldsymbol{t}})$, rather than on the full matrix globally. Thus, the head stage is row-wise and relatively sparse. Moreover, the lifting stage inherits the lower-triangular structure of the matrices $A^{(1)},\dots,A^{(K)}$, rather than a fully-connected interaction across all time-indices. In particular, if the effective support-size per block is $r_i$, then the associated linear preprocessing scales like $\mathcal{O}(r_i M)$ per block, rather than $\mathcal{O}(M^2)$; and if $r_i\in \mathcal{O}(d)$, then this becomes $\mathcal{O}(dM)$. In particular, in this paper, we will consider a \textit{generative version} of the causal transformer, evaluated on a multidimensional Brownian-motion path generating the filtration underlying $\mathcal{H}_T^2(\mathbb{R}^{d})$.

\begin{remark}[Other activation functions]
\label{rem:standardization}
All $\operatorname{MLPs}$ will use the $\operatorname{ReLU}$ activation function in our analysis, due to standardization results such as~\cite{JMLR:v25:23-0912}, which reduce most $\operatorname{MLPs}$ to $\operatorname{ReLU}$ networks at a cost of $\mathcal{O}(1)$ complexity in depth, width, and number of non-zero parameters.
\end{remark}

\section{Main Results}\label{s:Main_Results}
This section presents the main results of the current paper. As said in the introduction (cf.~Subsection \ref{subsec::our_results}), we have a positive and a negative result.
\subsection{Optimal non-linear approximation results}
\label{s:Positive_Results}
The main approximation result is our optimal approximation theorem in $\mathcal{H}^2_T(\mathbb{R}^{d})$.
\begin{theorem}[Efficient (universal) approximation]
\label{thrm:happytimes}
Let $T>0$ and $W_{\cdot}=(W_t)_{0 \leq t \leq T}$ be a $D$-dimensional Brownian motion. In addition, let $\mathfrak{N}\mathfrak{C}_{T,D,d}$ be the class of $\mathbb{R}^{d}$-valued $\operatorname{NeuralChaos}$ outputs driven by $W_{\cdot}$, allowing arbitrary finite deterministic grids, finitely many lower-triangular lifting channels, row-wise $\operatorname{ReLU}$ heads, and deterministic causal masks as in Section \ref{sec::section_3}. Then, $\mathfrak{N}\mathfrak{C}_{T,D,d}$ is dense in $\mathcal{H}_T^2(\mathbb{R}^{d})$.\\
\indent Moreover, suppose $S>\frac{1}{2}$, $s>0$ and $X_{\cdot} \in \mathcal{H}^2_T(\mathbb{R}^{d})$ is $S$-compressible and in $\mathbb{D}^{s,2:d}_{T}$. Then, for every $N \in \mathbb{N}_{+}$, $P \in \mathbb{N}$, and $\varepsilon>0$, there exists $\widehat{X}_{N,P,\varepsilon} \in \mathfrak{N}\mathfrak{C}_{T,D,d}$ such that 
\begin{equation*}
    \|X-\widehat{X}_{N,P,\varepsilon}\|_{\mathcal{H}^2_T(\mathbb{R}^{d})} \lesssim_X N^{-\left(S-\frac{1}{2}\right)} + (1+P)^{-s/2} + \varepsilon.
\end{equation*}
In addition, let $D_{\mathrm H}(P,\rho_{N,X,\varepsilon})$, $W_{\mathrm H}(P,\rho_{N,X,\varepsilon})$, $S_{\mathrm H}(P,\rho_{N,X,\varepsilon})$ denote, respectively, the worst-case depth, width, and number of non-zero parameters of the stochastic Hermite head from Proposition~\ref{prop:ChaosMode}, over all Hermite factors of
total degree at most $P$, with $L^2(\Omega)$-accuracy $\rho_{N,X,\varepsilon}$. Then the
approximating $\operatorname{NeuralChaos}$ model may be chosen with grid size $M_{\mathrm{grid}} \leq C\,N (P+1)$, depth bounded by
\begin{equation*}
    \operatorname{depth}
\left(
\widehat X_{N,P,\varepsilon}
\right)
\le
C\left[
1+
D_{\mathrm H}\left(P,\rho_{N,X,\varepsilon}\right)
\right], 
\end{equation*}
width bounded by
\begin{equation*}
\operatorname{width}
\left(
\widehat X_{N,P,\varepsilon}
\right)
\le
C\left[
d+
N
\left(
1+
W_{\mathrm H}\left(P,\rho_{N,X,\varepsilon}\right)
\right)
\right],
\end{equation*}
and number of non-zero parameters bounded by
\begin{equation*}
    \operatorname{size}
\left(
\widehat X_{N,P,\varepsilon}
\right)
\le
C\,N
\left[
P+d+1+
S_{\mathrm H}\left(P,\rho_{N,X,\varepsilon}\right)
\right].
\end{equation*}
In the previous inequalities, $C>0$ is a universal constant. 
\end{theorem}
\begin{proof}
    See Appendix \ref{app:Positive_Results}.
\end{proof}

Naturally, one may wonder how strong the compressibility assumption is, since (as far as we know) it is a new concept in stochastic analysis. Interestingly enough, compressible processes are probabilistically generic\footnote{Notice that these are two different notions of genericity, whose interrelations have motivated entire descriptive set-theoretic books on the topic; cf.~\cite{oxtoby2013measure}, and as such we will not enter into such foundational discussions. }; i.e., if one samples a process according to any reasonable ``truly infinite-dimensional'' law then it is compressible. However, constructing such a process by hand seems to always look artificial; this is a direct analogue of how expander graphs are virtually impossible to exhibit by hand but are sampled with arbitrarily high probability in most random graph models (e.g., random regular graphs). 

\subsection{Structures in $\mathcal{H}_T^2(\mathbb{R}^{d})$}
\label{subsec:negative-results}
We now contrast two classes of processes in $\mathcal{H}_T^2(\mathbb{R}^{d})$: compressible processes and the familiar time-discretized Markovian SDEs with bounded state space and an arbitrary time-discretization scheme. Interestingly, if one selects a process at random in any non-degenerate manner, then it will almost surely be compressible; by contrast, it will almost surely not be such a discretized Markovian SDE. This raises our main question: ``Are most processes compressible, even when SDEs provide interpretable models for them?''

We now impose the following assumption. 

\begin{assumption}[Random \(\ell^{p_{\mathrm{cmp}}}\)-stable chaoslet series]
\label{ass:random-lp-stable-chaoslet-series}
Fix $s \in \mathbb{N}$, and let $(\varphi_k)_{k\in\mathbb N_+}$ be a fixed enumeration of the vector-valued chaoslet orthonormal basis of $\mathcal{H}_T^2(\mathbb{R}^{d})$; $\deg(\varphi_k)\in\mathbb N$ denotes the (Wiener-chaos) degree of $\varphi_k$. In addition, let $(\xi_k)_{k\in\mathbb N_+}$ 
be centered, unit-variance, i.i.d. real-valued sub-Gaussian random variables. Suppose that there exist $0<p_{\mathrm{cmp}}<2$, $\vartheta=(\vartheta_k)_{k\in\mathbb N_+}\in \ell^{p_{\mathrm{cmp}}}$, and a unitary operator
$U:\mathcal{H}_T^2(\mathbb{R}^{d})\to \mathcal{H}_T^2(\mathbb{R}^{d})$ such that the matrix $A_U \eqdef \left(
\left\langle U\varphi_n,\varphi_m\right\rangle_H
\right)_{m,n\in\mathbb N_+}$ defines a bounded linear operator $A_U:\ell^{p_{\mathrm{cmp}}}\to\ell^{p_{\mathrm{cmp}}}$. If, in addition, the statement is required in the Malliavin--Sobolev space $D_T^{s,2:d}$, we assume that
$\vartheta\in \ell_{\varphi,s}^2$, where
\begin{equation*}
    \ell_{\varphi,s}^2 \eqdef
\left\{
a=(a_k)_{k\in\mathbb N_+}:
\sum_{k=1}^{\infty}
(1+\deg(\varphi_k))^s |a_k|^2<\infty
\right\}
\end{equation*}
and that $A_U:\ell_{\varphi,s}^2\to\ell_{\varphi,s}^2$ is bounded. We define the random predictable process $X_\cdot \eqdef \sum_{k=1}^{\infty} \vartheta_k\xi_k\, U\varphi_k$ and denote by $\mu_{\vartheta,U} \eqdef \operatorname{Law}(X_\cdot)$ the law of $X_\cdot$ on $H_T^2(\mathbb R^d)$.
\end{assumption}

We state and prove the following proposition.

\begin{proposition}[Compressible Processes are Generic]
    \label{prop:genericity_of_compressibility}
    Let $s \in \mathbb{N}_{+}$ and let $\mu_{\vartheta,U}$ be any Borel probability measure on $\mathbb{D}^{s,2:d}_{T}$, satisfying Assumption \ref{ass:random-lp-stable-chaoslet-series}. If $\operatorname{Law}(X_\cdot)=\mu_{\vartheta,U}$, then $X_{\cdot}$ is $\big(S=\frac{1}{p_{\rm cmp}}\big)$-compressible and belongs to $\mathbb{D}^{s,2:d}_T$ $\mu$-almost surely. In particular, its best $N$-term
chaoslet approximation error is $\mathcal{O}\Big(
N^{-\big(\tfrac{1}{p_{\mathrm{cmp}}}-\tfrac{1}{2}\big)}
\Big)$.
\end{proposition}
\begin{proof}
    The proof directly follows from applying Proposition \ref{prop:compressible__unitary_lp} in Appendix \ref{app::section_2} to the separable Hilbert space $\mathbb{D}^{s,2:d}_T$ (resp. $\mathcal{H}^2_T(\mathbb{R}^{d})$ when $s=0$).
\end{proof}

We now show that discrete implementations of SDEs with \textit{bounded state-space dimension}, i.e. those which are implementable on computers, even under the idealization that Brownian motion is exactly computable, are not generic in $\mathcal{H}_T^2(\mathbb R^{d_X})$.  A fortiori, they are meagre therein.
{\color{black}
This is \textit{not to say} that SDEs, in their full generality, meaning with path-dependent coefficients, Markovian lifts, or signature augmentations~\cite{kidger2021neural,issa2023non,doi:10.1137/22M1512338,bank2025stochastic,philippuniversal2026}, are poor models.

Nor are we claiming that neural SDEs are incapable of approximating other SDEs of bounded state-space dimension, which is well-established; cf.~\cite{gonon2023deep,biagini2024approximation,rong2026convergence,kwossek2025universal}.  Rather, informally, we are claiming that time-discretized implementations of the class of processes of the form
\begin{equation}
\label{eq:NeuralSDEsBounded}
X_t = x + \int_0^t \hat{\mu}(s,X_s)\ud s + \int_0^t \hat{\sigma}(s,X_s) \ud W_s,
\end{equation}
where $\hat{\mu}:\mathbb{R}^{1+d}\to \mathbb{R}^d$ and $\hat{\sigma}:\mathbb{R}^{1+d}\to \mathbb{R}^{d\times d}$ are multilayer perceptrons, are not generic in $\mathcal{H}_T^2(\mathbb{R}^{d})$.  This echoes results the smallness of ODE solution in classes of continuous functions; cf.~\cite[Theorem 3]{rouhvarzi2025incremental}.

Our results do not concern \textit{path-dependent} neural SDEs, where one allows $\hat{\mu}$ and $\hat{\sigma}$ to depend on arbitrarily many historical points of the process $X_{\cdot}$ or, a fortiori, on its entire path; e.g.~\cite{arribas2020sig,hu2025deep,philippuniversal2026}.  Indeed, such processes are known to be universal in large relevant classes of processes; cf.~\cite{philippuniversal2026}.
}

Here our results concern traditional neural SDEs of the form~\eqref{eq:NeuralSDEsBounded}, common in most areas of the literature~\cite{cuchiero2020generative,pmlr-v119-kong20b,kidger2021neural}, and we ask the question if such processes are universal not in the class of SDEs but in the broad class of predictable processes.
\hfill\\
We consider the ``\textit{small}'' class $\mathfrak{B}$ consisting of all processes $X_{\cdot}\eqdef (X_t)_{0\le t\le T}$ in $\mathcal{H}_T^2(\mathbb R^{d_X})$ for which: 
there exists some $X^{(0)}\in L^2(\mathcal{F}_0;\mathbb R^{d_X})$, $0<\delta<T$, and some 
Borel
$f:[0,\infty)\times \mathbb{R}^{d_X+2D}\to \mathbb{R}^{d_X}$ such that, upon setting
$
N_{\delta}\eqdef \Big\lceil \frac{T}{\delta}\Big\rceil,
$ and $
t_i\eqdef (i\delta)\wedge T,
$ for each $i\in [N_{\delta}]$, 
one has $\mathbb{P}$-a.s.\ for all $0\le t\le T$
\begin{equation}
\label{eq:badguys_by_name}
\begin{aligned}
    X_t 
& = 
    \sum_{i=0}^{N_{\delta}-1}\,
        \mathbf{1}_{[t_i,t_{i+1})}(t)
        \,
        \mathbb{E}\big[
            X^{(i)}
        \big|
            \mathcal{F}_{t_i}
        \big],
\\
X^{(i+1)} & \eqdef 
    f(t_i, X^{(i)},W_{t_i},W_{t_{i+1}}),
\qquad \mbox{ for }i\in [N_{\delta}-1].
\end{aligned}
\end{equation}
We first argue that this is a natural generalized surrogate for the class of Markovian SDEs.
\begin{proposition}[Generalized predictable Euler-Maruyama ``SDEs'' universally approximate SDEs in $\mathcal{H}_T^2(\mathbb{R}^{d_X})$]
\label{prop:goodnews}
For measurable functions
\[
b:[0,T]\times\mathbb R^{d_X}\to \mathbb R^{d_X},
\qquad
\Sigma:[0,T]\times\mathbb R^{d_X}\to \mathbb R^{d_X\times D},
\]
that are Lipschitz and of linear growth, and $\xi\in L^2(\mathcal{F}_0;\mathbb{R}^{d_X})$, we consider the process\footnote{Since the solution process has $\mathbb{P}$-a.s.\ continuous paths and is adapted, it is predictable.} $X_{\cdot} \in \mathcal{H}_T^2(\mathbb R^{d_X})$%
solving the induced (Markovian) SDE
$$
X_t = \xi +\int_0^t\, b(s,X_s)\,ds + \int_0^t\, \Sigma(s,X_s)\,\ud W_s, \qquad \mbox{ for } t \in [0,T].
$$
For every approximation error $\varepsilon>0$, there exists some $\hat{X}_{\cdot}\in \mathfrak{B}$ satisfying
$
    \big\|
        \hat{X}-X
    \big\|_{\mathcal{H}_T^2(\mathbb R^{d_X})}
    <
    \varepsilon
$. 
\end{proposition}
\begin{proof}
    See Appendix \ref{app::negative}.
\end{proof}
Now, for the ``bad news''; namely, this class, although rich enough to approximate every Markovian SDE with Lipschitz drift and diffusion having linear growth, {\color{black}it is somehow \textit{opposite} to \textit{compressible} processes, in the sense that they are rare one will never sample them while compressible processes are generic; one will always draw one if they pick them in a non-degenerate fashion; in other words the following result should be considered in contrast to Proposition~\ref{prop:genericity_of_compressibility}.
}
\begin{proposition}[Generalized predictable Euler-Maruyama ``SDEs'' are not representative of $\mathcal{H}_T^2(\mathbb{R}^{d_X})$]
\label{prop:badnews}
The following is true of the class $\mathfrak{B}\subseteq \mathcal{H}_T^2(\mathbb R^{d_X})$
\begin{enumerate}
    \item[(i)] \textbf{\emph{Probabilistic {\color{black}rarity}:}}
        Let $\mu_G$ be a centred Gaussian measure on $\mathcal{H}_T^2(\mathbb R^{d_X})$ whose covariance operator $Q$ is diagonal in some orthonormal basis $(\eta_n)_{n=1}^{\infty}$, with strictly positive eigenvalues $(\lambda_n)_{n=1}^{\infty}$. Then, $\mu_G(\mathfrak{B}) = 0$,
    \item[(ii)] \textbf{\emph{Topological {\color{black}rarity}:}}
        $\mathfrak{B}$ is meagre in $\mathcal{H}_T^2(\mathbb R^{d_X})$.
\end{enumerate}
\end{proposition}
\begin{proof}
    See Appendix \ref{app::negative}.
\end{proof}

We now make the following remark. It is important to notice that the previous proposition should be read as a structural limitation of finite-dimensional Markovian parametrizations, instead of a limitation of SDE models in their natural domain. Indeed, the class \(\mathfrak B\) is rich
enough to approximate classical Lipschitz Markovian SDEs in
\(\mathcal H_T^2(\mathbb R^{d_X})\), but it remains negligible in the ambient
space of square-integrable predictable processes. Thus, finite-dimensional
Euler--Maruyama-type representations do not provide a typical processes in
\(\mathcal H_T^2(\mathbb R^{d_X})\). This is precisely the gap addressed by $\operatorname{NeuralChaos}$: instead of enlarging the Markovian state variable, the architecture directly parametrizes predictable processes through finite Brownian sampling, causal lifts, row-wise heads, and deterministic time masks.

\section{Implications in stochastic
optimal control and mathematical finance}
\label{s:MathFin}
This section illustrates the practical use of the proposed $\operatorname{NeuralChaos}$ architecture in a stochastic optimal control example and a standard problem of mathematical finance, namely dynamic hedging.\footnote{The numerical experiments have been implemented in \texttt{Python} using \texttt{PyTorch} and were executed on a Lenovo ThinkPad P14s Gen6 with Processor Intel Core Ultra 7 265H, 2200 Mhz, 16 Cores. The code can be found under the following link: \url{https://github.com/AnastasisKratsios/CausalTransformer}.} Importantly, for these two problems, simple parameterizable, but universal, families of processes which do not violate predictability and square-integrability constraints have far-reaching consequences (provided that our proposed parametrization achieves these properties while remaining easy to train using stochastic gradient descent-type algorithms). In particular, in the former example $\operatorname{NeuralChaos}$ parameterizes the optimal control, while, in the second, a candidate hedging strategy.

Throughout the experiments, the $\operatorname{NeuralChaos}$ output is denoted by $\widehat\alpha=\mathcal{N}\mathcal{C}^{\boldsymbol{t}}(\cdot,W)$, where $\boldsymbol{t}=(t_m)_{m=0}^{M}$ with $0=t_0<t_1<\cdots<t_M=T$ denotes the deterministic sampling grid. We use, as usual, $D$ for the Brownian dimension, $d_{X}$ for the state dimension, and $d_{\alpha}$ for the dimension of the control (or strategy). The number of lifting channels is denoted by $K_{\mathrm{lift}}$, while $N_{\mathrm{pay}}$ denotes the number of calibration instruments. Before proceeding, we remind the following lemma.
\begin{lemma}
\label{lem:control}
Fix $T>0$, $D, d_X, d_\alpha \in\mathbb N_+$, and let $b:[0,T]\times\mathbb R^{d_X}\times\mathbb R^{d_\alpha} \to\mathbb R^{d_X}$, $\Sigma:[0,T]\times\mathbb R^{d_X}\times\mathbb R^{d_\alpha} \to\mathbb R^{d_X\times D}$ be measurable functions. Assume that $b$ and $\Sigma$ are bounded and globally Lipschitz in $(x,a)$, uniformly in $t$. Let $x_0\in L^2(\mathcal F_0;\mathbb R^{d_X})$. Then, for every $\alpha\in \mathcal{H}_T^2(\mathbb R^{d_\alpha})$, the SDE
\begin{equation*}
    X_t^\alpha = x_0 + \int_0^t b(s,X_s^\alpha,\alpha_s)\,\ud s + \int_0^t \Sigma(s,X_s^\alpha,\alpha_s)\,\ud W_s
\end{equation*}
admits a unique strong solution with continuous adapted paths. In particular,
\(X^\alpha\) is predictable and belongs to $\mathcal{H}_T^2(\mathbb R^{d_X})$. Moreover, the control-to-state map $\Gamma:\mathcal{H}_T^2(\mathbb R^{d_\alpha})\to \mathcal{H}_T^2(\mathbb R^{d_X})$ such that $\Gamma(\alpha)\eqdef X^\alpha$, is well-defined and globally Lipschitz. 
\end{lemma}
\begin{proof}
    See Appendix~\ref{app:auxiliary-sde-facts}.
\end{proof}

\subsection{A toy stochastic optimal control illustration}\label{s:stoch_opt_control}
In this simple illustration, we consider for $d_\alpha = 1$ and every $\alpha \in \mathcal{H}_T^2(\mathbb{R}^{d_{\alpha}})$ the controlled process $X^\alpha \eqdef (X^{\alpha,1}_t,X^{\alpha,2}_t)_{t \in [0,T]}$ satisfying the SDE
\begin{equation}
    \label{eq:stoch_opt_control}
    \begin{aligned}
        \ud X^{\alpha,1}_t & = \ud W_t, \\
        \ud X^{\alpha,2}_t & = \theta(\alpha_t) \ud t,
    \end{aligned}
\end{equation}
with initial values $X^\alpha = x_0 \eqdef 0 \in \mathbb{R}^2$, where $W$ is a one-dimensional Brownian motion and
\begin{equation}
    \theta(a) \eqdef 
    \begin{cases}
        -2 + e^{a+1} & \text{if } a \in (-\infty,-1), \\
        a & \text{if } a \in [-1,1], \\
        2 - e^{-a+1} & \text{if } a \in (1,\infty). \\
    \end{cases}
\end{equation}
Note that \eqref{eq:stoch_opt_control} satisfies the assumptions of Lemma~\ref{lem:control} with $b(s,x,a) \eqdef (0,\theta(a))$ and $\Sigma(s,x,a) = (1,0)$. Then, we try to find the optimal control process minimizing the cost functional
\begin{equation}
    \label{eq:J_stoch_opt_control}
    \mathcal{J}(\alpha_{\cdot}) = \mathbb{E}\left[ -\sin(X^{\alpha,1}_T) X^{\alpha,2}_T \right] + \frac{\lambda}{2} 
    \Vert \alpha_\cdot \Vert_{\mathcal{H}_T^2},
\end{equation}
where $\lambda > 1$. Our main result (Theorem \ref{thrm:happytimes}) imply that for every error tolerance $\varepsilon>0$ there exists a $\operatorname{NeuralChaos}$ $\widehat{\alpha}_{\cdot}^{\varepsilon}$ that is $\varepsilon$-optimal, i.e.,
\begin{equation*}
    \mathcal{J}(\widehat{\alpha}^{\varepsilon}_{\cdot})
\le 
    \inf_{\alpha_{\cdot} \in \mathcal{H}_T^2}
        \mathcal{J}(\alpha_{\cdot})
    +
    \varepsilon
.
\end{equation*}
Indeed, assuming enough regularity of the cost functional $\mathcal{J}(\cdot)$, the key missing ingredient is the density of our simply parameterized processes $\widehat{\alpha}^{\varepsilon}_{\cdot}$ in $\mathcal{H}_T^2$; this allows us to compute an $\varepsilon$-optimal strategy for~\eqref{eq:J_stoch_opt_control} using one of our easy-to-parameterize processes $\widehat{\alpha}_{\cdot}^{\varepsilon}$.

For the numerical experiment, we choose $T = 1$ and $\lambda = 2$, and apply some martingale arguments to find the optimal control $\alpha^\star \in \mathcal{H}_T^2$ given by
\begin{equation*}
    \alpha^\star_t \eqdef \frac{1}{\lambda} e^{-\frac{1}{2}(T-t)} \sin(W_t), \qquad t \in [0,T].
\end{equation*}
Since $e^{i W_t} = e^{-t/2} \sum_{n=0}^\infty \frac{i^n}{n!} I_n^t(1^{\otimes n})$, the optimal control has a nice chaos expansion, i.e.,
\begin{equation*}
    \alpha^\star_t = \frac{e^{-T/2}}{\lambda} \operatorname{Im}\left( \sum_{n=0}^\infty \frac{i^n}{n!} I^t_n\left( 1^{\otimes n} \right) \right), \qquad t \in [0,T].
\end{equation*}
We generate independent training/test sets, each consisting of $S = 10^4$ samples of $W_{\cdot}$ discretized over an equidistant time grid with $M = 101$ points. Then, the $\operatorname{NeuralChaos}$ $\mathcal{NC}^{\boldsymbol{t}}(\cdot,W)$ with $K_{\mathrm{lift}} = 30$, $\operatorname{Head}_{\operatorname{Width}} = 20$, and $\operatorname{Head}_{\operatorname{Depth}} = 2$ is trained over $200$ epochs with learning rate $10^{-4}$ and batch size $10^3$. Figure~\ref{fig:soc} empirically demonstrates that the $\operatorname{NeuralChaos}$ is able to learn the optimal control of \eqref{eq:J_stoch_opt_control}.
\begin{figure}[h!]
    \centering
    \includegraphics[width=0.9\linewidth]{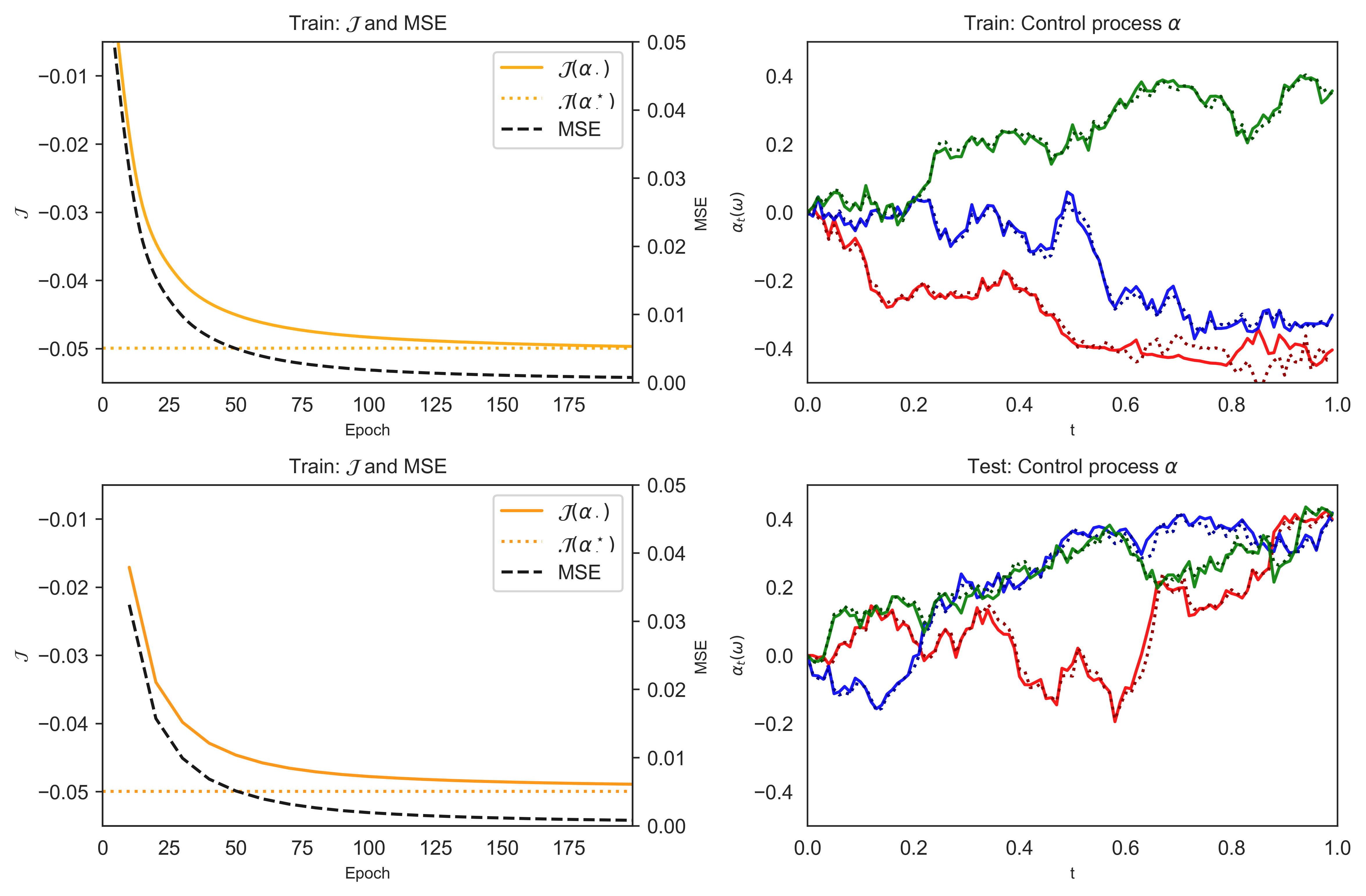}
    \caption{Stochastic optimal control: Learning the optimal control of the cost functional $\mathcal{J}(\alpha_\cdot)$ in \eqref{eq:J_stoch_opt_control}. Left-hand side: Value of $\mathcal{J}(\alpha_\cdot)$ and $\mathcal{J}(\alpha^\star_\cdot)$ as well as mean squared error (MSE) $\frac{1}{SM} \sum_{s=1}^S \sum_{m=1}^M \vert \alpha_{t_k}^\star(\omega_s) - \alpha_{t_k}(\omega_s) \vert^2$ on train/test set. Right-hand side: Comparison of learned control (dotted) and true optimal control (continuous) for three samples of the train/test set.}
    \label{fig:soc}
\end{figure}

\subsection{Dynamic hedging in the Black-Scholes model}\label{s:hedging}
For a coercive and lower semi continuous penalty functional $P:\mathcal{H}^2_T(\mathbb{R}^{d_{\alpha}}) \to [0,+\infty]$ that is proper (i.e., not identical $+\infty$), and a lower semi continuous functional $c:\mathbb{R}^{1+d_{X}+d_{\alpha}}\to [0,+\infty)$, the corresponding hedging functional
\begin{equation}
\label{eq:J_hedging}
    \mathcal{J}(\alpha_{\cdot})
=
    \mathbb{E}\biggl[
        \int_0^T\, 
        c(t,X_t^{\alpha},\alpha_t)
        \,\ud t
    \biggr]
    +
    P(\alpha_{\cdot})
\end{equation}
is coercive and lower semi continuous on $\mathcal{H}_T^2(\mathbb{R}^{d_{\alpha}})$. Then, by Tonelli's direct method (cf.~\cite[Theorem 1.15]{DalMasoGamma1993}), there exists at least one $\alpha^{\star}_{\cdot}\in \mathcal{H}_T^2(\mathbb{R}^{d_{\alpha}})$ achieving the optimal hedging energy; i.e.,
\begin{equation}
\label{eq:hedge_optimal}
    \mathcal{J}(\alpha_{\cdot}^{\star})
=
    \inf_{\alpha_{\cdot} \in \mathcal{H}_T^2(\mathbb{R}^{d_{\alpha}})}
    \mathcal{J}(\alpha_{\cdot}).
\end{equation}
Our main result (Theorem \ref{thrm:happytimes}) and the continuity of $\mathcal{J}(\cdot)$ now imply that for every slack parameter $\varepsilon>0$ there exists an $\widehat{\alpha}_{\cdot}^{\varepsilon}$ of the simple $\operatorname{NeuralChaos}$ ``parametric form'' which is $\varepsilon$-optimal for the hedging problem in Eq.~\eqref{eq:hedge_optimal}, i.e.,
\begin{equation}
\label{eq:chaos_energy}
    \mathcal{J}(\widehat{\alpha}^{\varepsilon}_{\cdot})
\le 
    \inf_{\alpha_{\cdot} \in \mathcal{H}_T^2(\mathbb{R}^{d_{\alpha}})}
        \mathcal{J}(\alpha_{\cdot})
    +
    \varepsilon
.
\end{equation}
For the numerical experiment, we consider a Black-Scholes model $\ud S_t = r S_t \ud t + \sigma_{\rm BS} S_t \ud W_t$, $t \in [0,T]$, with $S_0 = 100$, $r = 0.005$, $\sigma_{\rm BS} = 0.2$, $T = 1$, and learn the hedging strategy of a European call option by using the cost $c(X_t^{\alpha},\alpha_t) = \vert X^\alpha_t \vert^2$ with (discounted) profit-and-loss (PnL) process given by
\begin{equation*}
    X^\alpha_t = C(t,S_t) - \left( p + \int_0^t \alpha_s \,\ud \widetilde{S}_s \right), \qquad t \in [0,T],
\end{equation*}
where $C(t,S_t) \eqdef \mathbb{E}\left[ e^{-rT} (S_T-K_{\rm strike})_+ \big\vert \mathcal{F}_t \right]$ is the price of the European call option at time $t$ with spot $S_t$, and where $\widetilde{S}_s \eqdef e^{-rs} S_s$, $K_{\rm strike}= 100$, and $p = C(0,S_0) = \mathbb{E}[e^{-rT} (S_T-K_{\rm strike})_+] \approx 8.198$. Indeed, by the martingale representation theorem, i.e., $C(t,S_t) = p + \int_0^t \alpha^\star_s \,\ud\widetilde{S}_s$, we observe that
\begin{equation*}
    \begin{aligned}
        \mathbb{E}\left[ \int_0^T \vert X^\alpha_t \vert^2 \,\ud t \right] & = \int_0^T \mathbb{E}\left[ \left\vert \int_0^t (\alpha_s^\star - \alpha_s) \,\ud \widetilde{S}_s \right\vert^2 \right] \,\ud t = \sigma_{\rm BS}^2 \int_0^T \mathbb{E}\left[ \int_0^t (\alpha^\star_s - \alpha_s)^2 \widetilde{S}_s^2 \,\ud s \right] \ud t
    \end{aligned}
\end{equation*}
is minimal if and only if $\alpha=\alpha^\star$ (assuming the non-penalized setting $P(\alpha_\cdot) = 0$). Moreover, we choose the penalty function $P(\alpha_\cdot) = \lambda \Vert \alpha_\cdot \Vert_{\mathcal{H}^2_T}^2$ with $\lambda = 10^{-5}$ and generate independent training/test sets, each consisting of $10^4$ samples of $S_{\cdot}$ discretized over an equidistant time grid with $M = 101$ points. Then, the $\operatorname{NeuralChaos}$ $\mathcal{NC}^{\boldsymbol{t}}(\cdot,W)$ with $K_{\mathrm{lift}} = 30$, $\operatorname{Head}_{\operatorname{Width}} = 20$, and $\operatorname{Head}_{\operatorname{Depth}} = 2$ is trained over $200$ epochs with learning rate $5 \cdot 10^{-4}$ and batch size $10^3$. Figure~\ref{fig:hedg} empirically demonstrates that the $\operatorname{NeuralChaos}$ is indeed able to learn the considered hedging strategy.
\begin{figure}[h!]
    \centering
    \includegraphics[width=0.9\linewidth]{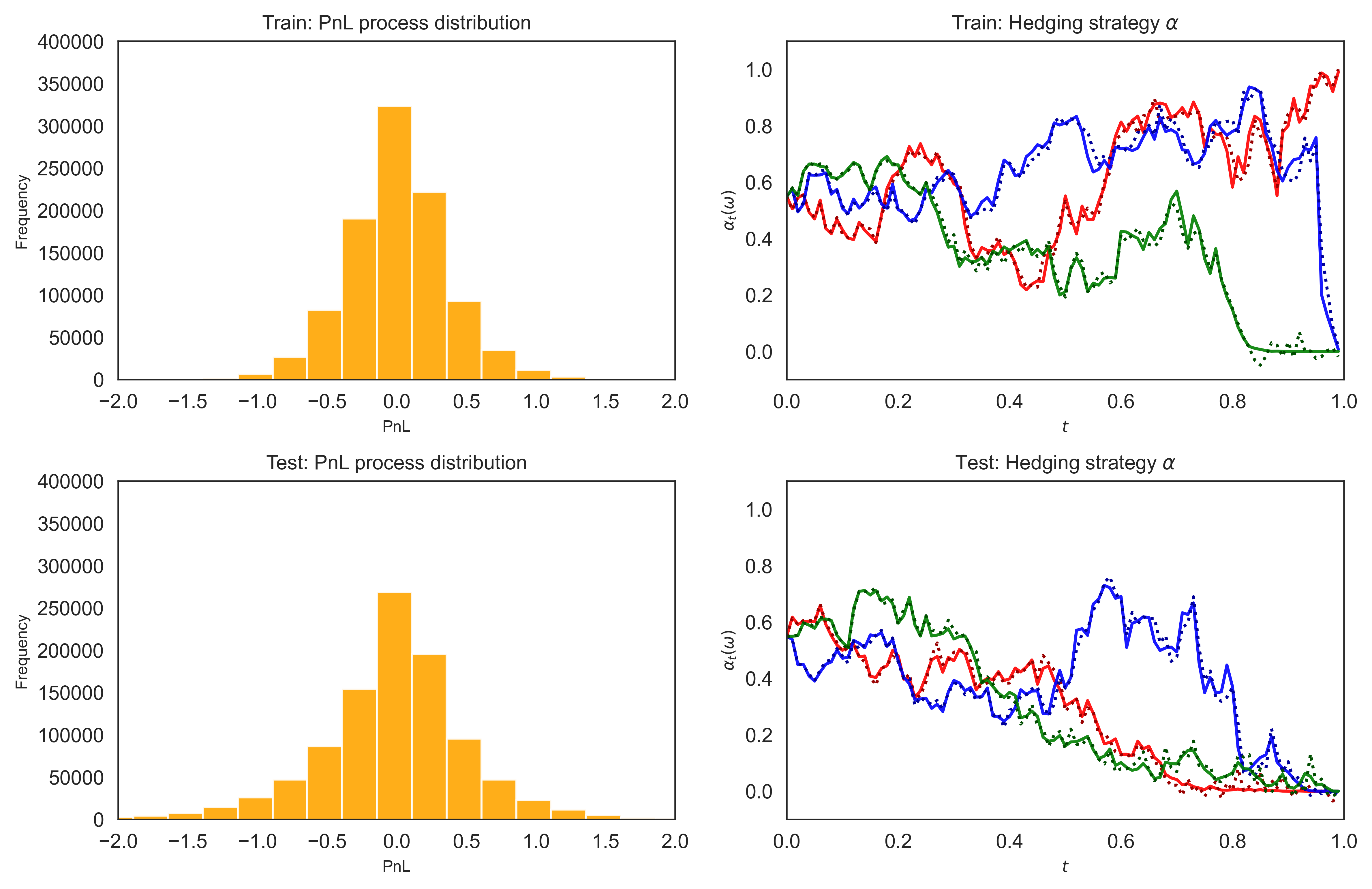}
    \caption{Hedging: Learning the hedging strategy of a European call option in a Black-Scholes model. Left-hand side: Empirical distribution of the profit-and-loss (PnL) process $C(t,S_t) - \big( p + \int_0^t \alpha_s \,\ud \widetilde{S}_s \big)$, $t \in [0,T]$, on train/test set. Right-hand side: Comparison of learned hedging strategy (dotted) and true hedging strategy (continuous) for three samples of the train/test set.}
    \label{fig:hedg}
\end{figure}

\section{Conclusion}
\label{sec:conclusion}
We introduced NeuralChaos, a finite-sampling neural parametrization of square-integrable predictable processes. The construction uses only finitely many evaluations of the driving Brownian motion, lower-triangular causal lifts, row-wise ReLU heads, and deterministic time masks, and therefore preserves predictability by construction. We proved that NeuralChaos is dense in \(\mathcal H_T^2(\mathbb R^d)\) and that, for processes which are both compressible and Malliavin--Sobolev regular, it inherits the corresponding best \(N\)-term chaoslet approximation rates. We also showed that compressibility is natural under random chaoslet-series models, while finite-dimensional Markovian Euler-Maruyama-type parametrizations are meagre and Gaussian-null sets in $\mathcal{H}_T^2$ and so are less likely to match real-world phenomena exactly. The numerical examples illustrate how the same causal architecture can be used in financial tasks where the unknown object is itself a predictable process.

The present results are approximation results, not statistical learning guarantees: they do not quantify finite-sample generalization, optimization error, or the effect of stochastic-gradient training. The quantitative rates are inherited from best \(N\)-term chaoslet approximation and therefore rely on compressibility and Malliavin--Sobolev regularity assumptions. Moreover, the complexity bounds are constructive and may be conservative, especially through the Hermite-head approximation and localization steps. Future work should develop sharper architecture-dependent estimates, statistical risk bounds, adaptive procedures for discovering sparse chaoslet structure, and extensions to more general noises, filtrations, and constrained stochastic control problems.

\section*{Acknowledgements and funding}
\label{s:Acknow}
A.\ Kratsios acknowledges financial support from the Natural Sciences and Engineering Research Council of Canada (NSERC) through Discovery Grant Nos.\ RGPIN-2023-04482 and DGECR-2023-00230.
We further acknowledge that resources used in the preparation of this research were provided, in part, by the Province of Ontario, the Government of Canada through CIFAR, and the industry sponsors of the Vector Institute.\footnote{\href{https://vectorinstitute.ai/partnerships/current-partners/}{https://vectorinstitute.ai/partnerships/current-partners/}} P.\ Schmocker acknowledges financial support from the FinsureTech Hub of ETH Zurich.

\appendix
\section{Proofs of technical results in Section \ref{sec::preliminaries}}\label{app::section_2}
\noindent\emph{Proof of Proposition \ref{prop:sparse_and_smooth}.} In order to prove Proposition \ref{prop:sparse_and_smooth}, we need the following preliminary lemma, which gives the standard best $N$-term consequence of compressibility; see, e.g., the non-linear approximation framework in \cite{alvarez2025neural}.
\begin{lemma}[Compressibility yields best \(N\)-term approximation]
\label{lemma::compressibility}
Let $S>\frac{1}{2}$ and let $X \in \mathcal{H}^2_{T}(\mathbb{R}^{d})$ be $S$-compressible with compressibility constant $C_X$. Then, for every $N \in \mathbb{N}_{+}$, we have
\begin{equation}
    \inf_{\substack{\beta_1,\ldots,\beta_N\in\mathbb R^d\\
\mathfrak{m}^{(1)},\ldots,\mathfrak{m}^{(N)}\in \mathfrak{C}}}
\left\|
X-\sum_{n=1}^{N}\beta_n \mathfrak m^{(n)}
\right\|_{\mathcal{H}_T^2}
\lesssim_X
N^{-(S-\frac12)}.
\end{equation}
\end{lemma}
\begin{proof}
First, we consider the (adapted) chaoslet orthonormal basis of $\mathcal{H}^2_T(\mathbb{R}^{d})$, graded by homogeneous Wiener-chaos degree $\mathfrak{C}=\bigcup_{k=0}^{\infty} \mathfrak{C}_k$. Then, we consider the chaoslet expansion of $X \in \mathcal{H}^2_T(\mathbb{R}^{d})$: $X = \sum_{k=0}^{\infty}\sum_{\mathfrak{m} \in \mathfrak{C}_k} \beta_{k,m} \mathfrak{m}$, with $\beta_{k,m} \in \mathbb{R}^{d}$, and the non-increasing rearrangement, say $(\beta_r^{\star})_{r \geq 1}$, of the scalar family $(|\beta_{k,m}|_{\mathbb{R}^{d}})_{k \geq 0, \mathfrak{m} \in \mathfrak{C}_k}$. Since $X$ is by assumption $S$-compressible, we have $\beta_r^{\star} \leq C r^{-S}$ for all $r \in \mathbb{N}_{+}$. Hence, by using that the chaoslets form an orthonormal basis of $\mathcal{H}^2_T(\mathbb{R}^{d})$, the best $N$-term approximation is obtained by retaining the $N$ largest coefficients in $\mathbb{R}^{d}$-norm; we denote this $N$-term truncation by $X_{N}$. Thus, we have
\begin{equation*}
    \|X-X_N\|^2_{\mathcal{H}^2_T(\mathbb{R}^{d})}=\sum_{r>N}|\beta_r^{\star}|^2 \leq C_X^2 \sum_{r>N} r^{-2 S},
\end{equation*}
where the last inequality follows from the compressibility bound. Since $S > \frac{1}{2}$ implying $2S > 1$, we have $\sum_{r>N} r^{-2 S} \lesssim_{X} N^{1-2 S}$ and thus $\|X-X_N\|_{\mathcal{H}^2_T(\mathbb{R}^{d})} \lesssim_{X} N^{-\left(S-\frac{1}{2}\right)}$. This concludes the proof.
\end{proof}
We are now ready to prove our main technical abstract non-linear approximation result. Again, we consider the chaoslet expansion $X = \sum_{k=0}^{\infty}\sum_{\mathfrak{m} \in \mathfrak{C}_k} \beta_{k,m} \mathfrak{m}$, with $\beta_{k,m} \in \mathbb{R}^{d}$. Since every $\mathfrak{m} \in \mathfrak{C}_{k}$ belongs to the $k^{th}$ homogeneous Wiener chaos, the orthogonal projection of $X$ onto the first $P \in \mathbb{N}$ Wiener chaoses is $\Pi_{\leq P} X \eqdef \sum_{k=0}^{P} \sum_{\mathfrak{m} \in \mathfrak{C}_k} \beta_{k,\mathfrak{m}}\mathfrak{m}$, and, therefore, its tail is $X-\Pi_{\leq P}X=\sum_{k=P+1}^{\infty}\sum_{\mathfrak{m} \in \mathfrak{C}_k} \beta_{k,\mathfrak{m}}\mathfrak{m}$. By orthonormality of the chaoslets, Parseval's identity, and the fact that $k>P$ (which implies $(1+k)^{s} \geq (1+P)^{s}$), we have
\begin{equation}
\begin{aligned}
        \|X-\Pi_{\le P}X\|_{\mathcal{H}_T^2(\mathbb R^d)}^2
&=
\left\|  \sum_{k=P+1}^{\infty}\sum_{\mathfrak{m} \in \mathfrak{C}_k} \beta_{k,\mathfrak{m}}\mathfrak{m} \right\|_{\mathcal{H}_T^2(\mathbb R^d)}^2 \\
&=
\sum_{k=P+1}^{\infty}\sum_{\mathfrak{m} \in \mathfrak{C}_k} \vert \beta_{k,\mathfrak{m}} \vert_{\mathbb{R}^d}^2 \\
&=
\sum_{k=P+1}^{\infty}
(1+k)^{-s}
(1+k)^s
\sum_{\mathfrak m\in\mathfrak C_k}
|\beta_{k,m}|_{\mathbb R^d}^2 \\
&\le
(1+P)^{-s}
\sum_{k=P+1}^{\infty}
(1+k)^s
\sum_{\mathfrak m\in\mathfrak C_k}
|\beta_{k,m}|_{\mathbb R^d}^2 \\
&\le
(1+P)^{-s}
\|X\|_{\mathbb{D}_T^{s,2:d}}^2.
\end{aligned}
\end{equation}
It remains to prove that we can approximate the low-chaos part $\Pi_{\leq P} X$ by $N$ chaoslets from $\cup_{k=0}^{P}\mathfrak{C}_k$. Since $X$ is $S$-compressible, so is $\Pi_{\leq P} X$ because the coefficients of $\Pi_{\leq P} X$ in the chaoslet expansion are $(\beta_{k,\mathfrak{m}}\mathsf{1}_{\{\mathfrak{m} \in \cup_{k=0}^{P}\mathfrak{C}_k\}})_{s,\mathfrak{m}}$. In particular $\Pi_{\leq P} X$ inherits the same coefficient rearrangement upper bound of $X$, and it has no non-zero coefficients in the orthogonal complement of span $(\cup_{k=0}^{P} \mathfrak{C}_k)$. The result follows by an application of Lemma \ref{lemma::compressibility} to $\Pi_{\leq P} X$ and the triangular inequality.

\begin{proposition}[Polynomial compressibility under
\(\ell^{p_{\mathrm{cmp}}}\)-stable unitary change of basis]
\label{prop:compressible__unitary_lp}
Let $(H,\langle\cdot,\cdot\rangle_H)$ be a separable infinite-dimensional Hilbert space, and let $e_\cdot=(e_n)_{n\in\mathbb N_+}$ be an orthonormal basis of $H$. In addition, let $(\xi_n)_{n\in\mathbb N_+}$ be centered, unit-variance, i.i.d. real-valued sub-Gaussian random variables. Fix $0<p_{\mathrm{cmp}}<2$, assume $\vartheta=(\vartheta_n)_{n\in\mathbb N_+}\in\ell^{p_{\mathrm{cmp}}}$, and  $U:H\to H$ be a unitary operator and define the matrix of $U$ in the basis $e_\cdot$ by $A_U
\eqdef
\left(
\langle Ue_n,e_m\rangle_H
\right)_{m,n\in\mathbb N_+}.
$
Assume that $A_U$ defines a bounded operator $A_U:\ell^{p_{\mathrm{cmp}}}\to\ell^{p_{\mathrm{cmp}}}$. Define $X\eqdef \sum_{n=1}^{\infty}
\vartheta_n\xi_n\,Ue_n$. Then, $X$ is a well-defined $H$-valued random variable, almost surely. Moreover, $X$ is $1/p_{\mathrm{cmp}}$-compressible in the basis
$e_\cdot$. In addition, if $\Pi_N^{e}(X)$ denotes a best $N$-term approximation of $X$ in the basis $e_\cdot$, then there exists an almost surely finite positive
random variable $\mathbf{C}$ such that
\begin{equation*}
    \mathbb P
\left(
\sup_{N\in\mathbb N_+}
N^{\frac1{p_{\mathrm{cmp}}}-\frac12}
\|X-\Pi_N^{e}(X)\|_H
\le
\mathbf{C}
\right)
=
1.
\end{equation*}
\end{proposition}
\begin{proof}[{Proof of Proposition~\ref{prop:compressible__unitary_lp}}]
Set
$
a_n\eqdef \vartheta_n \xi_n
$
for each $n\in \mathbb{N}_+$.  
Since $(\xi_n)_{n=1}^{\infty}$ are $(C,c)$-sub-Gaussian\footnote{We recall that a Borel probability measure $\mu$ on a separable Hilbert space $(H,\langle\cdot,\cdot\rangle_H)$ is $(C,c)$-\textit{sub-Gaussian} if $\int_H \exp(\langle x,h \rangle_H) \mu(dx) \leq C \exp\left( c \Vert h \Vert^2 \right)$, for all $h \in H$ and some constants $C,c\ge 0$.}, they have finite moments of every positive order, and in particular
$
\mathbb{E}[|\xi_1|^{p_{\mathrm{cmp}}}]<\infty
$.
Thus, by the Fubini-Tonelli Theorem we have
\[
    \mathbb{E}\Big[\sum_{n=1}^{\infty}|a_n|^{p_{\mathrm{cmp}}}\Big]
=
    \sum_{n=1}^{\infty}|\vartheta_n|^{p_{\mathrm{cmp}}}\,
    \mathbb{E}[|\xi_1|^{p_{\mathrm{cmp}}}]
=
    \mathbb{E}[|\xi_1|^{p_{\mathrm{cmp}}}]
    \|\vartheta\|_{\ell^{p_{\mathrm{cmp}}}}^{p_{\mathrm{cmp}}}
<
    \infty
.
\]
Therefore,
$
\sum_{n=1}^{\infty}|a_n|^{p_{\mathrm{cmp}}}<\infty,
$
$\mathbb{P}$-a.s.,
and so
$
a\eqdef (a_n)_{n=1}^{\infty}\in \ell^{p_{\mathrm{cmp}}},
$
$\mathbb{P}$-a.s.
Since $0<p_{\mathrm{cmp}}<2$, one has the continuous embedding
$\ell^{p_{\mathrm{cmp}}}\subseteq \ell^2$. Indeed, if
$a=(a_n)_{n=1}^{\infty}\in \ell^{p_{\mathrm{cmp}}}$, then
$
\{n\in \mathbb{N}_+:\ |a_n|>1\}
$
is finite, and hence
\[
\sum_{n} |a_n|^2
=
\sum_{|a_n|>1} |a_n|^2
+
\sum_{|a_n| \leq 1} |a_n|^2
\leq
\sum_{|a_n|>1} |a_n|^2
+
\sum_{|a_n| \leq 1} |a_n|^{p_{\mathrm{cmp}}}
<\infty
.
\]
Then, $
\sum_{n=1}^{\infty}\|a_nUe_n\|_H^2
=
\sum_{n=1}^{\infty}|a_n|^2
<
\infty
$, $\mathbb{P}$-a.s., as $U$ is unitary. Thus, the series
$
X
=
\sum_{n=1}^{\infty} a_n\,Ue_n
=
\sum_{n=1}^{\infty}\vartheta_n \xi_n\,Ue_n
$
converges in $H$ $\mathbb{P}$-a.s. In particular, $X$ is a well-defined $H$-valued random variable, being the almost sure limit of the measurable partial sums
$
X_M\eqdef \sum_{n=1}^{M}\vartheta_n \xi_n\,Ue_n
$.
Let
$
b_m\eqdef \langle X,e_m\rangle_H
$
for each $m\in \mathbb{N}_+$,
and write
$
b\eqdef (b_m)_{m=1}^{\infty}
$.
Since $(e_m)_{m=1}^{\infty}$ is an orthonormal basis of $H$, $b$ is precisely the coefficient sequence of $X$ in the basis $e_{\cdot}$. For each $M\in \mathbb{N}_+$ and $m\in \mathbb{N}_+$,
$
    \langle X_M,e_m\rangle_H
=
    \sum_{n=1}^{M}\,
        a_n\langle Ue_n,e_m\rangle_H
$.
Since $X_M\to X$ in $H$, $\mathbb{P}$-a.s., we obtain
$
    b_m
=
    \lim_{M\to\infty}\langle X_M,e_m\rangle_H
=
    \lim_{M\to\infty}\sum_{n=1}^{M}a_n\langle Ue_n,e_m\rangle_H
$,
$\mathbb{P}$-a.s.,
for all $m\in \mathbb{N}_+$. As 
$
A_U\eqdef \big(\langle Ue_n,e_m\rangle_H\big)_{m,n=1}^{\infty}
$
defines a bounded linear operator on $\ell^{p_{\mathrm{cmp}}}$, and since
$
a^{(M)}\eqdef (a_1,\dots,a_M,0,0,\dots)
$
converges to $a$ in $\ell^{p_{\mathrm{cmp}}}$, $\mathbb{P}$-a.s., we have
$
A_Ua^{(M)}\to A_Ua
$
in $\ell^{p_{\mathrm{cmp}}}$, $\mathbb{P}$-a.s., and therefore coordinate-wise. Thus
$
b=A_Ua
$, $\mathbb{P}$-a.s.
Therefore, if we write
$
K_{p_{\mathrm{cmp}}}(U)\eqdef \|A_U\|_{\ell^{p_{\mathrm{cmp}}}\to \ell^{p_{\mathrm{cmp}}}},
$
then
$
\|b\|_{\ell^{p_{\mathrm{cmp}}}}
\le
K_{p_{\mathrm{cmp}}}(U)\|a\|_{\ell^{p_{\mathrm{cmp}}}}
$, $\mathbb{P}$-a.s.
Thus, $b\in \ell^{p_{\mathrm{cmp}}}$, $\mathbb{P}$-a.s. 

Now fix $\omega$ in the full-probability event on which both
$a(\omega)\in \ell^{p_{\mathrm{cmp}}}$ and
$b(\omega)\in \ell^{p_{\mathrm{cmp}}}$. Let
$
(b_k^{\star})_{k=1}^{\infty}
$
denote the decreasing rearrangement of
$
(|b_m|)_{m=1}^{\infty}
$.
Since
$
    k\,(b_k^{\star})^{p_{\mathrm{cmp}}}
\le
    \sum_{j=1}^{k}(b_j^{\star})^{p_{\mathrm{cmp}}}
\le
    \sum_{j=1}^{\infty}(b_j^{\star})^{p_{\mathrm{cmp}}}
=
    \|b\|_{\ell^{p_{\mathrm{cmp}}}}^{p_{\mathrm{cmp}}}
$,
it follows that
$
b_k^{\star}
\le
\|b\|_{\ell^{p_{\mathrm{cmp}}}}\,k^{-1/p_{\mathrm{cmp}}}
$
for every $k\in \mathbb{N}_+$.  
Because $e_{\cdot}$ is an orthonormal basis, the best $N$-term approximation error of $X$ in the basis $e_{\cdot}$ is obtained by retaining the $N$ largest coefficients in modulus; equivalently,
$
\|X-\Pi_N(X)\|_H^2
=
\sum_{k>N}(b_k^{\star})^2
$.
Consequently,
$
\|X-\Pi_N(X)\|_H^2
\le
\|b\|_{\ell^{p_{\mathrm{cmp}}}}^2
\sum_{k>N}k^{-2/p_{\mathrm{cmp}}}
$.
Since $0<p_{\mathrm{cmp}}<2$, one has $2/p_{\mathrm{cmp}}>1$, and therefore there exists a constant $C_{p_{\mathrm{cmp}}}>0$, depending only on $p_{\mathrm{cmp}}$, such that
$
\sum_{k>N}k^{-2/p_{\mathrm{cmp}}}
\le
C_{p_{\mathrm{cmp}}}\,N^{1-2/p_{\mathrm{cmp}}},
$
for every $N\in \mathbb{N}_+$.
Taking square roots yields
\begin{equation*}
\|X-\Pi_N(X)\|_H
\le
C_{p_{\mathrm{cmp}}}^{1/2}
\|b\|_{\ell^{p_{\mathrm{cmp}}}}\,
N^{\frac12-\frac1{p_{\mathrm{cmp}}}}
=
C_{p_{\mathrm{cmp}}}^{1/2}
\|b\|_{\ell^{p_{\mathrm{cmp}}}}\,
N^{-S_{\mathrm{app}}},
\end{equation*}
where
$
S_{\mathrm{app}}\eqdef \frac1{p_{\mathrm{cmp}}}-\frac12>0
$.
The bound on $A_U$ gives for every $N\in \mathbb{N}_+$,
\[
    \|X-\Pi_N(X)\|_H
\le
    C_{p_{\mathrm{cmp}}}^{1/2}
    K_{p_{\mathrm{cmp}}}(U)
    \|a\|_{\ell^{p_{\mathrm{cmp}}}}\,
    N^{-S_{\mathrm{app}}}
    \qquad \mathbb{P}\mbox{-a.s.}
\]
We now define
$
    \mathbf{C}
\eqdef
    C_{p_{\mathrm{cmp}}}^{1/2}
    K_{p_{\mathrm{cmp}}}(U)
    \|(a_n)_{n=1}^{\infty}\|_{\ell^{p_{\mathrm{cmp}}}}
=
    C_{p_{\mathrm{cmp}}}^{1/2}
    K_{p_{\mathrm{cmp}}}(U)
\big(\sum_{n=1}^{\infty}|\vartheta_n \xi_n|^{p_{\mathrm{cmp}}}\big)^{1/p_{\mathrm{cmp}}}
$.
Since each
$
\omega\mapsto |\vartheta_n \xi_n(\omega)|^{p_{\mathrm{cmp}}}
$
is measurable, the partial sums
$
(S_M(\omega)\eqdef \sum_{n=1}^{M}|\vartheta_n \xi_n(\omega)|^{p_{\mathrm{cmp}}})_{M=1}^{\infty}
$
are measurable, and
$
\sum_{n=1}^{\infty}|\vartheta_n \xi_n|^{p_{\mathrm{cmp}}}
=
\lim_{M\to\infty}S_M
$
is measurable as monotone limit of measurable functions. Hence $\mathbf{C}$ is measurable and is a bona-fide random variable. Moreover, by the first part of the proof,
$
\mathbf{C}<\infty
$,
$\mathbb{P}$-a.s.
Therefore,
$
\|X-\Pi_N(X)\|_H
\le
\mathbf{C}\,N^{-S_{\mathrm{app}}}
$
for every $N\in \mathbb{N}_+$, 
$\mathbb{P}$-a.s., which is equivalent to
$
\mathbb{P}\big(
    \sup_{N\in \mathbb{N}_+}
        N^{S_{\mathrm{app}}}\|X-\Pi_N(X)\|_H
    \le
        \mathbf{C}
\big)
=
1.
$
This proves the claim.
\end{proof}

\section{Proof of Theorem \ref{thrm:happytimes}} \label{app:Positive_Results}
The proof of Theorem \ref{thrm:happytimes} involves several steps. The first step, in Subsection \ref{s:Proofs__ss:MainThrm___sss:Lifting}, connects the chaoslet representation from Subsection \ref{subsec::special_chaos} with the finite-sampling architecture of Section~\ref{sec::section_3}. A chaoslet contains stochastic factors of the form (cf.~\eqref{eq:reshuffled}) $h_{\alpha_r^{i,k}}\big(Z_{i,k}^r\big)$, where each $Z_{i,k}^{r}$ is a Brownian Haar coordinate, which are (in fact) deterministic linear combinations of finitely many Brownian samples. In particular, we need to show that our $\operatorname{NeuralChaos}$ architecture can produce the required Gaussian inputs $Z_{i,k}^{r}$ from the sampled values of the driving Brownian motion. This is the role of the causal lift in Subsection \ref{s:Proofs__ss:MainThrm___sss:Lifting}. The second step, in Subsection \ref{subsec::step_2}, proves that the stochastic Hermite-Haar factor of a chaoslet can be approximated by a MLP with ReLU activation
functions. With the latter result at hand, the third step in Subsection \ref{subsubsec:time-masking} proves that it is possible to approximate the elementary chaoslet in Eq.~\eqref{eq:reshuffled} by using $\operatorname{ReLU-MLPs}$ of $O(1)$ complexity in the time variable $t$ and the same complexity as with the Wiener chaos variable $\omega$. Finally, Step 4, assembles the preceding steps and completes the proof of Theorem \ref{thrm:happytimes}. 
\subsection{\emph{Step 1 -- causal lift}}\label{s:Proofs__ss:MainThrm___sss:Lifting}
Step 1 comprises three lemmas. Lemma \ref{lem:emulation} records that only finitely many dyadic sampling times are required. Instead, Lemma \ref{lem:triangular_filters} shows that lower-triangular linear maps preserve adaptedness. Finally, Lemma \ref{lem:construt_my_buddies_please__WtoX} combines the previous two facts to realize the finitely many Brownian Haar coordinates appearing in a chaoslet as a sparse (causal) linear transformation of Brownian samples.
\begin{lemma}[Counting the size of time-indexing sets]\label{lem:emulation}
Let $T>0$ be a (fixed) time-horizon, $I_{\rm H}, J_{\rm H} \in \mathbb{N}$, and $i_{0} \in \mathbb{N}$ such that $J_{\rm H}\leq 2^{i_0}-1$. In addition, let $\Lambda_{I_{\rm H}, J_{\rm H}, i_0}\eqdef\{(i,k)\,:\,i_0 \leq i \leq i_0 + I_{\rm H},\,\,0 \leq k \leq J_{\rm H}\}$; in particular, $0 \leq k < 2^{i}$ for every $(i,k) \in \Lambda_{I_{\rm H}, J_{\rm H}, i_0}$. Let $\mathcal{S}_{I_{\rm H}, J_{\rm H}, i_0}^T\eqdef \bigcup_{i=i_0}^{i_0+I_{\rm H}}\left\{\frac{j T}{2^{i+1}}\,:\,j=0,\ldots,2 J_{\rm H}+2\right\}$ and $0=s_0<s_1<\ldots<s_{n_{\rm grid}} \leq T$ be the increasing enumeration of the elements in $S_{I_{\rm H}, J_{\rm H}, i_0}^T$.  Then $n_{\rm grid}=(J_{\rm H}+1)(I_{\rm H}+2)$. Moreover, for every Brownian component $r \in [D]_{+}$ the normalized Gaussian coordinate in Eq.~\eqref{eq:stochastic_haar_coordinate} can be computed from the Brownian samples on this finite grid. In particular, if $\mathbf{W}^{I_{\rm H}, J_{\rm H}, i_0} \eqdef [W_{s_0}^{\top},\ldots,W_{n_{\rm grid}}^{\top}]^{\top}$, then $\mathbf{W}^{I_{\rm H}, J_{\rm H}, i_0} \in \mathbb{R}^{(n_{\rm grid}+1)\times D}$. 
\end{lemma}
\begin{proof}
First, we count the number of (distinct) dyadic sampling times. For every $i=i_0,\ldots,i_0+I_{\rm H}$, we set $S_i^T\eqdef\left\{\frac{j T}{2^{i+1}}\,:\,j=0,\ldots,2 J_{\rm H} + 2\right\}$; whence, the set in the statement can be written as $\mathcal{S}_{I_{\rm H}, J_{\rm H}, i_0}^T = \bigcup_{i=i_0}^{i_0+I_{\mathrm H}} S_i^T$. Notice that, since $\frac{2 J_{\rm H}+2}{2^{i+1}} \leq \frac{2 J_{\rm H}+2}{2^{i_0+1}}=\frac{J_{\rm H}+1}{2^{i_0}}\leq 1$, all points of $\mathcal{S}_{I_{\rm H}, J_{\rm H}, i_0}^T$ belong to $[0,T]$. Since multiplication by $T>0$ is injective, it is enough to count the number of distinct points in the set $\mathcal S_{I_{\mathrm H},J_{\mathrm H},i_0}^1\eqdef \bigcup_{i=i_0}^{i_0+I_{\mathrm H}}
\left\{ \frac{j}{2^{i+1}}: j=0,\ldots,2J_{\mathrm H}+2\right\}$. Now, set $S_{i_0}^{1}\eqdef\left\{\frac{j}{2^{i_0+1}}\,:\,j=0,\ldots,2J_{\mathrm H}+2\right\}$; it contains exactly $2 J_{\mathrm H}+3$ points. For every finer level $i=i_0+1,\ldots,i_0+I_{\mathrm H}$, the points with  even numerator, already appeared at the previous level. Therefore, the only new points are
those with odd numerator $\frac{2 a + 1}{2^{i+1}}$, $a=0,\ldots,J_{\mathrm{H}}$. Equivalently, if we write $r=i+1$, the new reduced dyadic rationals are $\frac{q}{2^{r}}$, $ r=i_0+2,\ldots,i_0+I_{\mathrm H}+1$ and $q\in\{1,3,\ldots,2J_{\mathrm H}+1\}$. Therefore, the following disjoint decomposition holds true:
\begin{equation*}
    \mathcal S_{I_{\mathrm H},J_{\mathrm H},i_0}^{1} = S_{i_0}^{1} \sqcup \left\{
\frac{q}{2^r}:
r=i_0+2,\ldots,i_0+I_{\mathrm H}+1,\ 
q\in\{1,3,\ldots,2J_{\mathrm H}+1\}
\right\}.
\end{equation*}
Therefore, $\#\mathcal S_{I_{\mathrm H},J_{\mathrm H},i_0}^{1}=
(2J_{\mathrm H}+3)+ I_{\mathrm H}(J_{\mathrm H}+1)=(J_{\mathrm H}+1)(I_{\rm H}+2)+1$. In particular, if $0=s_0<s_1<\cdots<s_{n_{\mathrm{grid}}}\le T$ is the increasing enumeration of the distinct elements of $\mathcal S_{I_{\mathrm H},J_{\mathrm H},i_0}^T$, then $n_{\mathrm{grid}}+1 = (J_{\mathrm H}+1)(I_{\mathrm H}+2)+1$. We need to verify that for every Brownian component $r \in [D]_{+}$, the normalized Gaussian coordinate in Eq.~\eqref{eq:stochastic_haar_coordinate} can be computed from the Brownian samples on this finite grid. In order to do so, it is sufficient to rewrite the sampling times appearing in Eq.~\eqref{eq:stochastic_haar_coordinate} with the common denominator $2^{i+1}$ and notice that, since $0 \leq k \leq J_{\rm H}$, the resulting numerators all belong to $\{0,\ldots,2 J_{H}+2\}$. Whence, all three sampling times belong to $S_i^{T} \subset S_{I_{\mathrm H},J_{\mathrm H},i_0}^{T}$. Finally, each $W_{s_j}$ is a $D$-dimensional Brownian vector. Therefore, the
Brownian sample matrix $\mathbf{W}^{I_{\rm H}, J_{\rm H}, i_0} \eqdef [W_{s_0}^{\top},\ldots,W_{n_{\rm grid}}^{\top}]^{\top} \in \mathbb{R}^{(n_{\rm grid}+1)\times D}$. This concludes the proof. 
\end{proof}
\begin{lemma}[Adaptedness of linear combinations]\label{lem:triangular_filters}
Let $0=s_0\leq s_1 \leq \cdots \leq s_{n_{\mathrm{grid}}}\le T$ be a deterministic time-grid. Let $q \in \mathbb{N}_{+}$, suppose that $\xi_{s_j} \in L^{2}(\mathcal{F}_{s_j};\mathbb{R}^{q})$, $j=0,\ldots,n_{\mathrm{grid}}$, and consider the random matrix $\Xi\eqdef[\xi_{s_0}^{\top},\ldots,\xi_{s_{n_{\mathrm{grid}}}}^{\top}]^{\top} \in \mathbb{R}^{(n_{\mathrm{grid}}+1)\times q}$. Now, let $R \in \mathbb{N}_{+}$, with $R \leq n_{\mathrm{grid}}+1$, and let $A \in \mathbb{R}^{R \times (n_{\rm grid}+1)}$ be lower triangular in the rectangular sense, i.e., $A_{j,i}=0$ for every $i>j$. In particular, for every $j \in [R]_{+}$, $(A\Xi)_{j:} \in L^2(\mathcal F_{s_{j-1}};\mathbb R^q)$.
\end{lemma}
\begin{proof}
Fix any $j \in [R]_{+}$. Since $A$ is lower triangular, $A_{j,i}=0$ for every $i>j$ and, therefore, $(A \Xi)_{j:}=\sum_{i=1}^{n_{\rm grid}+1} A_{j,i} \xi_{s_{i-1}}=\sum_{i=1}^{j} A_{j,i} \xi_{s_{i-1}}$. Now, for each $i \leq j$, since $s_{i-1} \leq s_{j-1}$ and $\mathcal{F}_{\cdot}=(\mathcal{F}_{t})_{0 \leq t \leq T}$ is a filtration, we have $\mathcal{F}_{s_{i-1}} \subseteq \mathcal{F}_{s_{j-1}}$. Hence, $\xi_{s_{i-1}} \in L^2(\mathcal{F}_{s_{j-1}};\mathbb{R}^{q})$, for every $i \in [j]_{+}$. Since $A$ is deterministic and the sum is finite, it follows that $(A \Xi)_{j:} \in L^2(\mathcal{F}_{s_{j-1}};\mathbb{R}^{q})$. Since $j \in [R]_{+}$ is arbitrary, this concludes the proof.
\end{proof}
Now, we observe (cf.~the discussion above) that one can always find a lower-triangular matrix mapping a path-wise random noise sample.
\begin{lemma}[Adapted linear realization of the arguments in the Hermite polynomials.]
\label{lem:construt_my_buddies_please__WtoX}
Let $T>0$ be a (fixed) time-horizon, $I_{\rm H}, J_{\rm H} \in \mathbb{N}$, and $i_{0} \in \mathbb{N}$ such that $J_{\rm H}\leq 2^{i_0}-1$. In addition, let $\Lambda_{I_{\rm H}, J_{\rm H}, i_0}\eqdef\{(i,k)\,:\,i_0 \leq i \leq i_0 + I_{\rm H},\,\,0 \leq k \leq J_{\rm H}\}$; in particular, $\# \Lambda_{I_{\rm H}, J_{\rm H}, i_0}=(I_{\rm H}+1)(J_{\rm H}+1)$ (cf.~Lemma \ref{lem:emulation}). Let $0<s_0<s_1<\ldots<s_{n_{\rm grid}} \leq T$ be the ordered dyadic sampling grid of Lemma \ref{lem:emulation}, where $n_{\rm grid}=(J_{\rm H}+1)(I_{\rm H}+2)$. Define the Brownian sample matrix  $\mathbf{W}^{I_{\rm H}, J_{\rm H}, i_0} \eqdef [W_{s_0}^{\top},\ldots,W_{n_{\rm grid}}^{\top}]^{\top} \in \mathbb{R}^{(n_{\rm grid}+1)\times D}$, and, for every $(i,k)\in\Lambda_{I_{\rm H}, J_{\rm H}, i_0}$ and $r \in [D]_{+}$
\begin{equation}\label{eq::gaussian_coordinates}
 Z_{i,k}^r= \tfrac{2^{i/2}}{\sqrt{T}} \,W^r_{\frac{Tk}{2^i}}
    -\tfrac{2^{i/2+1}}{\sqrt{T}}\,W^r_{\frac{T(1+2k)}{2^{i+1}}}
    +\tfrac{2^{i/2}}{\sqrt{T}}\,W^r_{\frac{T(k+1)}{2^i}}.
\end{equation}
Enumerate $\Lambda_{I_{\mathrm H},J_{\mathrm H},i_0} =
\{(i_\ell,k_\ell):\ell\in [R_{\mathrm H}]_{+}\}$ so that the terminal information times $\theta_\ell\eqdef \frac{(k_\ell+1)T}{2^{i_\ell}}$ are non-decreasing, i.e., $\theta_1\le\theta_2\le\cdots\le\theta_{R_{\mathrm H}}$, with arbitrary ordering among ties. If  $\mathbf{Z}^{I_{\mathrm H},J_{\mathrm H},i_0}\eqdef [Z_{i_1,k_1}^{\top},Z_{i_2,k_2}^{\top},\ldots,Z_{i_{R_{\mathrm H}},k_{R_{\mathrm H}}}^{\top}]^{\top}$, then there exists a deterministic matrix $A^{I_{\mathrm H},J_{\mathrm H},i_0} \in \mathbb R^{R_{\mathrm H}\times(n_{\mathrm{grid}}+1)}$ such that $A^{I_{\mathrm H},J_{\mathrm H},i_0}
W^{I_{\mathrm H},J_{\mathrm H},i_0}=
Z^{I_{\mathrm H},J_{\mathrm H},i_0}$. Moreover, $\left\|A^{I_{\mathrm H},J_{\mathrm H},i_0}\right\|_0 = 3R_{\mathrm H} = 3(I_{\mathrm H}+1)(J_{\mathrm H}+1)$.
\end{lemma}
\begin{proof}
    We write the rows of $\mathbf{Z}^{I_{\mathrm H},J_{\mathrm H},i_0}$ as $(Z_{i_\ell,k_\ell})_{\ell=1}^{R_{\rm H}}$ ordered so that $\frac{(k_\ell+1)T}{2^{i_\ell}}$ is non-decreasing in $\ell$. Fix $\ell \in [R_{\rm H}]_{+}$. By Lemma \ref{lem:emulation}, $\frac{k_\ell T}{2^{i_\ell}}, \frac{(2k_\ell+1)T}{2^{i_\ell+1}}, \frac{(k_\ell+1)T}{2^{i_\ell}}$ belong to $\{s_0,\ldots,s_{n_{\mathrm{grid}}}\}$. Therefore, there exist unique indices $a_\ell,b_\ell,c_\ell\in\{0,\ldots,n_{\mathrm{grid}}\}$ such that $s_{a_\ell}=\frac{k_\ell T}{2^{i_\ell}}, s_{b_\ell}= \frac{(2k_\ell+1)T}{2^{i_\ell+1}}, s_{c_\ell}=\frac{(k_\ell+1)T}{2^{i_\ell}}$. Define $A^{I_{\mathrm H},J_{\mathrm H},i_0} \in \mathbb R^{R_{\mathrm H}\times(n_{\mathrm{grid}}+1)}$ row by row as follows $A^{I_{\mathrm H},J_{\mathrm H},i_0}_{\ell,m}\eqdef\frac{2^{i_\ell/2}}{\sqrt T}\mathbf 1_{\{m=a_\ell\}}-\frac{2^{i_\ell/2+1}}{\sqrt T}\mathbf 1_{\{m=b_\ell\}}+ \frac{2^{i_\ell/2}}{\sqrt T}\mathbf 1_{\{m=c_\ell\}}$ for $m=0,\ldots,n_{\mathrm{grid}}$. Then, for each \(\ell=1,\ldots,R_{\mathrm H}\),
    \begin{equation*}
        \begin{aligned}
\left(
A^{I_{\mathrm H},J_{\mathrm H},i_0}
\mathbf{W}^{I_{\mathrm H},J_{\mathrm H},i_0}
\right)_{\ell:}
&=
\tfrac{2^{i_\ell/2}}{\sqrt T}
W_{s_{a_\ell}}
-
\tfrac{2^{i_\ell/2+1}}{\sqrt T}
W_{s_{b_\ell}}
+
\tfrac{2^{i_\ell/2}}{\sqrt T}
W_{s_{c_\ell}} \\
&=
\tfrac{2^{i_\ell/2}}{\sqrt T}
W_{\frac{k_\ell T}{2^{i_\ell}}}
-
\tfrac{2^{i_\ell/2+1}}{\sqrt T}
W_{\frac{(2k_\ell+1)T}{2^{i_\ell+1}}}
+
\tfrac{2^{i_\ell/2}}{\sqrt T}
W_{\frac{(k_\ell+1)T}{2^{i_\ell}}} \\
&=
Z_{i_\ell,k_\ell}.,
\end{aligned}
    \end{equation*}
and, therefore, $A^{I_{\mathrm H},J_{\mathrm H},i_0}
\mathbf{W}^{I_{\mathrm H},J_{\mathrm H},i_0}=\mathbf{Z}^{I_{\mathrm H},J_{\mathrm H},i_0}$.  Moreover, by construction, every non-zero entry in the $\ell^{th}$ row of $A^{I_{\mathrm H},J_{\mathrm H},i_0}$ occurs in a column indexed by a time no later than $\frac{k_{\ell}+1}{2^{i}\ell}$, so $A^{I_{\mathrm H},J_{\mathrm H},i_0}$ is causal with respect to the chosen ordering. Finally, since $\frac{k_\ell T}{2^{i_\ell}}<\frac{(2k_\ell+1)T}{2^{i_\ell+1}}<\frac{(k_\ell+1)T}{2^{i_\ell}}$, the indices $a_\ell, b_\ell, c_\ell$ are pairwise distinct. Hence the $\ell^{th}$ row of $A^{I_{\mathrm H},J_{\mathrm H},i_0}$ has exactly three non-zero entries. Since the matrix has $R_{\mathrm H} =(I_{\mathrm H}+1)(J_{\mathrm H}+1)$ rows, it follows that $\left\|A^{I_{\mathrm H},J_{\mathrm H},i_0}\right\|_0 =3R_{\mathrm H}=3(I_{\mathrm H}+1)(J_{\mathrm H}+1)$. This concludes the proof. 
\end{proof}

\subsection{\textbf{\emph{Step 2 -- Chaos-coding: coding the NeuralChaos into MLPs}}}\label{subsec::step_2}
The goal of Step 2 is to prove that once the Gaussian coordinates are available as finite-dimensional inputs (cf.~Step 1), the stochastic Hermite-Haar factor of a chaoslet can be approximated by a MLP with ReLU activation functions. To this end (cf.~the notation around Eq.~\eqref{eq:ith_WienerChaos}), let $\Lambda_n\eqdef\{(i_1,k_1),\ldots, (i_{J_n},k_{J_n})\}$ be the finite family of Haar coordinate retained at chaos degree $n$; $|\Lambda_n|=J_n$. For every\footnote{Notice that in Step 1 the Haar coordinates retained in the finite degree-$n$ Hermite representation are realized by applying the causal lift to a finite dyadic sampling grid. In particular, we have to choose $I_{\mathrm H},J_{\mathrm H},i_0$ so that the retained set $\Lambda_n$ is contained in $\Lambda_{I_{\mathrm H},J_{\mathrm H},i_0}$; cf.~Lemma \ref{lem:emulation}.} $j \in [J_n]_{+}$ and $r \in [D]_{+}$, we define $Z_{j,r}\eqdef Z_{i_j, k_j}^{r}$; in particular, $\mathbf{Z} \eqdef (Z_{j,r})_{j \in [J_n]_{+}, r \in [D]_{+}}$ and $h_{\boldsymbol{\alpha}}(\mathbf{Z}) \eqdef \prod_{r=1}^{D}\prod_{j=1}^{J_n} h_{\alpha_{j,r}}(Z_{j,r})$, where $\boldsymbol{\alpha} \in \mathbb{N}^{J_n \times D}$ and  $|\boldsymbol{\alpha}|=\sum_{r=1}^{D}\sum_{j=1}^{J_n}\alpha_{j,r}=n$. Our first main step is to solve the problem at any (arbitrary) fixed intermediate time $0 \leq t \leq T$.
\begin{proposition}[{Single-time adapted $L^2(\mathcal{F}_t)$-bound}]\label{prop:ChaosMode}
Let $n, D, J_n \in \mathbb{N}$ and $\boldsymbol{\alpha} \in \mathbb{N}^{J_n \times D}$ as above. Then, for every $\varepsilon>0$, there exists a $\operatorname{MLP}$ with $\operatorname{ReLU}$ activation function $\Phi_{\boldsymbol{\alpha},\varepsilon}:\mathbb{R}^{J_n \times D} \rightarrow \mathbb{R}$ such that 
\begin{equation}
\label{eq:main_lp_estimate}
        \mathbb{E}\left[ \left\vert h_{\boldsymbol\alpha}(\mathbf{Z}) - \Phi_{\boldsymbol\alpha,\varepsilon}(\mathbf{Z}) \right\vert^2 \right]^\frac{1}{2} \leq \varepsilon;
\end{equation}
$\Phi_{\boldsymbol\alpha,\varepsilon}(\mathbf{Z}) \in L^2(\mathcal{F}_t)$ when $\mathbf{Z}$ is $\mathcal{F}_t$-measurable. Moreover, the complexity of $\Phi_{\boldsymbol{\alpha},\varepsilon}(\cdot)$ can be chosen in the following way. We set $Q_n \eqdef J_n D$, write $h_m(x)=\sum_{\ell=0}^{m} c_{m,\ell} x^{\ell}$, and define $A_{m,B}\eqdef \sum_{\ell=0}^{m}|c_{m,\ell}|B^{\ell}$ and $A_{\boldsymbol{\alpha},B}\eqdef 1 + \max_{j \in [J_n]_{+} \\ r \in [D]_{+}} A_{\alpha_{j,r},B}$, for $B > 0$. In addition, we define $\mathcal{C}_{\boldsymbol{\alpha},B}\eqdef 1 + \sup_{z \in [-B,B]^{J_n \times D}}|h_{\boldsymbol{\alpha}}(z)|$ and choose $B_{\boldsymbol{\alpha},\varepsilon}$ such that
\begin{equation*}
\left(
\mathbb E\left[
\mathbf 1_{\{\|\mathbf{Z}\|_\infty>B_{\boldsymbol{\alpha},\varepsilon}\}}
|h_{\boldsymbol{\alpha}}(\mathbf{Z})|^2\right]\right)^{1/2}\le\frac{\varepsilon}{3},\quad\text{and}\quad\mathcal C_{\boldsymbol{\alpha},B_{\boldsymbol{\alpha},\varepsilon}}\,
\mathbb P\left(\|\mathbf{Z}\|_\infty>B_{\boldsymbol{\alpha},\varepsilon}\right)^{1/2}\le
\frac{\varepsilon}{3}.    
\end{equation*}
Finally, we set $\varepsilon_{\mathrm{loc}}\eqdef
\min\left\{\frac12,\frac{\varepsilon}{3}\right\}$ and $\ell_{Q_n}\eqdef
\left\lceil\log_2(Q_n\vee2)\right\rceil$.

Then, $\Phi_{\boldsymbol\alpha,\varepsilon}(\cdot)$ may be chosen with $\operatorname{width}(\Phi_{\boldsymbol\alpha,\varepsilon}(\cdot)) \leq C(Q_n+n+1)$,
\begin{equation*}
\begin{aligned}
    \operatorname{depth}(\Phi_{\boldsymbol\alpha,\varepsilon}(\cdot)) \leq C\Bigg[
& 1+
\max_{\substack{j \in [J_n]_{+}\\ r \in [D]_{+}}}
\log_2\left(
1+
\frac{
2 Q_n\,\alpha_{j,r}\,
A_{\alpha_{j,r},B_{\boldsymbol{\alpha},\varepsilon}} A_{\boldsymbol{\alpha},B_{\boldsymbol{\alpha},\varepsilon}}^{Q_n-1}}{\varepsilon_{\mathrm{loc}}}
\right) \\
&\qquad\quad
+
\ell_{Q_n}
\left(
1+
\log_2\left(
1+
\frac{
2^{Q_n+1}Q_n
A_{\boldsymbol{\alpha},B_{\boldsymbol{\alpha},\varepsilon}}^{\,Q_n}
}{
\varepsilon_{\mathrm{loc}}
}
\right)
\right)
\Bigg]
\end{aligned}
\end{equation*}
and number of non-zero parameters
\begin{equation*}
\begin{aligned}
\operatorname{size}(\Phi_{\boldsymbol\alpha,\varepsilon}(\cdot)) \leq C\Bigg[
&(Q_n+n)
\left(
1+
\max_{\substack{j \in [J_n]_{+}\\ r \in [D]_{+}}}
\log_2\left(
1+
\frac{
2Q_n\,\alpha_{j,r}\,
A_{\alpha_{j,r},B_{\boldsymbol{\alpha},\varepsilon}}\,
A_{\boldsymbol{\alpha},B_{\boldsymbol{\alpha},\varepsilon}}^{\,Q_n-1}
}{
\varepsilon_{\mathrm{loc}}
}
\right)
\right)
\\
&\qquad\qquad
+
Q_n
\left(
1+
\log_2\left(
1+
\frac{
2^{Q_n+1}Q_n
A_{\boldsymbol{\alpha},B_{\boldsymbol{\alpha},\varepsilon}}^{\,Q_n}
}{
\varepsilon_{\mathrm{loc}}
}
\right)
\right)
+1
\Bigg],
\end{aligned}
\end{equation*}
where $C>0$ is a universal constant independent of $n,J_n,D,\boldsymbol{\alpha},\varepsilon$.
\end{proposition}

The proof of the previous proposition hinges on a number of intermediate steps. The first one is the content of the next subsection, and it is a non-random version of the previous proposition (uniformly on compact sets).    
\subsubsection{The non-random local-uniform sub-case}
Here, we show the non-random version of Proposition \ref{prop:ChaosMode}, namely that Hermite polynomials evaluated at a deterministic point $\mathbf{z}=(z_{j,\ell})_{j \in [J_n]_{+},\ r \in [D]_{+}}$ -- rather than at an integral of a symmetric function over an arbitrarily large hypercube -- can be approximated arbitrarily well by a $\operatorname{ReLU-MLP}$ evaluated at the same point. Once this result is in place, we apply it to the stochastic integral of a tensorized Haar wavelet, which amounts to sampling a Brownian path at finitely many (time-instants) points  and then applying a simple linear transformation.   

\begin{lemma}[Local-uniform ReLU approximation of (normalized) univariate Hermite polynomials]\label{prop:Hermite_poly_approx_on_compact} There exists a strictly positive constant $C$ such that for every $m \in \mathbb{N}$, $B>0$, and every $\varepsilon \in (0,1)$ there exists a $\operatorname{ReLU-MLP}$ $\Phi_{m, B, \varepsilon}^{\rm Herm}:\mathbb{R} \rightarrow \mathbb{R}$ satisfying 
\begin{equation*}
    \sup_{x \in [-B,B]}|h_m(x)-\Phi_{m, B, \varepsilon}^{\rm Herm}(x)| \leq \varepsilon.
\end{equation*}
Moreover, $\Phi_{m, B, \varepsilon}^{\rm Herm}$ satisfies the bounds $\operatorname{width}(\Phi_{m, B, \varepsilon}^{\rm Herm}) \leq C m$, $\operatorname{depth}(\Phi_{m, B, \varepsilon}^{\rm Herm}) \leq C\left(1+\log_{2}(\frac{m A_{m,B}}{\varepsilon})\right)$, and $\operatorname{size}(\Phi_{m, B, \varepsilon}^{\rm Herm}) \leq C m\left(1 +\log_2\left(
\frac{m A_{m,B}}{\varepsilon}
\right)\right)$.
\end{lemma}
\begin{proof}
If $m \in \{0,1\}$, then $h_m(\cdot)$ is affine; whence, it can be represented by a $\operatorname{ReLU-MLP}$ with a bounded number of parameters and (hidden) layers. In particular,  after increasing the universal constant $C$ if necessary, we obtain the claimed bound in Lemma \ref{prop:Hermite_poly_approx_on_compact}. Therefore, we assume $m \geq 2$, and define the rescaled polynomial $q_{m,B}(u)\eqdef h_m(Bu)=\sum_{r=0}^{m}h_{m,r}B^r u^r$, for every $u \in [-1,1]$. Whence, the coefficient $\ell^{1}$-norm of $q_{m,B}(\cdot)$ is bounded by $\sum_{r=0}^{m}|h_{m,r}|B^r\leq A_{m,B}$. By~\cite[Lemma 7.5]{petersen2024mathematical}, there exists a universal constant $C_0>0$ such that,  after setting $\delta \eqdef \min\Big\{1,\,\frac{\varepsilon}{C_0A_{m,B}}\Big\}$, we deduce that there exists a $\operatorname{ReLU-MLP}$ $\Psi_{m,B,\delta}:\mathbb R\to\mathbb R$ such that 
\begin{equation*}
\sup_{u\in[-1,1]}
\left|q_{m,B}(u)-\Psi_{m,B,\delta}(u)\right|\le C_0\delta A_{m,B}\le
\varepsilon. 
\end{equation*}
Moreover\footnote{We notice that the width can be traced back to the proof of~\cite[Proposition 7.4]{petersen2024mathematical}.}, $\operatorname{depth}(\Psi_{m,B,\delta}) \le
C_0\log_2(m/\delta)$, $\operatorname{width}(\Psi_{m,B,\delta})
\le C_0 m$ and $\operatorname{size}(\Psi_{m,B,\delta})\le
C_0m\log_2(m/\delta)$. We set $\Phi^{\mathrm{Herm}}_{m,B,\varepsilon}(x)\eqdef \Psi_{m,B,\delta}(x/B)$. Since $x\in[-B,B]$ implies $x/B\in[-1,1]$, we obtain
\begin{equation*}
    \sup_{x\in[-B,B]} \left|
h_m(x)
-
\Phi^{\mathrm{Herm}}_{m,B,\varepsilon}(x)
\right|
\le
\varepsilon.
\end{equation*}
Importantly, the affine map $x\mapsto x/B$ can be absorbed into the first layer of the network. Whence, it does not change the depth, width, or number of non-zero parameters, up to the universal constant. Finally, since $\delta
=\min\left\{1,\frac{\varepsilon}{C_0A_{m,B}}\right\}$, and since $A_{m,B}\ge 1$, $m\ge2$, and $0<\varepsilon<1$, we have, after increasing $C$, $\log_2(m/\delta)\le
C\left(1+\log_2\left(\frac{mA_{m,B}}{\varepsilon}\right)
\right)$. This gives the stated depth and size bounds. The width bound is $C m$. This concludes the proof. 
\end{proof}
In order to obtain the multivariate tensorized version, we recall the following lemma of~\cite{yarotsky2018optimal} as formulated in~\cite{petersen2024mathematical}; with a small renormalization which we explain in the following now.  We also tweak it using the trick in~\cite[Corollary 5.2]{hong2024bridging} for the map to be compactly supported.  We will use the following variant of that ``localization trick''.
\begin{lemma}[Localization trick]\label{lem:main_localization_trick}
Let $q_{\mathrm{in}},q_{\mathrm{out}}\in\mathbb N_{+}$, $\beta>0$, and $f\,:\,\mathbb{R}^{q_{\mathrm{in}}}\to\mathbb{R}^{q_{\mathrm{out}}}$ be realized by a $\operatorname{ReLU-MLP}$ $\Phi$ with $\operatorname{depth}(\Phi) \leq L$, $\operatorname{width}(\Phi) \leq W$, and $\operatorname{size}(\Phi) \leq S$. Let $S_{\rm in}(\Phi) \leq q_{\mathrm{in}} B$ denote the number of non-zero weights in the first weight matrix of $\Phi$. Then, there exists a $\operatorname{ReLU-MLP}$ $\bar{\Phi}_{\beta}$ realizing a map $\bar{f}\,:\,\mathbb{R}^{q_{\mathrm{in}}}\to\mathbb{R}^{q_{\mathrm{out}}}$ satisfying:
\begin{itemize}
    \item \textbf{\emph{Extension:}} $\bar{f}(x)=f(x)$ for every $x \in [-1,1]^{q_{\mathrm{in}}}$;
    \item \textbf{\emph{Cutoff:}} $\bar{f}(x)=0$ for every $x \in \mathbb{R}^{q_{\mathrm{in}}}$ with $\|x\|_{\infty} \geq 1+\frac{1}{\beta}$.
\end{itemize}
Moreover, with $\ell_{q_{\mathrm{in}}}\eqdef\left\lceil
\log_2(q_{\mathrm{in}}\vee2)\right\rceil$, the network $\bar{\Phi}_{\beta}$ may be chosen so that $\operatorname{width}(\bar{\Phi}_{\beta}) \leq W + D q_{\rm in } + 2 q_{\rm out} + 2$, $\operatorname{depth}(\bar{\Phi}_{\beta}) \leq \max\{L + 2,\ell_{q_{\mathrm{in}}}+2\}+1$, and $\operatorname{size}(\bar{\Phi}_{\beta}) \leq S + C (S_{\rm in}(\Phi) + q_{\rm in } + q_{\rm out} \ell_{q_{\rm in }} + L + q_{\rm out})$, where $C$ is a strictly positive universal constant. 
\end{lemma}
\begin{proof}
First, we construct the one-dimensional cutoff
map $\mathcal{L}:\mathbb{R}\to\mathbb{R}$, and then parallelize it coordinate-wise.
For later use in the lemma, we will specialize the parameter $\eta$ in
\eqref{eq:pwlin_guy} to $\eta=1$.
\begin{equation}
\label{eq:pwlin_guy}
\begin{aligned}
\mathcal{L}(x)
&=
\begin{cases}
0, & x\le -\eta-\eta/\beta,\\
-\eta-\beta(x+\eta), & -\eta-\eta/\beta \le x\le -\eta,\\
x, & -\eta\le x\le \eta,\\
\eta-\beta(x-\eta), & \eta\le x\le \eta+\eta/\beta,\\
0, & x\ge \eta+\eta/\beta.
\end{cases}
\end{aligned}
\qquad
\vcenter{\hbox{
\begin{tikzpicture}[scale=0.85]
    \pgfmathsetmacro{\etaV}{1.2}
    \pgfmathsetmacro{\betaV}{4}
    \pgfmathsetmacro{\cutV}{\etaV + \etaV/\betaV}

    \draw[->] (-\cutV-0.55,0) -- (\cutV+0.65,0) node[right] {$x$};
    \draw[->] (0,-\etaV-0.55) -- (0,\etaV+0.65) node[above] {$\mathcal{L}(x)$};

    \draw[dashed,black,opacity=0.5]
        (-\cutV-0.25,-\cutV-0.25) -- (\cutV+0.25,\cutV+0.25);

    \draw[dashed,gray] (-\etaV,0) -- (-\etaV,-\etaV);
    \draw[dashed,gray] (\etaV,0) -- (\etaV,\etaV);

    \draw[very thick,blue]
        (-\cutV-0.35,0)
        -- (-\cutV,0)
        -- (-\etaV,-\etaV)
        -- (\etaV,\etaV)
        -- (\cutV,0)
        -- (\cutV+0.35,0);

    \draw (-\cutV,0.05) -- (-\cutV,-0.05)
        node[above=3pt] {{\fontsize{5}{6}\selectfont $-\big(\eta+\tfrac{1}{\beta}\big)$}};
    \draw (-\etaV,0.05) -- (-\etaV,-0.05);
    \draw (\etaV,0.05) -- (\etaV,-0.05);
    \draw (\cutV,0.05) -- (\cutV,-0.05)
        node[below=3pt] {{\fontsize{5}{6}\selectfont $\eta+\tfrac{1}{\beta}$}};

    \draw (0.05,-\etaV) -- (-0.05,-\etaV)
        node[left=3pt] {{\fontsize{5}{6}\selectfont $-\eta$}};
    \draw (0.05,\etaV) -- (-0.05,\etaV)
        node[left=3pt] {{\fontsize{5}{6}\selectfont $\eta$}};
\end{tikzpicture}
}}
\end{equation}
We compute as a two-layer $\operatorname{ReLU}$-MLP as $\mathcal{L}=\mathcal{L}_2\circ\mathcal{L}_1$, where
\begin{equation}
\label{eq:MLP}
\begin{aligned}
    \mathcal{L}_1(x)
 & \eqdef
    \left(
        \begin{pmatrix}
            1 & -1  & 0 & 0\\
            0 & 0 & -\beta & 0 \\
            0 & 0 & 0 & -\beta
        \end{pmatrix}
        \operatorname{ReLU}
        \bullet
            \left(
                \begin{pmatrix}
                    x + 1 \\
                    x - 1 \\
                    x - 1 \\
                    -x - 1
                \end{pmatrix}
            \right)
    \right)
    +
        \begin{pmatrix}
            -1 \\
            0 \\
            0
        \end{pmatrix}
\\
    \mathcal{L}_2(u_1,u_2,u_3)
 & \eqdef
    (1,-1)^{\top}
        \operatorname{ReLU}
        \bullet
        \left(
            \begin{pmatrix}
                1 & 1 & 0 \\
                -1 & 0 & 1
            \end{pmatrix}
            \begin{pmatrix}
                u_1\\
                u_2 \\
                u_3
            \end{pmatrix}
        \right)
.
\end{aligned}
\end{equation}
Thus $\mathcal{L}(x)=x$ on $[-1,1]$, $\mathcal{L}(x)=0$ on
$(-\infty,-1-1/\beta]\cup[1+1/\beta,\infty)$, and
$\mathcal{L}(\mathbb{R})\subseteq[-1,1]$.
Observe that $\mathcal{L}_1$ has width $4$, depth $1$, and $13$ non-zero parameters; while $\mathcal{L}_2$ has width $3$, depth $1$, and $9$ non-zero parameters. Together, $\mathcal{L}$ thus has width $4$, depth $2$, and can be expressed using $22$ non-zero parameters

Now, we parallelize this construction over the $q_{\rm in }$ coordinates; define
$\mathcal{L}^{\uparrow}:\mathbb{R}^{q_{\rm in }}\to\mathbb{R}^{q_{\rm in }}$ by
\[
        \mathcal{L}^{\uparrow}(x)
    \eqdef
        \bigl(
            \mathcal{L}(x_1),\dots,\mathcal{L}(x_{q_{\rm in }})
        \bigr).
\]
Then $\mathcal{L}^{\uparrow}(x)=x$ for every $x\in[-1,1]^{q_{\rm in }}$, and
$\mathcal{L}^{\uparrow}(\mathbb{R}^{q_{\rm in }})\subseteq[-1,1]^{q_{\rm in }}$.  In particular,
$\mathcal{L}^{\uparrow}(x)=x$ for every $x\in[0,1]^{q_{\rm in }}$.
Observe also that $\mathcal{L}^{\uparrow}$ has width $4 q_{\rm in }$, depth $2$, and can be expressed using at most $22 q_{\rm in }$ non-zero parameters; it implements $\mathcal{L}$ in~\eqref{eq:pwlin_guy} component-wise.

By setting $
    g
    \eqdef
    f\circ \mathcal{L}^{\uparrow}
$ and using that $\mathcal{L}^{\uparrow}(\mathbb{R}^{q_{\rm in }})\subseteq[-1,1]^{q_{\rm in }}$, and that $f$ is continuous,
the constant
$
    C_{f}
\eqdef
    \max\big\{
    1,
    \sup_{z\in[-1,1]^d}\|f(z)\|_{\infty}
    \big\}
$ is finite. The map $g$ agrees with $f$ on $[0,1]^{q_{\rm in }}$, but it need not vanish outside the localization region.  We therefore add one scalar cutoff gate.
Define
$
    r(x)
\eqdef
    \max_{i\in [q_{\rm in }]_{+}}
    \,\{
        \operatorname{ReLU}(x_i-1),\,\operatorname{ReLU}(-x_i-1)
    \}
$.  
Then $r(x)=0$ for every $x\in[-1,1]^{q_{\rm in }}$, while $r(x)\ge 1/\beta$ whenever
$\|x\|_{\infty}\ge 1+1/\beta$.  Set
$
    q(x)
\eqdef
    2 C_{f}\beta\, r(x)
$.  
Finally, we define the components of $\bar{f}$, for $j\in[q_{\rm out}]_{+}$, by
\begin{equation}
\label{eq:def_bar_f}
    \bar f_j(x)
\eqdef
    \frac{1}{2}\,
    \operatorname{ReLU}\bigl(g_j(x)+C_{f}-q(x)\bigr)
    -
    \frac{1}{2}\,
    \operatorname{ReLU}\bigl(-g_j(x)+C_{f}-q(x)\bigr)
.
\end{equation}
We now verify the two desired properties.  If $x\in[0,1]^{q_{\rm in}}$, then
$\mathcal{L}^{\uparrow}(x)=x$, and hence $g(x)=f(x)$.  Moreover, $r(x)=0$, so
$q(x)=0$.  Since $|g_j(x)|=|f_j(x)|\le C_{f}$, we have
$
    \bar f_j(x)
=
    \frac{1}{2}\bigl(g_j(x)+C_{f}\bigr)
    -
    \frac{1}{2}\bigl(-g_j(x)+C_{f}\bigr)
=
    g_j(x)
=
    f_j(x)
$.
Thus $\bar f(x)=f(x)$ on $[0,1]^{q_{\rm in}}$.
Conversely, suppose that $\|x\|_{\infty}\ge 1+1/\beta$.  Then $r(x)\ge 1/\beta$,
and therefore $q(x)\ge 2C_{f}$.  Since
$\mathcal{L}^{\uparrow}(x)\in[-1,1]^{q_{\rm in}}$, the definition of $C_{f}$ gives
$|g_j(x)|\le C_{f}$ for every $j\in[q_{\rm out}]_{+}$.  Hence
$
    g_j(x)+C_{f}-q(x)\le 0
$ and 
    $-g_j(x)+C_{f}-q(x)\le 0
$.  
Both ReLU terms in the definition of $\bar f_j(x)$ therefore vanish, and so
$\bar f_j(x)=0$ for every $j\in[q_{\rm out}]_{+}$.  Thus $\bar f(x)=0$ whenever
$\|x\|_{\infty}\ge 1+1/\beta$.

It remains to finish recording the network complexity.  Composing the first affine layer of
$\Phi$ with the output of $\mathcal{L}^{\uparrow}$ costs at most an additional
$S_{\operatorname{in}}(\Phi)$ non-zero parameters, since each non-zero
input weight of $\Phi$ is applied to a coordinate produced by the localization block.
Consequently, the branch computing
$
    g=f\circ\mathcal{L}^{\uparrow}
$
satisfies $\operatorname{depth}(g) \leq L+2$, $\operatorname{width}(g) \leq W+4q_{\rm in}$, and $\operatorname{size}(g) \leq S + S_{\operatorname{in}}(\Phi) + C \,q_{\rm in}$.
The scalar branch computing $r$ uses the standard ReLU maximum identity
$
    \max\{a,b\}
=
    b+\operatorname{ReLU}(a-b)
$, 
iterated as a binary tree over the $2 q_{\rm in}$ non-negative quantities
$\operatorname{ReLU}(x_i-1)$ and $\operatorname{ReLU}(-x_i-1)$, $i\in[q_{\rm in}]_{+}$.  This gives depth at most
$\ell_{q_{\rm in}}+2$ and size $O(q_{\rm in})$, after increasing the universal constant if necessary.
The affine rescaling $r\mapsto q=2C_{f}\beta r$ costs one additional non-zero parameter.
Let $L_g$ and $L_q$ denote the depths of the branches realizing $g$ and $q$, respectively.  We choose the above realizations so that
$L_g=L+2$ and $L_q\le \ell_{q_{\rm in}}+2$.  Set
$
    m\eqdef \max\{L+2,\ell_{q_{\rm in}}+2\}
$, $n_g\eqdef m-L_g$, and $n_q\eqdef m-L_q$.  
Then $n_g,n_q\in\mathbb{N}_{+}$, $n_g\le \ell_{q_{\rm in}}$, and $n_q\le L+2$.

We now equalize the depths of the two branches.  For the $g$-branch, use the exact ReLU identity
\[
    \operatorname{Id}_{\mathbb{R}^q_{\rm in}}(z)
    =
    \begin{pmatrix}
        I_{q_{\rm out}} & -I_{q_{\rm out}}
    \end{pmatrix}
    \operatorname{ReLU}
    \bullet
    \begin{pmatrix}
        I_{q_{\rm out}}\\
        -I_{q_{\rm out}}
    \end{pmatrix}
    z,
    \qquad z\in\mathbb{R}^{_{q_{\rm out}}}.
\]
Each such identity block has width $2 q_{\rm out}$ and exactly $4 q_{\rm out}$ non-zero weights.  Composing the $g$-branch with $n_g$ such blocks preserves the realized map exactly and increases its depth to $m$.  The resulting additional size is at most
$
    4 q_{\rm out} n_g
\le
    4 q_{\rm out}\ell_d
$.

For the $q$-branch, observe that $q(x)\ge0$ for every $x\in\mathbb{R}^{q_{\rm in}}$.  Hence, on the range of the $q$-branch, the scalar identity is realized by
$
    \operatorname{Id}_{[0,\infty)}(s)
    =
    \operatorname{ReLU}(s)
$ for $s\ge0$.
Each such padding block has width $1$ and one non-zero weight.  Composing the $q$-branch with $n_q$ such blocks preserves the realized map exactly and increases its depth to $m$.  The additional size is at most
$
    n_q
\le
L+2
$.
Thus, after padding, the two branches realize the same functions $g$ and $q$, respectively, and both outputs are available at depth $m$.  The padding increases the width by at most $2 q_{\rm out} +1$ and the size by at most
$
    4 q_{\rm out} \ell_{q_{\rm in}} + L + 2
\le
    C(q_{\rm out} \ell_{q_{\rm in}} + L)
$
for a universal strictly positive constant $C>0$.

It remains to implement the terminal gate.  Let $(z,s)\in\mathbb{R}^{q_{\rm out}}\times[0,\infty)$ represent the output of the padded branches, where $z=g(x)$ and $s=q(x)$.  Define the affine map
$
    A:\mathbb{R}^{q_{\rm out}+1}\to\mathbb{R}^{2 q_{\rm out}}
$
by
$
    (A(z,s))_{2j-1}=z_j+C_{f}-s,
$ and $
    (A(z,s))_{2j}=-z_j+C_{f}-s
$ for each $j\in[q_{\rm out}]_+$.
and all $2 q_{\rm out}$ biases are equal to $C_{f}$.  Therefore $A$ has $4 q_{\rm out}$ non-zero weights and $2 q_{\rm out}$ non-zero biases.  Let
$
    R(z,s)\eqdef \operatorname{ReLU}\bullet A(z,s)\in\mathbb{R}^{2 q_{\rm out}}.
$
Finally, define the linear map $T:\mathbb{R}^{2 q_{\rm out}}\to\mathbb{R}^{q_{\rm out}}$ by
$
    (Ty)_j
\eqdef
    \frac{1}{2}y_{2j-1}-\frac{1}{2}y_{2j}
$ for each $j\in[q_{\rm out}]_+$.
The map $T$ has exactly $2 q_{\rm out}$ non-zero weights.  Hence $T\circ R$ is realized by one ReLU layer of width $2 q_{\rm out}$ and has at most $
    4 q_{\rm out}+2q_{\rm out}+2q_{\rm out}=8q_{\rm out}
$
non-zero parameters.  Moreover, for every $x\in\mathbb{R}^{q_{\rm in}}$ and every $j\in[q_{\rm out}]_{+}$,
\[
    (T\circ R)_j(g(x),q(x))
    =
    \frac{1}{2}
    \operatorname{ReLU}\bigl(g_j(x)+C_{f}-q(x)\bigr)
    -
    \frac{1}{2}
    \operatorname{ReLU}\bigl(-g_j(x)+C_{f}-q(x)\bigr)
    =
    \bar f_j(x).
\]
Therefore, the terminal gate realizes $\bar f$ exactly, increases the depth by one, increases the width by at most $2 q_{\rm out}$, and increases the size by at most $8 q_{\rm out}$. Combining the padding and terminal-gate estimates, the additional size beyond the two branches is bounded by
$
    4 q_{\rm out} \ell_{q_{\rm in}}+L+2+8 q_{\rm out}
    \le
    C(q_{\rm out}\ell_{q_{\rm in}}+L+q_{\rm out}),
$
after increasing the universal constant $C>0$ if necessary.
Combining the estimates yields
$
\operatorname{depth}(\bar{\Phi}_{\beta})
    \le
    \max\{L+2,\ell_{q_{\rm in}}+2\}+1
$, $
\operatorname{width}(\bar{\Phi}_{\beta})
    \le
    W+4 q_{\rm in}+2q_{\rm out}+2
$, 
and
$
    \operatorname{size}(\bar{\Phi}_{\beta})
    \le
    S
    +
    S_{\operatorname{in}}(\Phi)
    +
    C\bigl(q_{\rm in}+q_{\rm out}\ell_{q_{\rm in}}+L+q_{\rm out}\bigr)
$, 
after increasing $C>0$ if necessary.  This concludes the proof.
\end{proof}
Our first usage usage of the previous lemma is in the following corollary.

\begin{corollary}[Coordinate-wise ReLU localization]
\label{cor:coordinate-wise-localization}
Let $q \in \mathbb{N}_{+}$, $R>0$, and $\beta > \frac{1}{R}$. Then, there exists a $\operatorname{ReLU-MLP}$ $P^{(q)}_{R,\beta}\,:\,\mathbb{R}^{q} \to [-R,R]^{q}$ with depth $1$, width $4 q$, and with no-more than $12 q$ non-zero parameters such that for every $x \in [-R+\frac{1}{\beta},R-\frac{1}{\beta}]^{q}$ we have $P^{(q)}_{R,\beta}(x)=x$. Moreover, for every $i \in [q]_{+}$, if $x_i \notin (-R,R)$, then $(P^{(q)}_{R,\beta}(x))_{i}=0$. 
\end{corollary}
\begin{proof}
It is sufficient to apply Lemma \ref{prop:Hermite_poly_approx_on_compact} only in the scalar identity case; i.e., by localizing the identity map on the interval $[-R,R]$ with transition width $1/\beta$, and then applying the resulting scalar map coordinate-wise. The asserted complexity bounds follow directly from Lemma \ref{lem:main_localization_trick}.
\end{proof}

\begin{lemma}[Efficient localized approximate multiplication]
\label{lem:localized-product}
There exists a universal constant $C>0$, such that for every $q_{\times}\in\mathbb N_+ \cap [2,\infty)$, $R_{\times}>0$ with $\beta_{\times}>\frac1{R_{\times}}$, and $\varepsilon_{\times}>0$ there exists a $\operatorname{ReLU-MLP}$  $\Phi^{\times,R_{\times},\beta_{\times}}_{q_{\times},\varepsilon_{\times}}:
\mathbb R^{q_{\times}}\to\mathbb R$ such that
\begin{equation*}
 \sup_{
x\in
\left[
-R_{\times}+\frac1{\beta_{\times}},
R_{\times}-\frac1{\beta_{\times}}
\right]^{q_{\times}}
}
\left|
\prod_{i=1}^{q_{\times}}x_i
-
\Phi^{\times,R_{\times},\beta_{\times}}_{q_{\times},\varepsilon_{\times}}(\mathbf{x})
\right|
\le
\varepsilon_{\times}   
\end{equation*}
Moreover, $\Phi^{\times,R_{\times},\beta_{\times}}_{q_{\times},\varepsilon_{\times}}
(\mathbb R^{q_{\times}}) \subseteq [-R_{\times}^{q_{\times}},R_{\times}^{q_{\times}}]$, and $\sup_{
x\in\mathbb R^{q_{\times}}
\setminus
(-R_{\times},R_{\times})^{q_{\times}}
}
\left|
\Phi^{\times,R_{\times},\beta_{\times}}_{q_{\times},\varepsilon_{\times}}(\mathbf{x})
\right|
\le
\varepsilon_{\times}.
$ Finally, the network may be chosen with $\operatorname{depth}(\Phi^{\times,R_{\times},\beta_{\times}}_{q_{\times},\varepsilon_{\times}}) \leq C \log(q_{\times})\left(
1+ \Big|\log_2\left(\frac{\varepsilon_{\times}}{q_{\times} R_{\times}^{q_{\times}}}\right)\Big|
\right)$, $\operatorname{width}(\Phi^{\times,R_{\times},\beta_{\times}}_{q_{\times},\varepsilon_{\times}}) \leq C q_{\times}$, and $\operatorname{size}(\Phi^{\times,R_{\times},\beta_{\times}}_{q_{\times},\varepsilon_{\times}}) \leq C q_{\times}$. 
\end{lemma}
\begin{proof}
By~\cite[Proposition 7.4]{petersen2024mathematical} there exists a $\operatorname{ReLU-MLP}$ $\tilde{\Phi}^{\times,R_{\times},\beta_{\times}}_{q_{\times},\varepsilon_{\times}}:\mathbb{R}^{q_{\times}}\to\mathbb{R}$ supported in $[-R_{\times}+\frac{1}{\beta_{\times}}, R_{\times}-\frac{1}{\beta_{\times}}]^{q_{\times}}$ and sending $[-R_{\times},R_{\times}]^{q_{\times}}$ to $[-R_{\times}^{q_{\times}},R_{\times}^{q_{\times}}]$
such that
\begin{equation*}
        \sup_{x_j\in[-R_{\times},R_{\times}]}
    \,
    \Big|
        \prod_{j=1}^{q_{\times}} x_j
        -
        \tilde{\Phi}^{\times,R_{\times},\beta_{\times}}_{q_{\times},\varepsilon_{\times}}(x_1,\dots,x_{q_{\times}})
    \Big|
\le 
    \varepsilon_{\times}.
\end{equation*}
Moreover, we have $\operatorname{depth}(\tilde{\Phi}^{\times,R_{\times},\beta_{\times}}_{q_{\times},\varepsilon_{\times}}) \leq C \log(q_{\times})\left(
1+ \Big|\log_2\left(\frac{\varepsilon_{\times}}{q_{\times} R_{\times}^{q_{\times}}}\right)\Big|
\right)$, $\operatorname{width}(\tilde{\Phi}^{\times,R_{\times},\beta_{\times}}_{q_{\times},\varepsilon_{\times}}) \leq C q_{\times}$, and $\operatorname{size}(\tilde{\Phi}^{\times,R_{\times},\beta_{\times}}_{q_{\times},\varepsilon_{\times}}) \leq C q_{\times}\left(1+ \Big|\log_2\left(\frac{\varepsilon_{\times}}{q_{\times} R_{\times}^{q_{\times}}}\right)\Big|
\right)$. Now, let $p_{R_{\times},\beta_{\times}}:\mathbb{R}^{q_{\times}} \to \mathbb{R}^{q_{\times}}$ be the depth-$1$ $\operatorname{ReLU-MLP}$ 
Corollary~\ref{cor:coordinate-wise-localization}, and define the $\operatorname{ReLU-MLP}$ in the statement of the lemma as $\Phi^{\times,R_{\times},\beta_{\times}}_{q_{\times}}\eqdef \tilde{\Phi}^{\times,R_{\times},\beta_{\times}}_{q_{\times},\varepsilon_{\times}} \circ p_{R_{\times},\beta_{\times}}$. Then, for every $x\in [-R_{\times}-\frac{1}{\beta_{\times}}, R_{\times}+\frac{1}{\beta_{\times}}]^{q_{\times}}$, since $p_{R_{\times},\beta_{\times}}(x)=x$, we have 
\begin{equation*}
   \Big|
        \prod_{j=1}^{q_{\times}} x_j
        -
        \Phi^{\times,-R_{\times},\beta_{\times}}_{q_{\times},\varepsilon_{\times}}(\mathbf{x})
    \Big|
=
    \Big|
        \prod_{j=1}^{q_{\times}} x_j
        -
        \widetilde{\Phi}^{\times,-R_{\times},\beta_{\times}}_{q_{\times},\varepsilon_{\times}}(\big(p_{R_{\times},\beta_{\times}}(\mathbf{x})\big)
    \Big|
=
    \Big|
        \prod_{j=1}^{q_{\times}} x_j
        -
        \widetilde{\Phi}^{\times,-R_{\times},\beta_{\times}}_{q_{\times},\varepsilon_{\times}}(\mathbf{x}\big)
    \Big|\leq \varepsilon_{\times}. 
\end{equation*}

Taking the supremum over $\mathbf{x}\in [-R_{\times}-\frac{1}{\beta_{\times}}, R_{\times}+\frac{1}{\beta_{\times}}]^{q_{\times}}$ yields the desired estimate. Since $p_{R_{\times},\beta_{\times}}(\mathbb{R}^{q_{\times}})\subseteq [-R_{\times}-\frac{1}{\beta_{\times}}, R_{\times}+\frac{1}{\beta_{\times}}]^{q_{\times}}$ sends $[-R_{\times}-\frac{1}{\beta_{\times}}, R_{\times}+\frac{1}{\beta_{\times}}]^{q_{\times}}$ to $[-R_{\times}^{q_{\times}}, R_{\times}^{q_{\times}}]$, it follows that $\Phi^{\times,R_{\times},\beta_{\times}}_{q_{\times},\varepsilon_{\times}}
(\mathbb R^{q_{\times}}) \subseteq [-R_{\times},R_{\times}]^{q_{\times}}$. Now, let $\mathbf{x}\in \mathbb{R}^{q_{\times}}\setminus (-R_{\times},R_{\times})^{q_{\times}}$.  Then there exists $i\in [q_{\times}]_{+}$ such that $x_i\not\in (-R_{\times},R_{\times})$.  By Corollary~\ref{cor:coordinate-wise-localization}, $p_{R_{\times},\beta_{\times}}(x)_{i}=0$ and, therefore, $\prod_{j=1}^{q_{\times}} 
p_{R_{\times},\beta_{\times}}(x)_{j}=0$. Since $p_{R_{\times},\beta_{\times}}(\mathbf{x})\in [-R_{\times}^{q_{\times}},R_{\times}^{q_{\times}}]^{q_{\times}}$, the approximation property of $\widetilde{\Phi}^{\times,R_{\times}}_{q_{\times},\varepsilon_{\times}}$ gives
\[
\begin{aligned}
    \big|
        \Phi^{\times,R_{\times},\beta_{\times}}_{q_{\times},\varepsilon_{\times}}(\mathbf{x})
    \big|
&=
    \Big|
        \widetilde{ \Phi}^{\times,R_{\times}}_{q_{\times},\varepsilon_{\times}}\big(p_{R_{\times},\beta_{\times}}(\mathbf{x})\big)
        -
        \prod_{j=1}^{q_{\times}} p_{R_{\times},\beta_{\times}}(x)_{j}
    \Big|
\le
    \varepsilon.
\end{aligned}
\]
Finally, composition with the depth $1$ width $4 q_{\times}$ network $p_{R_{\times},\beta_{\times}}$ increases the depth by $1$, while the width remains of order $q_{\times}$ and the number of non-zero parameters increases by only $O(q_{\times})$.  After enlarging the absolute constant $C$ if necessary, the claimed complexity bounds follow.
\end{proof}
We are now ready to approximate the tensorized Hermite polynomials rather efficiently while also controlling the support and clipping the range of our approximated versions. The notation is as the beginning of Step 2, where now the input is $\boldsymbol{x}=(x_{j,r})_{j \in [J_n]_{+}, r \in [D]_{+}} \in \mathbb{R}^{J_{n} \times D}$. We recall that $Q_n \eqdef J_n D$. We have the following lemma. 
\begin{lemma}[Local-uniform ReLU approximation of tensorized Hermite polynomials]
\label{lem:tensorized_Hermite_poly_approx_on_compact}
There exists a universal constant $C>0$ such that for every $D, J_n \in \mathbb{N}_{+}$, $\boldsymbol{\alpha} \in \mathbb{R}^{J_{n} \times D}$ with $|\boldsymbol{\alpha}|=\sum_{r=1}^{D}\sum_{j=1}^{J_n}\alpha_{j,r}=n$, $B>0$, and $\varepsilon \in (0,1)$ there exists a $\operatorname{ReLU-MLP}$ $\Phi_{\boldsymbol{\alpha}, B, \varepsilon}^{\rm Herm}:\mathbb{R}^{J_{n} \times D}\to \mathbb{R}$ satisfying
\begin{equation*}
    \sup_{\boldsymbol{x}\in[-B,B]^{J_n\times D}}
\left|
h_{\boldsymbol\alpha}(\boldsymbol{x})
-
\Phi^{\rm Herm}_{\boldsymbol\alpha,B,\varepsilon}(\boldsymbol{x})
\right|
\le
\varepsilon.
\end{equation*}
Moreover, the width, depth, and number of non-zero parameters can be chosen as in (the second part of) Proposition \ref{prop:ChaosMode}
\end{lemma}
\begin{proof}
If $n=0$, then $h_{\boldsymbol\alpha}\equiv1$ and the claim is represented exactly by a constant $\operatorname{ReLU-MLP}$, and the stated bounds hold after increasing the positive constant $C>0$. We, therefore, assume that $n \geq 1$. Now, if $Q_n=1$, then the tensorized polynomial is univariate and the claim follows directly from Lemma \ref{prop:Hermite_poly_approx_on_compact}, again after
increasing the universal constant $C$. Whence, hereafter, we assume $Q_n \geq 2$; for every pair $(j,r)_{j \in [J_n]_{+}, r \in [D]_{+}}$, we construct a uni-variate $\operatorname{ReLU-MLP}$ $\Phi_{j,r}:\mathbb{R}\rightarrow\mathbb{R}$. If $\alpha_{j,r}=0$, then $h_{\alpha_{j,r}} \equiv 1$, and we let $\Phi_{j,r}\equiv 1$. Otherwise, Lemma \ref{prop:Hermite_poly_approx_on_compact} yields a $\operatorname{ReLU-MLP}$ $\Phi_{j,r}:\mathbb{R}\rightarrow\mathbb{R}$ such that $\sup_{x \in [-B,B]}|h_{\alpha_{j,r}}(x)-\Phi_{j,r}(x)| \leq \delta$, where $\delta=\frac{\varepsilon}{2 Q_n A_{\boldsymbol{\alpha},B}^{Q_n-1}}$ ($< 1$). Moreover, for every $x \in [-B,B]$, $|h_{\alpha_{j,r}}(x)| \le A_{\alpha_{j,r},B}-1 \le A_{\boldsymbol\alpha,B}-1$. Since $\delta<1$, it follows that $|\Phi_{j,r}(x)|\le A_{\boldsymbol\alpha,B}$ for every $x \in [-B,B]$. In particular, setting, $\Psi(\boldsymbol{x}) \eqdef \Phi_{j,r}(x_{j,r})_{j \in [J_n]_{+}, r \in [D]_{+}} \in \mathbb{R}^{Q_n}$, we have, as desired, that $\Psi([-B,B]^{J_n \times D}) \subseteq [-A_{\boldsymbol\alpha,B},A_{\boldsymbol\alpha,B}]^{Q_n}$. Now, apply Lemma \ref{lem:localized-product} with (the admissible choice of) $q_{\times}=Q_n, R_{\times}=2A_{\boldsymbol\alpha,B}, \beta_{\times}=A_{\boldsymbol\alpha,B}^{-1}, \varepsilon_{\times}=\frac{\varepsilon}{2}$, yields a $\operatorname{ReLU-MLP}$ $\Psi_{Q_n}^{\times}:\mathbb{R}^{Q_n}\to\mathbb{R}$ such that 
\begin{equation*}
    \sup_{\boldsymbol{y}\in[-A_{\boldsymbol\alpha,B},A_{\boldsymbol\alpha,B}]^{Q_n}}
\left|
\prod_{i=1}^{Q_n}y_i
-
\Phi^\times_{Q_n}(\boldsymbol{y})
\right|
\le
\frac{\varepsilon}{2}.
\end{equation*}
We now consider the composition $\Phi^{\rm Herm}_{\boldsymbol{\alpha}, B, \varepsilon} \eqdef \Psi_{Q_n}^{\times} \circ \Psi$, and fix $\boldsymbol{x} \in [-B,B]^{J_n \times D}$. Moreover, we choose any enumeration $\{(j_i, r_i)\,:\,i=1,\ldots,Q_n\}$ of $[J_n]_{+} \times [D]_{+}$, and we set $u_i\eqdef h_{\alpha_{j_i, r_i}}(x_{j_i, r_i})$ and $v_i \eqdef \Phi_{j_i, r_i}(x_{j_i, r_i})$, $i \in [Q_n]_{+}$; in particular, $h_{\boldsymbol{\alpha}}(\boldsymbol{x})=\prod_{i=1}^{Q_n} u_i$ and $\Phi^{\rm Herm}_{\boldsymbol{\alpha}, B, \varepsilon}(\boldsymbol{x})=\Phi^{\times}_{Q_n}(v_1,\ldots,v_{Q_n})$. Whence, 
\begin{equation*}
    |h_{\boldsymbol{\alpha}}(\boldsymbol{x})-\Phi^{\rm Herm}_{\boldsymbol{\alpha}, B, \varepsilon}(\boldsymbol{x})| \leq \left|
\prod_{i=1}^{Q_n}u_i
-
\prod_{i=1}^{Q_n}v_i
\right|
+
\left|
\prod_{i=1}^{Q_n}v_i
-
\Phi^\times_{Q_n}(v_1,\ldots,v_{Q_n})
\right| \leq \left|
\prod_{i=1}^{Q_n}u_i
-
\prod_{i=1}^{Q_n}v_i
\right| + \frac{\varepsilon}{2}.
\end{equation*}
In order to estimate the first term, we use the telescoping identity
\begin{equation*}
\prod_{i=1}^{Q_n}u_i-\prod_{i=1}^{Q_n}v_i
=
\sum_{m=1}^{Q_n}
\left(\prod_{i=1}^{m-1}v_i\right)
(u_m-v_m)
\left(\prod_{i=m+1}^{Q_n}u_i\right),\,\,\text{and}\,\,\left|
\prod_{i=1}^{Q_n}u_i-\prod_{i=1}^{Q_n}v_i
\right|
\le
Q_n\delta A_{\boldsymbol\alpha,B}^{Q_n-1}
=
\frac{\varepsilon}{2},
\end{equation*}
where we use the fact that $|u_i|\le A_{\boldsymbol\alpha,B}$, $|v_i|\le A_{\boldsymbol\alpha,B}$, and $|u_i-v_i| \leq \delta$. Combining the two estimates yields
$\left|h_{\boldsymbol\alpha}(z)-\Phi^{\rm Herm}_{\boldsymbol\alpha,B,\varepsilon}(z)
\right|\le \varepsilon$. Taking the supremum over $\boldsymbol{x} \in[-B,B]^{J_n\times D}$ proves the approximation claim. Finally, we verify the complexity bounds. For every $(j,r)$ with $\alpha_{j,r}\ge1$, Lemma~\ref{prop:Hermite_poly_approx_on_compact} and the
choice of $\delta$ give a univariate network with $\operatorname{width}(\Phi^{\rm Herm}_{\boldsymbol{\alpha}, B, \varepsilon}) \leq C\alpha_{j,r}$, $\operatorname{depth}(\Phi^{\rm Herm}_{\boldsymbol{\alpha}, B, \varepsilon}) \leq C\left[
1 \!+\!
\log_2\left(
1 \!+\!
\frac{
2Q_n\alpha_{j,r}
A_{\alpha_{j,r},B}
A_{\boldsymbol\alpha,B}^{Q_n-1}
}{
\varepsilon
}
\right)
\right]
$, and $\operatorname{size}(\Phi^{\rm Herm}_{\boldsymbol{\alpha}, B, \varepsilon}) \leq C\alpha_{j,r}
\left[
1 \!+\!
\log_2\left(
1 \!+\!
\frac{
2Q_n\alpha_{j,r}
A_{\alpha_{j,r},B}
A_{\boldsymbol\alpha,B}^{Q_n-1}
}{
\varepsilon
}
\right)
\right]
$.
Notice that for entries with $\alpha_{j,r}=0$, the constant network contributes only
bounded complexity. After parallelizing all $Q_n$ univariate networks and
padding shallower networks to a common depth, the univariate stage $\Psi$
has $\operatorname{depth}(\Psi) \leq C\left[
1+
\max_{\substack{j \in [J_n]_{+}\\ r\ \in [D]_{+}}}
\log_2\left(
1+
\frac{
2Q_n\alpha_{j,r}
A_{\alpha_{j,r},B}
A_{\boldsymbol\alpha,B}^{Q_n-1}
}{
\varepsilon
}
\right)
\right]$,
$\operatorname{width}(\Psi) \leq C(Q_n+n)$, and
$\operatorname{size}(\Psi) \leq C(Q_n+n)
\left[
1+
\max_{\substack{j \in [J_n]_{+}\\ r\ \in [D]_{+}}}
\log_2\left(
1+
\frac{
2Q_n\alpha_{j,r}
A_{\alpha_{j,r},B}
A_{\boldsymbol\alpha,B}^{Q_n-1}
}{
\varepsilon
}
\right)
\right]$.
The product stage is controlled by Lemma~\ref{lem:localized-product}. For $q_{\times}=Q_n$, $R_{\times}=2A_{\boldsymbol\alpha,B}$, and $\varepsilon_{\times}=\varepsilon/2$, we have $\frac{q_{\times}R_{\times}^{q_{\times}}}{\varepsilon_{\times}}=
\frac{
2^{Q_n+1}Q_nA_{\boldsymbol\alpha,B}^{Q_n}
}{
\varepsilon
}$. Thus, $\Phi^\times_{Q_n}$ has $\operatorname{depth}(\Phi^\times_{Q_n}) \leq C L_{Q_n}
\left[
1+
\log_2\left(
\frac{
2^{Q_n+1}Q_nA_{\boldsymbol\alpha,B}^{Q_n}
}{
\varepsilon
}
\right)
\right]
$, $\operatorname{width}(\Phi^\times_{Q_n}) \leq CQ_n$, and $\operatorname{size}(\Phi^\times_{Q_n}) \leq  CQ_n
\left[
1+
\log_2\left(
\frac{
2^{Q_n+1}Q_nA_{\boldsymbol\alpha,B}^{Q_n}
}{
\varepsilon
}
\right)
\right].
$
Since, $\Phi^{\rm Herm}_{\boldsymbol\alpha,B,\varepsilon}=\Phi^\times_{Q_n}\circ\Psi$, its
depth is bounded by the sum of the depths of the two stages, the width by
the maximum of the widths, and the number of non-zero parameters by the sum of
the numbers of non-zero parameters. After increasing the universal constant
$C$, this gives exactly the stated bounds.
\end{proof}
\subsubsection{From local-uniform to global $L^{2}(\mathcal{F}_t)$ approximation rate}
In order to complete the proof of Proposition \ref{prop:ChaosMode}, we use the following tool, which transfers the compact approximation of the previous section to an $L^{2}(\mathcal{F}_t)$-approximation when the input is the Gaussian vector of Brownian Haar coordinates.  Admittedly, the argument is standard, in the sense that we first approximate on a large cube, then we clip the output of the approximating network, and, finally, we use Gaussian tail estimate on the complement of the cube. Importantly, we show that the local Hermite approximating network has, for fixed compact radius $B>0$, multi-index $\boldsymbol{\alpha}$, and input dimension $Q_n=J_n D$, depth and a number of non-zero parameters growing at most logarithmically in the reciprocal of the approximation error, whereas its width is independent. We now state and prove the following lemma.  
\begin{lemma}[Clipped $\operatorname{ReLU-MLP}$]
\label{lem:clipped-local-hermite}
Let $n, D, J_n \in \mathbb{N}_{+}$, and $\boldsymbol{\alpha} \in \mathbb{N}^{J_n \times D}$ with $|\boldsymbol{\alpha}|=\sum_{r=1}^{D}\sum_{j=1}^{J_n}\alpha_{j,r}=n$ (cf.~also Lemma \ref{lem:tensorized_Hermite_poly_approx_on_compact}). Then, for every $\varepsilon \in (0,1)$, positive constant $B>0$, there exists a $\operatorname{ReLU-MLP}$ $\Phi_{\boldsymbol{\alpha}, B, \varepsilon}:\mathbb{R}^{J_n \times D} \to \mathbb{R}$ satisfying:
\begin{itemize}
    \item \textbf{Effective uniform approximation:} $\sup_{\boldsymbol{z} \in[-B,B]^{J_n\times D}}\left| h_{\boldsymbol\alpha}(\boldsymbol{z})-\Phi_{\boldsymbol{\alpha}, B, \varepsilon}(\boldsymbol{z})\right| \le \varepsilon$.
    \item \textbf{Global growth:} $\left|\Phi_{\boldsymbol{\alpha}, B, \varepsilon}(\boldsymbol{z})\right| \le |h_{\boldsymbol\alpha}(\boldsymbol{z})|+1$, for every $\boldsymbol{z}\in\mathbb R^{J_n\times D}$.
\end{itemize}
Moreover, for fixed $B, J_n, D$ and $\boldsymbol{\alpha}$, the $\operatorname{ReLU-MLP}$ $\Phi_{\boldsymbol{\alpha}, B, \varepsilon}$ has $\operatorname{depth}(\Phi_{\boldsymbol{\alpha}, B, \varepsilon}) \in \mathcal{O}(\log_2(Q_n+1)\log(\varepsilon^{-1}))$, $\operatorname{width}(\Phi_{\boldsymbol{\alpha}, B, \varepsilon}) \in \mathcal{O}(Q_n + n)$, and $\operatorname{size}(\Phi_{\boldsymbol{\alpha}, B, \varepsilon}) \in \mathcal{O}((Q_n + n)\log_2(Q_n \vee 2)\log(\varepsilon^{-1}))$. 
\end{lemma}
\begin{proof}
Morally, the proof hinges on Lemmas \ref{lem:main_localization_trick} and \ref{lem:tensorized_Hermite_poly_approx_on_compact}. Precisely, we identify $\mathbb{R}^{J_n \times D}$ with $\mathbb{R}^{Q_n}$ and write $h_{\boldsymbol{\alpha}}(\boldsymbol{z})=\prod_{i=1}^{Q_n}h_{\alpha_i}(z_i)$. In addition, we set $B_{+} \eqdef B+1$ and $\mathcal{C}_{\boldsymbol{\alpha}, B} \eqdef 1 + \sup_{\boldsymbol{z} \in [-B_{+},B_{+}]^{Q_n}} |h_{\boldsymbol{\alpha}}(\boldsymbol{z})|$. By Lemma \ref{lem:tensorized_Hermite_poly_approx_on_compact}, for every $\delta\in(0,1)$, there exists a $\operatorname{ReLU-MLP}$, say $\Psi_{\boldsymbol{\alpha}, B_{+}, \delta}$ (with $\delta$ chosen below) such that $\sup_{\boldsymbol{z}\in[-B_+,B_+]^{Q_n}}
\left|
h_{\boldsymbol\alpha}(\boldsymbol z)
-
\Psi_{\boldsymbol\alpha,B_+,\delta}(\boldsymbol z)
\right|
\le
\delta$. For every coordinate, we define the $\operatorname{ReLU-MLP}$ -realizable hat function
\begin{equation*}
    \theta_B(t)\eqdef 1-\operatorname{ReLU}(|t|-B)
+
\operatorname{ReLU}(|t|-B-1),
\qquad
|t|=\operatorname{ReLU}(t)+\operatorname{ReLU}(-t).
\end{equation*}
In particular, we have $0\le\theta_B(t)\le1$, $\theta_B(t)=1$  for $t \in [-B,B]$, and $\theta_B(t)=0$ for $t \not\in (-B_{+},B_{+})$. As a consequence, $\prod_{i=1}^{Q_n}\theta_B(z_i)=1$ on $[-B, B]^{Q_n}$, while this product vanishes outside $(-B_+,B_+)^{J_n\times D}$. At this point, we apply Corollary~\ref{cor:coordinate-wise-localization} in dimension one with (localization) radius $\mathcal{C}_{\boldsymbol{\alpha},B}+2$ and margin $1$. This gives a $\operatorname{ReLU-MLP}$ $P^{(1)}_{B_{+},1}:\mathbb{R}\rightarrow [-(\mathcal{C}_{\boldsymbol{\alpha}, B}+2),(\mathcal{C}_{\boldsymbol{\alpha}, B}+2)]$ which is the identity on $[-(\mathcal{C}_{\boldsymbol{\alpha}, B}+1), (\mathcal{C}_{\boldsymbol{\alpha}, B}+1)]$ and vanishes outside $[-(\mathcal{C}_{\boldsymbol{\alpha}, B}+2),(\mathcal{C}_{\boldsymbol{\alpha}, B}+2)]$. Since, provided that $\delta \leq 1$, $|\Psi_{\boldsymbol\alpha,B_+,\delta}(\boldsymbol{z})|\leq \mathcal{C}_{\boldsymbol{\alpha}, B}+1$ on $[-(B_{+}+1),(B_{+}+1)]^{Q_n}$, we have that $P^{(1)}_{B_{+},1}(\Psi_{\boldsymbol\alpha,B_+,\delta}(\boldsymbol{z}))=\Psi_{\boldsymbol\alpha,B_+,\delta}(\boldsymbol{z})$ for every $\boldsymbol{z} \in [-(B_{+}+1),(B_{+}+1)]^{Q_n}$. Now, we use Lemma~\ref{lem:localized-product} with $q_{\times}=Q_n+1$, $R_{\times}=\mathcal{C}_{\boldsymbol{\alpha}, B}+3$, $\beta_{\times}=1$, and $\delta_{\times}\in (0,1)$, and apply the resulting product network to the $Q_n+1$ scalar inputs $P^{(1)}_{B_{+},1}(\Psi_{\boldsymbol\alpha,B_+,\delta}(\boldsymbol{z})), \theta_B(z_1),\ldots,\theta_B(z_{Q_n})$. By choosing the multiplication accuracy $\delta_{\times}>0$ sufficiently small, we obtain a $\operatorname{ReLU-MLP}$ $\Phi_{\boldsymbol{\alpha}, B, \varepsilon}(z)\eqdef \Phi_{Q_n+1, \delta_{\times}}^{\times, R_{\times},1}(P^{(1)}_{B_{+},1}(\Psi_{\boldsymbol\alpha,B_+,\delta}(\boldsymbol{z})), \theta_B(z_1),\ldots,\theta_B(z_{Q_n}))$. If we choose $\delta,\delta_{\times}>0$ with $\delta+\delta_{\times}\le\varepsilon$, we can obtain the local-uniform estimate in the following way. Let $\boldsymbol{z} \in [-B,B]^{Q_n}$, then $\prod_{i=1}^{Q_n}\theta_B(z_i)=1$ and $P^{(1)}_{B_{+},1}(\Psi_{\boldsymbol\alpha,B_+,\delta}(\boldsymbol{z}))=\Psi_{\boldsymbol\alpha,B_+,\delta}(\boldsymbol{z})$; whence, 
\begin{equation*}
    \left|
h_{\boldsymbol\alpha}(z)
-
\Phi_{\boldsymbol\alpha,B,\varepsilon}(\boldsymbol z)
\right| \le
\left|
h_{\boldsymbol\alpha}(\boldsymbol z)
-
\Psi_{\boldsymbol\alpha,B_+,\delta}(\boldsymbol z)
\right|
+
\delta_{\times}
\le
\delta+\delta_{\times}
\le
\varepsilon.
\end{equation*}
\noindent It remains to verify the global growth bound. If $\boldsymbol{z}\in[-B_+,B_+]^{J_n\times D}$, then $P^{(1)}_{B_{+},1} (\Psi_{\boldsymbol\alpha,B_+,\delta}(\boldsymbol{z}))=\Psi_{\boldsymbol\alpha,B_+,\delta}(\boldsymbol{z})$ and $0 \leq \prod_{i=1}^{Q_n} \theta_{B}(x_i) \leq 1$, so 
\begin{equation*}
    |\Phi_{\boldsymbol\alpha,B_+,\delta}(\boldsymbol{z})| \leq |\Psi_{\boldsymbol\alpha,B_+,\delta}(\boldsymbol z)|+\delta_{\times} \leq |h_{\boldsymbol{\alpha}}(\boldsymbol{z})| + \delta+\delta_{\times} \leq |h_{\boldsymbol{\alpha}}(\boldsymbol{z})| + 1. 
\end{equation*}
Instead, if $\boldsymbol z \notin [-B, B]^{Q_n}$, then $\theta_{Q_n}(x_i)=0$ for at least one coordinate $i$, and hence 
\begin{equation*}
    |\Phi_{\boldsymbol\alpha,B_+,\delta}(\boldsymbol{z})| \leq \delta_{\times} \leq 1 \leq |h_{\boldsymbol{\alpha}}(\boldsymbol{z})| + 1. 
\end{equation*}
This proves the  global growth bound for every $\boldsymbol{z} \in \mathbb{R}^{J_n \times D}$.\\
\noindent Finally, we need to record the complexity. To this end, we remind and write that $h_m(z)=\sum_{\ell=0}^{m} c_{m,\ell} x^{\ell}$, $A_{m,B}=\sum_{\ell=0}^{m}|c_{m,\ell}|B^{\ell}$, and $A_{\boldsymbol{\alpha},B} = 1 + \max_{j\in [J_n]_{+}, r \in [S]_{+}} A_{\alpha_{j,r},B}$. Moreover, define $\ell_{Q_n} \eqdef \left\lceil
\log_2(Q_n\vee2)\right\rceil$, 
\begin{equation*}
L_{\boldsymbol\alpha,B,\varepsilon_{\mathrm{loc}}}
\eqdef
1+
\max_{\substack{j \in [J_n]_{+}, r \in [D]_{+}}}
\log_2\left(
1+
\frac{
4 Q_n\alpha_{j,r}
A_{\alpha_{j,r},B_+}
A_{\boldsymbol\alpha,B_+}^{Q_n-1}
}{
\varepsilon
}
\right)    +
L_{Q_n}
\left(
1+
\log_2\left(
\frac{
2^{Q_n+2}Q_n
A_{\boldsymbol\alpha,B_+}^{Q_n}
}{
\varepsilon
}
\right)
\right)
\end{equation*}
and
\begin{equation*}
    M_{\boldsymbol\alpha,B,\varepsilon}
\eqdef
1+
\log_2\left(
\frac{
4(Q_n+1)
(A_{\boldsymbol\alpha,B_+}+3)^{Q_n+1}
}{
\varepsilon_{\mathrm{loc}}
}
\right)
\end{equation*}
The local Hermite approximant $\Phi_{\boldsymbol\alpha,B_+,\delta}(\cdot)$ has the complexity given by Lemma \ref{lem:tensorized_Hermite_poly_approx_on_compact} evaluated at radius $B_+$ and accuracy $\delta$. The coordinate masks
$\theta_B(z_i)$ add $O(Q_n)$ neurons and non-zero parameters. The scalar
localization map $P_{B_+,1}^{(1)}$ has $O(1)$ complexity. Therefore, the final product network has $Q_n+1$ inputs and contributes depth of order $L_{Q_n+1}M_{\boldsymbol\alpha,B,\varepsilon}$, width of order $Q_n+1$, and size of order $(Q_n+1)M_{\boldsymbol\alpha,B,\varepsilon}$. Combining these costs and increasing the universal constant $C$ gives the stated complexity in the statement of the lemma. This completes the proof. 
\end{proof}
We are now ready to prove Proposition~\ref{prop:ChaosMode}. 
\begin{proof}[Proof of Proposition~\ref{prop:ChaosMode}]
First, for $j \in [J_n]_{+}$ and $r \in [D]_{+}$, we observe $Z_{j,r} = Z_{i_j,k_j}^{r}=\int_0^T \psi_{i_j,k_j}(s)\,\ud W_s^{r}$. In particular, 
\begin{equation*}
\mathbb E[Z_{j,r}Z_{j',r'}]=\mathbf 1_{\{r=r'\}}\int_0^T
\psi_{i_j,k_j}(s)\psi_{i_{j'},k_{j'}}(s)\,\ud s = \mathbf 1_{\{r=r'\}}\mathbf 1_{\{j=j'\}};
\end{equation*}
whence $\boldsymbol{Z}=(Z_{j,r})_{j \in [J_n]_{+},r \in [D]_{+}}$ is a standard Gaussian vector in  $\mathbb{R}^{J_n \times D}$. In particular, $h_{\boldsymbol\alpha}(\boldsymbol Z)\in L^2(\Omega)$ because $h_{\boldsymbol\alpha}$ is a polynomial. Now, as in the proof of Lemma \ref{lem:tensorized_Hermite_poly_approx_on_compact}, we can consider the case $n \geq 1$. Fix $\varepsilon>0$. Since $1+|h_{\boldsymbol\alpha}(\boldsymbol Z)|\in L^2(\Omega)$, there exists $B_{\boldsymbol{\alpha}, \varepsilon}>0$ such that 
\begin{equation*}
\left\|
\mathbf 1_{\{\|\boldsymbol Z\|_\infty>B_{\boldsymbol\alpha,\varepsilon}\}}
h_{\boldsymbol\alpha}(\boldsymbol{Z})\right\|_{L^2(\Omega)}\le \frac{\varepsilon}{4},\quad\text{and}\quad  \left\|
\mathbf 1_{\{\|\boldsymbol Z\|_\infty>B_{\boldsymbol\alpha,\varepsilon}\}}
\left(1+|h_{\boldsymbol\alpha}(\boldsymbol Z)|\right)
\right\|_{L^2(\Omega)}
\le
\frac{\varepsilon}{4}  
\end{equation*}
Now, we set $\varepsilon_{\mathrm{loc}} \eqdef \min\{\frac{1}{2},\frac{\varepsilon}{2}\}$. By Lemma \ref{lem:clipped-local-hermite} applied with $B=B_{\boldsymbol\alpha,\varepsilon}$ there exists a $\operatorname{ReLU-MLP}$ $\Phi_{\boldsymbol\alpha,\varepsilon}:\mathbb R^{J_n\times D}\to\mathbb R$ such that
\begin{equation*}
\begin{aligned}
    & \sup_{\boldsymbol z\in[-B_{\boldsymbol\alpha,\rho},
B_{\boldsymbol\alpha,\varepsilon}]^{J_n\times D}}
\left|
h_{\boldsymbol\alpha}(\boldsymbol z)
-
\Phi_{\boldsymbol\alpha,\varepsilon}(\boldsymbol{z})
\right|
\le
\varepsilon_{\mathrm{loc}}, \quad \text{and} \\
& |\Phi_{\boldsymbol\alpha,\varepsilon}(\boldsymbol{z})|
\le
1+|h_{\boldsymbol\alpha}(\boldsymbol z)|,
\quad \text{ for all }
\boldsymbol z\in\mathbb R^{J_n\times D}
\end{aligned}
\end{equation*}
By setting $B_\varepsilon \eqdef B_{\boldsymbol\alpha,\varepsilon}$, we use that $|\Phi_{\boldsymbol\alpha,\varepsilon}(\boldsymbol{z})|
\le 1+|h_{\boldsymbol\alpha}(\boldsymbol Z)|$ to obtain
\begin{equation*}
\begin{aligned}
\left\|
h_{\boldsymbol\alpha}(\boldsymbol Z)
-
\Phi_{\boldsymbol\alpha,\varepsilon}(\boldsymbol{Z})
\right\|_{L^2(\Omega)} &\le
\left\|
\mathbf 1_{\{\|\boldsymbol Z\|_\infty\le B_\varepsilon\}}
\left(
h_{\boldsymbol\alpha}(\boldsymbol Z)
-
\Phi_{\boldsymbol\alpha,\varepsilon}(\boldsymbol{Z})
\right)
\right\|_{L^2(\Omega)}
\\
&\quad+
\left\|
\mathbf 1_{\{\|\boldsymbol Z\|_\infty>B_\varepsilon\}}
h_{\boldsymbol\alpha}(\boldsymbol Z)
\right\|_{L^2(\Omega)}
\\
&\quad+
\left\|
\mathbf 1_{\{\|\boldsymbol Z\|_\infty>B_\varepsilon\}}
\Phi_{\boldsymbol\alpha,\varepsilon}(\boldsymbol{Z})
\right\|_{L^2(\Omega)} \\
& \leq \varepsilon_{\mathrm{loc}} + \frac{\varepsilon}{4} + \left\|
\mathbf 1_{\{\|\boldsymbol Z\|_\infty>B_\varepsilon\}} (1+|h_{\boldsymbol\alpha}(\boldsymbol Z)|) \right\|_{L^2(\Omega)} \\
& \leq \frac{\varepsilon}{2} + \frac{\varepsilon}{4} + \frac{\varepsilon}{4} = \varepsilon.
\end{aligned}
\end{equation*}
Finally, since every coordinate of $\boldsymbol{Z}$ is \(\mathcal F_i\)-measurable and $\Phi_{\boldsymbol\alpha,\varepsilon}(\cdot)$ is deterministic and Borel measurable, we have $\Phi_{\boldsymbol\alpha,\varepsilon}(\boldsymbol{Z})\in \mathcal F_i$. Moreover, $|\Phi_{\boldsymbol\alpha,\varepsilon}(\boldsymbol{Z})|
\le 1+|h_{\boldsymbol\alpha}(\boldsymbol Z)|\in L^2(\Omega)$ and so $\Phi_{\boldsymbol\alpha,\varepsilon}(\boldsymbol{Z})\in L^2(\mathcal F_t)$ if $\boldsymbol{Z}$ is $\mathcal{F}_t$-measurable . The complexity is the one supplied by Lemma~\ref{lem:clipped-local-hermite}, evaluated at
$B=B_{\boldsymbol\alpha,\varepsilon}$ and  $\varepsilon_{\mathrm{loc}}=
\min\left\{\frac12,\frac{\varepsilon}{2}\right\}$. This concludes the proof. 
\end{proof}

\subsection{\textbf{Step 3 -- Time masking}}
\label{subsubsec:time-masking}
This subsection approximates the elementary chaoslet in Eq.~\eqref{eq:reshuffled} by using $\operatorname{ReLU-MLPs}$ of $O(1)$ complexity in the time variable $t$ and the same complexity as with the Wiener chaos variable $\omega$. This is accomplished via the subsequent proposition, which is stated for scalar processes only for notational simplicity; the same statement holds component-wise for $\mathbb{R}^{d}$-valued predictable processes in $\mathcal{S}^2_T(\mathbb{R}^{d})$ with the Euclidean norm. We notice that the current Step 3 uses $(p,q)$ to denote the time-Haar index and $(i,k)$ for the stochastic Brownian Haar coordinates $Z^r_{i,k}$.
\begin{proposition}[Predictable ReLU time masking]
\label{prop:predictable-relu-time-mask}
Let $p \in \mathbb{N}$, $q \in \{0,\ldots,2^{p}-1\}$, $0<\eta<2^{-(p+2)}$, and consider, for $0 \leq t \leq T$, the following normalized Haar (wavelet) system
\begin{equation}
\label{eq:basic_Haar}
    \psi_{p,q}(t) \eqdef 
 \tfrac{2^{p/2}}{\sqrt{T}}\left(
        -\,\mathbf{1}_{[T\frac{q}{2^p},T\frac{1+2q}{2^{p+1}})}(t)
        +
            \mathbf{1}_{[T\frac{1+2q}{2^{p+1}},T\frac{q+1}{2^p})}(t)
    \right).
\end{equation}
Then, there exists a $\operatorname{ReLU-MLP}$ $\Phi_{p,q;\eta}:\mathbb{R}\rightarrow \mathbb{R}$ of depth $1$, width $6$, and with $18$ non-zero parameters such that for any $\mathcal{F}_{\cdot}$-predictable process $u_{\cdot} \in \mathcal{H}^{2}_T$ with $\|u_{\cdot}\|_{\mathcal{S}^2_T}<\infty$, the process $u^{p,q}_{\cdot} \eqdef (u_t \psi_{p,q}(t))_{0 \leq t \leq T}$ and $u_{\cdot}^{p,q;\eta} \eqdef (u_t \Phi_{p,q;\eta}(t))_{0 \leq t \leq T}$ satisfy the following:
\begin{itemize}
    \item \emph{\textbf{Predictable masking:}} $u_{\cdot}^{p,q;\eta}$ is $\mathcal{F}_{\cdot}$-predictable and $\operatorname{supp}(t \to u_t^{p,q;\eta}) \subseteq [T \frac{q}{2^{p}}, T \frac{(q+1)}{2^p}]$, $\mathbb{P}$-a.s.
    \item  \emph{$O(1)$-\textbf{approximation}}: $\|u_{\cdot}^{p,q;\eta}-u_{\cdot}^{p,q}\|_{\mathcal{H}^2_T} \leq \frac{\|u_{\cdot}\|_{\mathcal{S}^2_T} 2^{1+\frac{p}{2}}}{\sqrt{3}}\sqrt{\eta}$.
\end{itemize}
Moreover, the conclusion holds for any $u_{\cdot} \in \mathcal{H}^2_T$ of the form $u_t=\psi_{i,j}(t) U$ with $U \in L^2(\mathcal{F}_{\frac{T q}{2^{p}}})$.
\end{proposition}
In order to prove the previous proposition, we first turn our attention to the piecewise linear approximation, in $L^2([0,T])$, of the Haar wavelet $\psi_{p,q}(\cdot)$ in Eq.~\eqref{eq:basic_Haar}. Before doing so, to simplify the notation, we set $a_{p,q}\eqdef\frac{qT}{2^p}$, $c_{p,q}\eqdef\frac{(2q+1)T}{2^{p+1}}$, and $b_{p,q}\eqdef\frac{(q+1)T}{2^p}$; in particular, $\psi_{p,q}(t)
=
A_p
\left(
-\mathsf 1_{[a_{p,q},c_{p,q})}(t)
+
\mathsf 1_{[c_{p,q},b_{p,q})}(t)
\right)$ with $A_p \eqdef \frac{2^{p/2}}{T^{1/2}}$. Then, we consider our piecewise-linear ``building block" 
\begin{equation*}
\mathfrak b(t) \eqdef
\begin{cases}
0, & t<0,\\
t, & 0\le t<1,\\
1, & t\ge1.
\end{cases}    
\end{equation*}
Besides, for $0<\eta_t<2^{-(p+1)}$, we define $h_t \eqdef T \eta_t$, and consider the piecewise-affine approximation of the time-Haar atom by 
\begin{equation*}
    \psi^{(\eta_{\mathrm t})}_{p,q}(t)\eqdef A_p
\left[
-\mathfrak b\left(\frac{t-a_{p,q}}{h_{\mathrm t}}\right)
+
2\mathfrak b\left(\frac{t-(c_{p,q}-h_{\mathrm t})}{2h_{\mathrm t}}\right)
-
\mathfrak b\left(\frac{t-(b_{p,q}-h_{\mathrm t})}{h_{\mathrm t}}\right)
\right].
\end{equation*}
Now, in $L^2([0,T])$, we may use $\psi^{(\eta_{\mathrm t})}_{p,q}(t)$ to approximate $\psi_{p,q}(\cdot)$. Critically, $\eta_t$ is a free parameter which does not add additional linear pieces. Instead, it simply modulates slopes for small positive values thereof; thus, it will not increase the complexity of these functions, as we will see below.

\begin{lemma}[Piecewise-linear approximation of time-Haar atoms]
\label{lem:piecewise-affine-time-haar}
Let $p \in \mathbb{N}_{+}$ and let $q \in \{0,\ldots,2^{p}-1\}$. If $0<\eta_t<2^{-(p+2)}$, then 
\begin{equation*}
\left\|\psi_{p,q}-\psi^{(\eta_{\mathrm t})}_{p,q}
\right\|_{L^2([0,T])}=
\frac{2^{1+p/2}}{\sqrt3}
\sqrt{\eta_{\mathrm t}}
\end{equation*}
\end{lemma}
\begin{proof}
For the sake of notation, we write $a \eqdef a_{p,q}$, $c\eqdef c_{p,q}$, $b \eqdef b_{p,q}$, $\delta \eqdef \delta_{\mathrm t}$, $A\eqdef A_p$. In particular, $c-a$ coincides with $b-c$ and it is equal to $\frac{T}{2^{p+1}}$. Since $\eta_{\mathrm t}<2^{-(p+2)}$, we have that $2 \delta < c-a=b-c$. Therefore, the four intervals $[a,a+\delta]$, $[c-\delta,c]$, $[c,c+\delta]$, $[b-\delta,b]$ are pairwise disjoint (up to endpoints). By construction $\psi^{(\eta_{\mathrm t})}_{p,q} = \psi_{p,q}$ on the complement of the union of the previous intervals. Moreover, on each of the four transition intervals, the error is affine, with maximum absolute value $A$. More
precisely,
\begin{equation*}
    |\psi_{p,q}(t)-\psi^{(\eta_{\mathrm t})}_{p,q}(t)|=
    \begin{cases}
        A\frac{a+\delta-t}{\delta},
\qquad
\,\,\, t\in[a,a+\delta]\\
A\frac{t-(c-\delta)}{\delta},
\qquad
t\in[c-\delta,c]\\
A\frac{c+\delta-t}{\delta},
\qquad
\,\,\, t\in[c,c+\delta]\\
A\frac{t-(b-\delta)}{\delta},
\qquad
t\in[b-\delta,b]
    \end{cases}
\end{equation*}
Thus, each transition interval contributes $\int_0^\delta A^2\left(\frac{s}{\delta}\right)^2\,\ud s =
\frac{A^2\delta}{3}$ to the squared $L^2([0,T])$-error. Hence $\left\|
\psi_{p,q}
-
\psi^{(\eta_{\mathrm t})}_{p,q}
\right\|_{L^2([0,T])}^2
=
4\frac{A^2\delta}{3}$.
Since $A^2=\frac{2^p}{T}$, $\delta=T\eta_{\mathrm t}$, we obtain
\begin{equation*}
    \left\|
\psi_{p,q}
-
\psi^{(\eta_{\mathrm t})}_{p,q}
\right\|_{L^2([0,T])}^2
=
\frac{4}{3}2^p\eta_{\mathrm t}
=
\frac{2^{p+2}}{3}\eta_{\mathrm t}. 
\end{equation*}
Taking square roots gives the desired result. 
\end{proof}
We now return to our neural argument by first observing that $\mathfrak{b}$ can be realized as a $\operatorname{ReLU-MLP}$ of depth $1$, width $2$, and $6$ non-zero parameters
\begin{equation*}
    \mathfrak b(x)=\operatorname{ReLU}(x)-\operatorname{ReLU}(x-1).
\end{equation*}
Combining the definition of $\psi^{(\eta_{\mathrm t})}_{p,q}$ with the $\operatorname{ReLU}$ representation of $\mathfrak{b}$, we obtain
\begin{equation}\label{eq::representation}
    \begin{aligned}
\psi^{(\eta_{\mathrm t})}_{p,q}(t)
=
A_p\Bigg(
&-\operatorname{ReLU}\left(\frac{t-a_{p,q}}{\delta_{\mathrm t}}\right)
+\operatorname{ReLU}\left(\frac{t-a_{p,q}-\delta_{\mathrm t}}{\delta_{\mathrm t}}\right)
\\
&+2\operatorname{ReLU}\left(\frac{t-c_{p,q}+\delta_{\mathrm t}}{2\delta_{\mathrm t}}\right)
-2\operatorname{ReLU}\left(\frac{t-c_{p,q}-\delta_{\mathrm t}}{2\delta_{\mathrm t}}\right)
\\
&-\operatorname{ReLU}\left(\frac{t-b_{p,q}+\delta_{\mathrm t}}{\delta_{\mathrm t}}\right)
+\operatorname{ReLU}\left(\frac{t-b_{p,q}}{\delta_{\mathrm t}}\right)
\Bigg).
\end{aligned}
\end{equation}
Thus $\psi^{(\eta_{\mathrm t})}_{p,q}$ is represented exactly by a
depth-one ReLU network with width $6$ and $18$ non-zero parameters. Upon applying Lemma~\ref{lem:piecewise-affine-time-haar} we have deduced the following result.
\begin{lemma}[Arbitrarily-Accurate ReLU MLP Approximation of Haar Basis of $O(1)$ complexity] \label{lem:relu-time-haar}
Let $p \in \mathbb{N}_{+}$ and $q\in\{0,\ldots,2^p-1\}$. For every $0<\eta_{\mathrm t}<2^{-(p+2)}$ there exists a $\operatorname{ReLU-MLP}$  $\Phi_{p,q;\eta_{\mathrm t}}:\mathbb R\to\mathbb R$ of depth $1$, width $6$, and at most $18$ non-zero parameters such that
\begin{equation*}
    \left\|
\psi_{p,q}
-
\Phi_{p,q;\eta_{\mathrm t}}
\right\|_{L^2([0,T])}
=
\frac{2^{1+p/2}}{\sqrt3}
\sqrt{\eta_{\mathrm t}}
\end{equation*}
Moreover, $\operatorname{supp}(\Phi_{p,q;\eta_{\mathrm t}}) = \operatorname{supp}(\psi_{p,q}) = [a_{p,q},b_{p,q}]$.
\end{lemma}

\begin{proof}
Define $\Phi_{p,q;\eta_{\mathrm t}} \eqdef \psi^{(\eta_{\mathrm t})}_{p,q}$. The first claim follows from Lemma \ref{lem:piecewise-affine-time-haar} and the representation in Eq.~\eqref{eq::representation}; whence, it remains only to prove the support identity. If $t<a_{p,q}$, then all arguments of the clipped ramp in the definition of $\psi^{(\eta_{\mathrm t})}_{p,q}$ are negative; whence $\Phi_{p,q;\eta_{\mathrm t}}(t)=0$.  If $t\ge b_{p,q}$, then all three ramp terms are equal to $1$, and therefore $\Phi_{p,q;\eta_{\mathrm t}}(t)
= A_p(-1+2-1) = 0$. Thus, $\operatorname{supp}(\Phi_{p,q;\eta_{\mathrm t}})
\subseteq [a_{p,q},b_{p,q}]$. On the other hand, $\Phi_{p,q;\eta_{\mathrm t}}$ is non-zero on subintervals of $(a_{p,q},b_{p,q})$ arbitrarily close to every point of the support. Hence its closed support is exactly $[a_{p,q},b_{p,q}]$. The same is true for the Haar atom $\psi_{p,q}$. This concludes the proof. 
\end{proof}

We are now prepared to prove the main result of this subsection; namely, Proposition \ref{prop:predictable-relu-time-mask}. 

\begin{proof}[Proof of Proposition~\ref{prop:predictable-relu-time-mask}]
By Lemma~\ref{lem:relu-time-haar}, there is a ReLU MLP $\Phi_{p,q;\eta_{\mathrm t}}:\mathbb R\to\mathbb R$ of depth $1$, width $6$, and with \(18\) non-zero parameters such that $\operatorname{supp}(\Phi_{p,q;\eta_{\mathrm t}})=\operatorname{supp}(\psi_{p,q})=
[a_{p,q},b_{p,q}]$ and $\left\|\psi_{p,q} \!-\! \Phi_{p,q;\eta_{\mathrm t}}\right\|_{L^2([0,T])}=\frac{2^{1+p/2}}{\sqrt3}\sqrt{\eta_{\mathrm t}}$. This proves the existence of the deterministic $\operatorname{ReLU}$ time network. Now, we first prove the predictability and support statement. Since $t\mapsto \Phi_{p,q;\eta_{\mathrm t}}(t)$ is deterministic and continuous, it is Borel measurable. Therefore, if $u\in\mathbb S_T^2(\mathbb R^d)$ is $\mathcal{F}_{\cdot}$-predictable, then $u^{p,q;\eta_{\mathrm t}}_t = u_t\Phi_{p,q;\eta_{\mathrm t}}(t)$ is predictable as the product of a predictable process and a deterministic
Borel function. Moreover, because $\Phi_{p,q;\eta_{\mathrm t}}(t)=0$ for $t\notin[a_{p,q},b_{p,q}]$, $u^{p,q;\eta_{\mathrm t}}_t=0$ for every such $t$.  Therefore, $\operatorname{supp}\left(t\mapsto u^{p,q;\eta_{\mathrm t}}_t\right)
\subseteq
[a_{p,q},b_{p,q}]$ almost surely. We also have $u^{p,q;\eta_{\mathrm t}}\in \mathcal{H}_T^2$. Indeed, \(\Phi_{p,q;\eta_{\mathrm t}}\) is bounded and supported in $[a_{p,q},b_{p,q}]$, and therefore
\begin{equation*}
    \mathbb E\left[\int_0^T
|u_t\Phi_{p,q;\eta_{\mathrm t}}(t)|^2\,\ud t\right]
\le
\|\Phi_{p,q;\eta_{\mathrm t}}\|_{L^\infty([0,T])}^2
T
\mathbb E\left[
\sup_{0\le s\le T}|u_s|^2
\right]
<\infty
\end{equation*}
We next prove the approximation estimate. By the definition of the $\mathcal{H}_T^2$-norm,
\begin{equation*}
    \left\|
u^{p,q;\eta_{\mathrm t}}
-
u^{p,q}
\right\|_{\mathcal{H}_T^2}^2 =
\mathbb E\left[\int_0^T
|u_t|_{\mathbb R^d}^2
\left|
\Phi_{p,q;\eta_{\mathrm t}}(t)-\psi_{p,q}(t)
\right|^2\right]\,\ud t .
\end{equation*}
Since $u\in\mathbb S_T^2$, we obtain
\begin{equation*}
\left\|
u^{p,q;\eta_{\mathrm t}}
-
u^{p,q}
\right\|_{\mathcal{H}_T^2}^2 \leq \|u\|_{\mathbb S_T^2}^2
\left\|
\Phi_{p,q;\eta_{\mathrm t}}
-
\psi_{p,q}
\right\|_{L^2([0,T])}^2 .
\end{equation*}
By applying Lemma~\ref{lem:relu-time-haar} and taking square roots gives that the latter quantity is less than or equal to $\|u\|_{\mathbb S_T^2} \frac{2^{1+p/2}}{\sqrt{3}} \sqrt{\eta_t}$. Now, it remains to prove the final statement for a random coefficient. Let $U\in L^2(\mathcal F_{a_{p,q}})$. The process $t\mapsto U\Phi_{p,q;\eta_{\mathrm t}}(t)$ is adapted because $\Phi_{p,q;\eta_{\mathrm t}}$ vanishes before $a_{p,q}$, while for $t\ge a_{p,q}$ the random variable $U$ is $\mathcal F_t$-measurable. Since
$\Phi_{p,q;\eta_{\mathrm t}}$ is continuous, this process is predictable. Similarly, $t\mapsto U\psi_{p,q}(t)$ is predictable. Finally, by Fubini's theorem,
\begin{equation*}
    \begin{aligned}
&
\left\|
U\Phi_{p,q;\eta_{\mathrm t}}
-
U\psi_{p,q}
\right\|_{\mathcal{H}_T^2}^2=
\mathbb E\left[\int_0^T
|U|^2
\left|
\Phi_{p,q;\eta_{\mathrm t}}(t)-\psi_{p,q}(t)
\right|^2\,\ud t\right]=\mathbb E\left[|U|^2\right]
\left\|
\Phi_{p,q;\eta_{\mathrm t}}
-
\psi_{p,q}
\right\|_{L^2([0,T])}^2
\end{aligned}
\end{equation*}
Taking square roots and applying Lemma~\ref{lem:relu-time-haar} completes the proof. 
\end{proof}
\subsection{\textbf{\emph{Step 4 -- Assembly}}}
\begin{proof}[Proof of Theorem \ref{thrm:happytimes}]
We first prove the density statement in the scalar case. The vector-valued case
will then follow component-wise by stacking the corresponding scalar
$\operatorname{NeuralChaos}$ blocks. Let $\mathfrak{C}^{\rm all}$ be the family in Eq.~\eqref{eq:chaosless_disorganized_ungraded}, which is an orthonormal basis of $\mathcal{H}_T^2$. Whence, every finite linear combinations of chaoslets are dense in $\mathcal{H}_T^2$. In particular, for a fixed $\varepsilon>0$, we may choose a finite set $\overline{\mathcal{F}} \subset \mathfrak{C}^{\rm all}$ and coefficients $(\beta_m)_{m \in \overline{\mathcal{F}}}$ such that the quantity $\overline{Y}_{\cdot} \eqdef \sum_{m \in \overline{\mathcal{F}}} \beta_m m $ satisfies $\|X_{\cdot}-\overline{Y}_{\cdot}\|_{\mathcal{H}^2_T} \leq \frac{\varepsilon}{2}$. Whence, we need to prove that we can approximate $\overline{Y}_{\cdot}$ by a $\operatorname{NeuralChaos}$ model (cf.~Section \ref{sec::section_3}). To this end, we consider a non-constant scalar chaoslet $m=\mathfrak{C}^{\boldsymbol{\alpha_{(p,q)}}}$. By definition (cf.~Eq.~\eqref{eq:reshuffled}), $\mathfrak{C}_{(p,q)}^{\boldsymbol{\alpha}}
    (t,\omega)=\psi_{p,q}(t)u^{(T)}_{\boldsymbol\alpha;p,q}(\omega)$, where $u^{(T)}_{\boldsymbol\alpha;p,q}(\omega)= \prod_{(i,k)\in \mathcal{I}_{p,q}}
        \,
        \prod_{r=1}^{D}
        \,
        h_{\alpha_r^{i,k}}\big(
            Z_{i,k}^r(\omega)
        \big)$. Now, set, again, $a_{p,q}\eqdef\frac{qT}{2^p}$. By the definition of $\mathcal{I}_{p,q}$, every Brownian Haar coordinate entering the quantity $u^{(T)}_{\boldsymbol\alpha;p,q}(\omega)$ is measurable no later than $a_{p,q}$; whence $u^{(T)}_{\boldsymbol\alpha;p,q}\in L^2(\mathcal F_{a_{p,q}})$. Besides, $\|u^{(T)}_{\boldsymbol\alpha;p,q}\|_{L^2(\Omega)}=1$. Now, we observe that since $\boldsymbol \alpha$ is finitely supported, only finitely many Brownian Haar coordinates appear in $u^{(T)}_{\boldsymbol \alpha;p,q}$; we let $\Lambda_{\boldsymbol \alpha;p,q}
\eqdef
\left\{
(i,k)\in \mathcal{I}_{p,q}:
\exists r\in[D]_{+}\ \text{such that}\ \alpha^r_{i,k}>0
\right\}$ and define $\boldsymbol{Z}_{\boldsymbol\alpha;p,q} \eqdef \left(
Z^r_{i_j,k_j}\right)_{j \in J_{\boldsymbol{\alpha}}, r \in D_{+}}$, $\widetilde{\alpha}_{j,r} \eqdef \alpha^r_{i_j,k_j}$, so $u^{(T)}_{\boldsymbol \alpha;p,q}=h_{\widetilde{\boldsymbol\alpha}}(\boldsymbol{Z}_{\boldsymbol\alpha;p,q})$, where $h_{\widetilde{\boldsymbol\alpha}}(\boldsymbol{z})\eqdef \prod_{r=1}^{D}
\prod_{j=1}^{J_{\boldsymbol\alpha}} h_{\widetilde{\boldsymbol\alpha}}(z_{j,r})$. Applying Proposition~\ref{prop:ChaosMode}, for every $\rho>0$, there exists a $\operatorname{ReLU-MLP}$ $\Phi_{\boldsymbol{\alpha},\rho}:\mathbb{R}^{J_{\boldsymbol{\alpha}} \times D} \rightarrow \mathbb{R}$ such that 
\begin{equation*}
    \mathbb{E}\left[|u^{(T)}_{\boldsymbol\alpha;p,q}-\Phi_{\boldsymbol{\alpha},\rho}(\boldsymbol{Z}_{\boldsymbol\alpha;p,q})|^2\right]^{\frac{1}{2}} \leq \rho.
\end{equation*}
Next, by Lemma \ref{lem:relu-time-haar}, for every sufficiently small $\eta_{\mathrm t}>0$, there exists a $\operatorname{ReLU-MLP}$ $\Phi_{p,q;\eta_{\mathrm t}}:\mathbb{R} \rightarrow \mathbb{R}$ of depth $1$, width $6$, and at most $18$ non-zero parameters such that
\begin{equation*}
    \|\psi_{p,q}-\Phi_{p,q;\eta_{\mathrm t}}\|_{L^2([0,T])}
=
\frac{2^{1+p/2}}{\sqrt3}\sqrt{\eta_{\mathrm t}}
\end{equation*}
and $\operatorname{supp}(\Phi_{p,q;\eta_{\mathrm t}}) = \operatorname{supp}(\psi_{p,q}) = [a_{p,q},b_{p,q}]$. We thus define the single-chaoslet neural approximation as $\widehat{\mathfrak{C}}^{\boldsymbol{\alpha},\rho,\eta_{\mathrm t}}_{(p,q)}(t,\omega) \eqdef \Phi_{p,q;\eta_{\mathrm t}}(t)\,\Phi_{\boldsymbol{\alpha},\rho}(\boldsymbol{Z}_{\boldsymbol\alpha;p,q})$, which is predictable and square-integrable. Moreover, 
\begin{equation*}
 \left\|
\mathfrak C^{\boldsymbol{\alpha}}_{(p,q)}
-
\widehat{\mathfrak C}^{\boldsymbol{\alpha},\rho,\eta_{\mathrm t}}_{(p,q)}
\right\|_{\mathcal{H}_T^2}  \le
\left\|
(\psi_{p,q}-\Phi_{p,q;\eta_{\mathrm t}})
u^{(T)}_{\boldsymbol\alpha;p,q}
\right\|_{\mathcal{H}_T^2} + \left\|
\Phi_{p,q;\eta_{\mathrm t}}
\left(
u^{(T)}_{\boldsymbol\alpha;p,q}
-\Phi_{\boldsymbol{\alpha},\rho}(\boldsymbol{Z}_{\boldsymbol\alpha;p,q}) \right)\right\|_{\mathcal{H}_T^2}
\end{equation*}
The first term on the right-hand side of the inequality if finite and, since $\|u^{(T)}_{\boldsymbol\alpha;p,q}\|_{L^2(\Omega)}=1$, it is equal to $\frac{2^{1+p/2}}{\sqrt3}\sqrt{\eta_{\mathrm t}}$ (see above).  The second term, instead is easily bounded by $\rho \|\Phi_{p,q;\eta_{\mathrm t}}\|_{L^2([0,T])}$, which can be bounded by $1+
\|\Phi_{p,q;\eta_{\mathrm t}}-\psi_{p,q}\|_{L^2([0,T])}
=
1+
\frac{2^{1+p/2}}{\sqrt3}\sqrt{\eta_{\mathrm t}}$. In particular, 
\begin{equation*}
 \left\|
\mathfrak C^{\boldsymbol{\alpha}}_{(p,q)}
-
\widehat{\mathfrak C}^{\boldsymbol{\alpha},\rho,\eta_{\mathrm t}}_{(p,q)}
\right\|_{\mathcal{H}_T^2}  \le
\frac{2^{1+p/2}}{\sqrt3}\sqrt{\eta_{\mathrm t}}
+
\rho
\left(
1+
\frac{2^{1+p/2}}{\sqrt3}\sqrt{\eta_{\mathrm t}}
\right)
\end{equation*}
As a consequence, given $\delta>0$, we can opportunely choose $\eta_{\mathrm t}$ and then $\rho>0$ such that the previous quantity is bounded by $\delta$.  At this point, we need to prove that $\widehat{\mathfrak C}^{\boldsymbol{\alpha},\rho,\eta_{\mathrm t}}_{(p,q)}$ is produced by the $\operatorname{NeuralChaos}$ architecture in Section \ref{sec::section_3}. We therefore choose a deterministic grid $=(t_i)_{i=1}^{M}$, $0=t_0<t_1<\cdots<t_M=T$
$a_{p,q}$ and all Brownian sampling times needed to compute
are at-most evaluated only at these grid times only.
Step 1 in Subsection \ref{s:Proofs__ss:MainThrm___sss:Lifting} gives a
lower-triangular linear lift that computes these Brownian Haar coordinates from the sampled Brownian values. The row-wise head at information time $a_{p,q}$ is chosen to be the $\operatorname{ReLU-MLP}$ $\Phi_{\boldsymbol{\alpha}; p, q, \rho}$. Instead, the
deterministic time network is chosen to be $\Phi_{p,q;\eta_{\mathrm t}}$.
Therefore, one row/head/mask block of $\operatorname{NeuralChaos}$ realizes $(t,\omega)\to \Phi_{p,q;\eta_{\mathrm t}}\Phi_{\boldsymbol{\alpha}; p, q, \rho}$. Notice that if the literal Section~\ref{sec::section_3} mask convention includes the additional threshold factor $\sigma\left(\frac{t-a_{p,q}}{\lambda}\right)$, we replace the time factor by $\Phi_{p,q;\eta_{\mathrm t}}(t)
\sigma\left(\frac{t-a_{p,q}}{\lambda}\right)$. Since $\Phi_{p,q;\eta_{\mathrm t}}$ is bounded and the modification is
supported on a deterministic interval of length $O(\lambda)$, the additional $\mathcal{H}_T^2$-error can be made arbitrarily small by taking $\lambda>0$ sufficiently small; this error is absorbed into the one-chaoslet tolerance. We now approximate the finite scalar chaoslet sum
$\overline{Y}=\sum_{m\in \overline{F}}\beta_m m$. We set $B_{\overline{F}} \eqdef \sum_{m\in F}|\beta_m$, and for every non-constant $m\in \overline{F}$, construct a $\operatorname{NeuralChaos}$ block $\widehat{\mathfrak C}^{\boldsymbol{\alpha},\rho,\eta_{\mathrm t}}_{(p,q)}$ as above such that\footnote{The constant basis element $T^{-1/2}$ is represented exactly if the architecture includes a constant output bias. Otherwise, it can be approximated
arbitrarily well by a deterministic head and a deterministic time mask; this
error is also absorbed into the same tolerance.} $\left\|
\mathfrak C^{\boldsymbol{\alpha}}_{(p,q)}
-
\widehat{\mathfrak C}^{\boldsymbol{\alpha},\rho,\eta_{\mathrm t}}_{(p,q)}
\right\|_{\mathcal{H}_T^2} \leq \frac{\varepsilon}{2 (1+B_{\overline{F}})}$. Now, we define $\widehat{\overline{Y}} \eqdef \sum_{\overline{\mathcal{F}}} \beta_m \widehat{\mathfrak C}^{\boldsymbol{\alpha},\rho,\eta_{\mathrm t}}_{(p,q)}$. The $\operatorname{NeuralChaos}$ class is closed under finite parallelization: we use one block for each $\widehat{\mathfrak C}^{\boldsymbol{\alpha},\rho,\eta_{\mathrm t}}_{(p,q)}$ and sum the resulting scalar outputs in the final
affine layer. Whence $\widehat{\overline{Y}} \in \mathfrak{N}\mathfrak{C}_{T,D,1}$, and together with $\|X_{\cdot}-Y_{\cdot}\|_{\mathcal{H}^2_T}\le \frac{\varepsilon}{2}$, we obtain the density in the scalar case. 

We now prove the quantitative statement. Assume that $X_{\cdot}\in\mathcal H_T^2(\mathbb R^d)$ is $S$-compressible, with $S>1/2$, and belongs to $\mathbb D_T^{s,2:d}$, with $s>0$. By Proposition~\ref{prop:sparse_and_smooth}, for every $N\in\mathbb N_{+}$ and every $P\in\mathbb N$, there exist distinct chaoslets $\mathfrak m^{(1)},\ldots,\mathfrak m^{(N)}\in\cup_{k=0}^{P}\mathfrak C_k$ and coefficients $a_1,\ldots,a_N\in\mathbb R^d$ such that
\begin{equation*}
\left\|
X_{\cdot}
-
\sum_{\nu=1}^{N}a_\nu \mathfrak m^{(\nu)}
\right\|_{\mathcal H_T^2(\mathbb R^d)}
\lesssim_X
N^{-\left(S-\frac12\right)}
+
(1+P)^{-s/2}.
\end{equation*}
Set $Y_{N,P}
\eqdef
\sum_{\nu=1}^{N}a_\nu \mathfrak m^{(\nu)}$. Let $C_X^{\rm app}>0$ be an $X$-dependent constant for the preceding estimate. Since $N\ge1$ and $(1+P)^{-s/2}\le1$, there exists a constant $A_X>0$, depending only on $X$, such that $\|Y_{N,P}\|_{\mathcal H_T^2(\mathbb R^d)} \le
A_X$ for all $N$ and $P$. For example, one may take $A_X
\eqdef
1+\|X_{\cdot}\|_{\mathcal H_T^2(\mathbb R^d)}+2C_X^{\rm app}$. Since the chaoslets are orthonormal and the retained chaoslets are distinct,
\begin{equation*}
\sum_{\nu=1}^{N}|a_\nu|_{\mathbb R^d}
\le
\sqrt{N}
\left(
\sum_{\nu=1}^{N}|a_\nu|_{\mathbb R^d}^2
\right)^{1/2}
=
\sqrt{N}\,
\|Y_{N,P}\|_{\mathcal H_T^2(\mathbb R^d)}
\le
\sqrt{N}A_X.
\end{equation*}
Now, we define, $\rho_{N,X,\varepsilon}\eqdef
\frac{\varepsilon}{8(1+\sqrt{N}A_X)}.
$
For each retained chaoslet $\mathfrak m^{(\nu)}$, construct a scalar $\operatorname{NeuralChaos}$ approximant $\widehat{\mathfrak m}^{(\nu)}$ as above with one-chaoslet error at most $\rho_{N,X,\varepsilon}$. This is achieved by using stochastic-head accuracy $\rho_{N,X,\varepsilon}$ and by choosing the corresponding time-mask parameter $\eta_{\mathrm t}^{(\nu)}$ sufficiently small. We define $\widehat X_{N,P,\varepsilon}
\eqdef
\sum_{\nu=1}^{N}a_\nu \widehat{\mathfrak m}^{(\nu)}$. By finite parallelization and final affine summation, $\widehat X_{N,P,\varepsilon}\in\mathfrak N\mathfrak C_{T,D,d}$. Moreover,
\begin{equation*}
\begin{aligned}
\|Y_{N,P}-\widehat X_{N,P,\varepsilon}\|_{\mathcal H_T^2(\mathbb R^d)}
&\le
\sum_{\nu=1}^{N}
|a_\nu|_{\mathbb R^d}
\|\mathfrak m^{(\nu)}-\widehat{\mathfrak m}^{(\nu)}\|_{\mathcal H_T^2}
\\
&\le
\sqrt{N}A_X\,\rho_{N,X,\varepsilon}
\le
\varepsilon.
\end{aligned}
\end{equation*}
Therefore,
\begin{equation*}
\begin{aligned}
\|X_{\cdot}-\widehat X_{N,P,\varepsilon}\|_{\mathcal H_T^2(\mathbb R^d)}
&\le
\|X_{\cdot}-Y_{N,P}\|_{\mathcal H_T^2(\mathbb R^d)}
+
\|Y_{N,P}-\widehat X_{N,P,\varepsilon}\|_{\mathcal H_T^2(\mathbb R^d)}
\\
&\lesssim_X
N^{-\left(S-\frac12\right)}
+
(1+P)^{-s/2}
+
\varepsilon.
\end{aligned}
\end{equation*}
This proves the quantitative approximation estimate.

It remains to record the complexity. Since each retained chaoslet belongs to $\cup_{k=0}^{P}\mathfrak C_k$, its Wiener-chaos degree is at most $P$. Hence its stochastic Hermite-Haar factor contains at most $P$ active scalar Brownian Haar coordinates. Each such coordinate is a linear combination of three Brownian samples. Thus the causal lifting stage contributes at most $3P$ non-zero linear coefficients per retained chaoslet, up to a universal constant. The deterministic time mask has depth $1$, width $6$, and at most $18$ non-zero parameters. By the definition of $D_{\mathrm H}(P,\rho)$, $W_{\mathrm H}(P,\rho)$, and $S_{\mathrm H}(P,\rho)$, the stochastic head for each retained chaoslet has depth, width, and number of non-zero parameters bounded by $D_{\mathrm H}(P,\rho_{N,X,\varepsilon})$, $W_{\mathrm H}(P,\rho_{N,X,\varepsilon})$, $S_{\mathrm H}(P,\rho_{N,X,\varepsilon})$, respectively.

After parallelizing the $N$ retained chaoslet blocks, the depth is the maximum of the block depths, while widths and numbers of non-zero parameters add. The final affine layer implements the vector coefficients $a_\nu\in\mathbb R^d$, $\nu=1,\ldots,N$, and contributes at most $Nd$ additional non-zero parameters. Therefore, for a universal constant $C>0$,
\begin{equation*}
\operatorname{depth}
\left(
\widehat X_{N,P,\varepsilon}
\right)
\le
C
\left[
1+
D_{\mathrm H}
\left(
P,\rho_{N,X,\varepsilon}
\right)
\right],
\end{equation*}
\begin{equation*}
\operatorname{width}
\left(
\widehat X_{N,P,\varepsilon}
\right)
\le
C
\left[
d+
N
\left(
1+
W_{\mathrm H}
\left(
P,\rho_{N,X,\varepsilon}
\right)
\right)
\right],
\end{equation*}
and
\begin{equation*}
\operatorname{size}
\left(
\widehat X_{N,P,\varepsilon}
\right)
\le
C\,N
\left[
P+d+1+
S_{\mathrm H}
\left(
P,\rho_{N,X,\varepsilon}
\right)
\right].
\end{equation*}
Finally, the finite sampling grid is obtained by collecting all Brownian sampling times required by the retained chaoslets, together with the corresponding information times $a_{p,q}$. Since each retained chaoslet uses at most $3P$ Brownian sample times and one information time, the grid can be chosen with $M_{\mathrm{grid}} \le C\,N(P+1)$. This completes the proof of Theorem~\ref{thrm:happytimes}.
\end{proof}

\section{Proof of the results in Subsection \ref{subsec:negative-results}}\label{app::negative}
\begin{proof}[{Proof of Proposition~\ref{prop:goodnews}}]
    Fix $\delta\in (0,T)$, set
$
N_\delta\eqdef \big\lceil \tfrac{T}{\delta}\big\rceil
$,
and set $t_i\eqdef (i\delta)\wedge T$ for each $i=0,\dots,N_{\delta}$.  Define the ``Euler-Maruyama update-rule'' as the map
\[
f_\delta:[0,\infty)\times \mathbb{R}^{d_X+2D}\to \mathbb{R}^{d_X},
\qquad
f_\delta(t,x,w,u)
\eqdef
x+b(t,x)\bigl(((t+\delta)\wedge T)-t\bigr)+\Sigma(t,x)(u-w)
.
\]
Since $b$ and $\Sigma$ are Lipschitz, they are Borel-measurable; thus, so is $f_\delta$.  
Now define recursively
$
X^{(0)}\eqdef \xi
$
and for each $i=0,\dots,N_{\delta}-1$ define
$
X^{(i+1)}
\eqdef
f_\delta(t_i,X^{(i)},W_{t_i},W_{t_{i+1}})$.
A simple induction using Gr\"{o}nwall's inequality shows that, for each $i=0,\dots,N_{\delta}$, we have 
$
X^{(i)}\in L^2(\mathcal{F}_{t_i};\mathbb{R}^{d_X})
$.
Now, define the piecewise-constant process by
\[
\hat{X}_t^{\delta}
\eqdef
\sum_{i=0}^{N_\delta-1}
\mathbf{1}_{[t_i,t_{i+1})}(t)\,
\mathbb{E}\big[X^{(i)}\mid \mathcal{F}_{t_i}\big]
\]
where $0\le t\le T$.
Since $X^{(i)}$ is $\mathcal{F}_{t_i}$-measurable, one has
$
\mathbb{E}[X^{(i)}\mid \mathcal{F}_{t_i}]
=
X^{(i)}
$ 
$\mathbb{P}$-a.s.\ for each $i=0,\dots,N_{\delta}-1$; 
therefore
$
\hat{X}_t^{\delta}
=
\sum_{i=0}^{N_\delta-1}
\mathbf{1}_{[t_i,t_{i+1})}(t)\,
X^{(i)}
$
and so $\hat X^{\delta}\in \mathfrak{B}$.

It remains to show that the family of processes $(\hat{X}^{\delta}_\cdot)_\delta$ converges to $X_{\cdot}$ in $\mathcal{H}_T^2(\mathbb R^{d_X})$.  Consider the classical continuous-time Euler-Maruyama approximations $(\bar{X}^{\delta}_{\cdot})_{\delta>0}$ to $X_{\cdot}$; namely,
\[
    \bar{X}_t^{\delta}
\eqdef
    X^{(i)}
    +
    b(t_i,X^{(i)})(t-t_i)
    +
    \Sigma(t_i,X^{(i)})(W_t-W_{t_i})
\] 
for $t\in [t_i,t_{i+1})$ and each $i=0,\dots,N_{\delta}-1$.
The main result of~\cite{maruyama1955continuous} implies $\mathcal{S}_T^2(\mathbb R^{d_X})$-norm convergence, i.e.,
\begin{equation*}
    \lim_{\delta\downarrow 0}
    \,
    \mathbb{E}\Big[\sup_{0\le t\le T}\|X_t-\bar{X}_t^{\delta}\|^2\Big]=0.
\end{equation*}
Hence, by the usual inequalities between the $\mathcal{H}_T^2$ and $\mathcal{S}_T^2$ norms, we have
\begin{equation}
\label{eq:convergence_EM_approx}
    0
\le 
    \lim_{\delta\downarrow 0}
        \,
    \mathbb{E}\biggl[
        \int_0^T\, \|X_t-\bar{X}_t^{\delta}\|^2\,dt
    \biggr]
\le 
    \lim_{\delta\downarrow 0}
        \,
    T
    \mathbb{E}\biggl[
        \sup_{0\le t\le T}\|X_t-\bar{X}_t^{\delta}\|^2
    \biggr]
=
    0
.
\end{equation}
We only need to show that $\hat{X}^{\delta}_{\cdot}$ converges to $\bar{X}_{\cdot}^{\delta}$ to obtain the conclusion.  
For any $i=0,\dots,N_{\delta}-1$ and any $t\in [t_i,t_{i+1})$ we have
\[
\bar{X}_t^{\delta}-\hat{X}_t^{\delta}
=
b(t_i,X^{(i)})(t-t_i)
+
\Sigma(t_i,X^{(i)})(W_t-W_{t_i}).
\]
Therefore, using the basic inequality $(a+b)^2\le 2a^2+2b^2$,
$$
\begin{aligned}
\mathbb{E}\int_{t_i}^{t_{i+1}}
\|\bar{X}_t^{\delta}-\hat{X}_t^{\delta}\|^2\,dt
\le
2\,
\mathbb{E}\int_{t_i}^{t_{i+1}}
\|b(t_i,X^{(i)})\|^2 (t-t_i)^2\,dt
+
2
\mathbb{E}\int_{t_i}^{t_{i+1}}
\|\Sigma(t_i,X^{(i)})(W_t-W_{t_i})\|^2\,dt.
\end{aligned}
$$
Since
$
\int_{t_i}^{t_{i+1}} (t-t_i)^2\,\ud t\le \frac{\delta^3}{3}
$; while
$
    \mathbb{E}\big[\|\Sigma(t_i,X^{(i)})(W_t-W_{t_i})\|^2\mid \mathcal{F}_{t_i}\big]
=
    (t-t_i)\|\Sigma(t_i,X^{(i)})\|_{HS}^2
$; then
\begin{equation}
\label{eq:bounding_nicething}
    \mathbb{E}\int_{t_i}^{t_{i+1}}
    \|\bar{X}_t^{\delta}-\hat{X}_t^{\delta}\|^2\,dt
\le
    C\delta^2\,
\mathbb{E}\big[1+\|X^{(i)}\|^2\big]
.
\end{equation}
Using the standard uniform second-moment bound for the Euler scheme, there exists some finite constant $C_T>0$ such that, after summing over $i=0,\dots,N_\delta-1$ in~\eqref{eq:bounding_nicething}, we find that
\begin{equation}
\label{eq:bound_me_bby}
\|\bar{X}^{\delta}-\hat X^{\delta}\|_{\mathcal{H}_T^2(\mathbb R^{d_X})}^2
=
\mathbb{E}\int_0^T \|\bar{X}_t^{\delta}-\hat{X}_t^{\delta}\|^2\,dt
\le
C_T\,\delta
.
\end{equation}
Upon setting $\delta>0$ small enough, together~\eqref{eq:convergence_EM_approx} and~\eqref{eq:bound_me_bby} yield
\[
    \lim_{\delta\downarrow 0}
        \,
    \mathbb{E}\biggl[
        \int_0^T\, \|\hat{X}^{\delta}_t-X_t\|^2\,dt
    \biggr]
\le 
    \lim_{\delta\downarrow 0}
        \,
    \mathbb{E}\biggl[
        \int_0^T\, \|X_t-\bar{X}_t^{\delta}\|^2\,dt
    \biggr]
    +
    \lim_{\delta\downarrow 0}
        \,
    \|\bar{X}^{\delta}-\hat X^{\delta}\|_{\mathcal{H}_T^2(\mathbb R^{d_X})}^2
=
    0,
\]
which completes the proof.
\end{proof}

\begin{proof}[{Proof of Proposition~\ref{prop:badnews}}]
We first prove the negative topological result, namely (i), before using some of the tools developed therein to conclude the negative measure-theoretic result (ii).
\hfill\\
\noindent
\textbf{Proof of (ii) - Topological Non-Representativeness:} 
Let us begin by simplifying our problem as follows.  
For each $N\in\mathbb N_+$, define $\mathfrak{S}_N$ to denote the set of all processes $X\in \mathcal{H}_T^2(\mathbb R^{d_X})$ for which there exist deterministic times
$
0=t_0\leq t_1\leq \cdots \leq t_N=T
$
and random variables $\xi_i\in L^2(\mathcal F_{t_i};\mathbb R^{d_X})$, $i=0,\dots,N-1$, such that
\[
    X_t
=
    \sum_{i=0}^{N-1}
    \mathbf{1}_{[t_i,t_{i+1})}(t)\,\xi_i
    \qquad
    \ud t\otimes d\mathbb P\text{-a.e.}
\]
By construction,
$
\mathfrak{B} \subseteq \bigcup_{N=1}^{\infty}\mathfrak{S}_N
$.
By~\eqref{eq:badguys_by_name}, for any $X_{\cdot}\in\mathfrak{B}$, there exists some ``frequency'' $0<\delta <T$ such that $X$ is constant on the deterministic grid intervals
$
[t_i,t_{i+1})\cap[0,T],
$
so $X\in \mathfrak{S}_N$ for $N\eqdef \lceil T/\delta\rceil$.

Since a set is meagre if it is contained in a countable union of nowhere dense sets, it is enough to show that each $\mathfrak{S}_N$ is nowhere dense in $\mathcal{H}_T^2(\mathbb R^{d_X})$.  To this end, we show that each $\mathfrak{S}_N$ is a closed subset of $\mathcal{H}_T^2(\mathbb R^{d_X})$ with empty interior.
To see that each $\mathfrak{S}_N$ is closed, let $(X^n)_{n\geq 1}\subseteq \mathfrak{S}_N$ converge in $\mathcal{H}_T^2(\mathbb R^{d_X})$ to some $X$. For each $n$, choose a representation
$$
X^n_t
=
\sum_{i=0}^{N-1}
\mathbf{1}_{[t_i^n,t_{i+1}^n)}(t)\,\xi_i^n,
\qquad
\xi_i^n\in L^2(\mathcal F_{t_i^n};\mathbb R^{d_X}),
$$
with
$
0=t_0^n\leq t_1^n\leq \cdots \leq t_N^n=T.
$
Passing to a subsequence, we may assume that, for every $i=0,\dots,N$, the sequence $(t_i^n)_n$ converges to some $t_i\in[0,T]$, and moreover is monotone in $n$; consequently
$
0=t_0\leq t_1\leq \cdots \leq t_N=T.
$
Now, fix $i\in\{0,\dots,N-1\}$ such that $t_i<t_{i+1}$. Let $I\subset (t_i,t_{i+1})$ be a compact sub-interval. Then, for all sufficiently large $n$, one has
$
I\subset [t_i^n,t_{i+1}^n),
$
and therefore
$
X_t^n=\xi_i^n
$
$\ud t\otimes d\mathbb{P}$-a.e.\ on $I\times\Omega$.  
Hence, for $n,m\in \mathbb{N}_+$ large enough,
\[
        |I|\,
        \mathbb{E}\big[\|\xi_i^n-\xi_i^m\|^2\big]
    =
        \mathbb{E}\biggl[
            \int_I \|X_t^n-X_t^m\|^2\,\ud t
        \biggr]
    \le
        \|X^n-X^m\|_{\mathcal{H}_T^2(\mathbb R^{d_X})}^2
\]
where $|I|$ denotes the Lebesgue measure/length of the interval $I$.
Since $(X^n)_n$ is Cauchy in $\mathcal{H}_T^2(\mathbb R^{d_X})$, it follows that $(\xi_i^n)_n$ is Cauchy in $L^2$, hence converges in $L^2$ to some $\xi_i\in L^2(\Omega;\mathbb R^{d_X})$.
The $\mathbb{P}$-a.s.\ continuity of $W_{\cdot}$ implies that, since each $\xi_i^n$ is $\mathcal F_{t_i^n}$-measurable and $(t_i^n)_n$ is monotone with limit $t_i$, the continuity of the Brownian filtration implies that $\xi_i$ is $\mathcal F_{t_i}$-measurable; and thus
$
X_t=\xi_i
$
$\ud t\otimes d\mathbb{P}$-a.e.\ on $I\times \Omega$.
Since $I\subset (t_i,t_{i+1})$ was arbitrary, we upgrade this deduction to the entire sub-interval $(t_i,t_{i+1})$; i.e.\
$
X_t=\xi_i
$
$\ud t\otimes d\mathbb{P}$-a.e.\ on $(t_i,t_{i+1})\times \Omega$.
For indices $i$ with $t_i=t_{i+1}$, set $\xi_i\eqdef 0$. 
We thus conclude that
\begin{equation}
\label{eq:nice_representation}
    X_t
=
    \sum_{i=0}^{N-1}
    \,
        \mathbf{1}_{[t_i,t_{i+1})}(t)\,\xi_i
    \qquad
    \ud t\otimes \ud \mathbb P\text{-a.e.}
,
\end{equation}
with $\xi_i\in L^2(\mathcal F_{t_i};\mathbb R^{d_X})$. Hence $X_{\cdot}\in \mathfrak{S}_N$, and so $\mathfrak{S}_N$ is closed.

We now only need to show that it has empty interior.  Again, fix an arbitrary $X\in \mathfrak{S}_N$, $\varepsilon>0$, and choose any representation of $X_{\cdot}$ as in~\eqref{eq:nice_representation}.  
Then there must exist some $j\in \{0,\dots,N-1\}$ such that $t_j<t_{j+1}$, and hence some
$
t_j<c<d<t_{j+1}
$.  
Let $B_t\eqdef W_t^{(1)}$ be the first coordinate of $W$, let $e_1^{d_X}$ be the first canonical vector of $\mathbb R^{d_X}$, fix some $\eta>0$, and consider the ``bridge-type'' process
$
    R_t^{\eta}
\eqdef
    \eta\,\mathbf{1}_{[c,d)}(t)\,(B_t-B_c)e_1^{d_X}
$.
Then, for each $\eta>0$, the process $R^{\eta}_{\cdot}$ belongs to $\mathcal{H}_T^2(\mathbb R^{d_X})$ and its norm can be estimated as
\begin{equation}
\label{eq:norme_me_up}
        \|R_{\cdot}^{\eta}\|_{\mathcal{H}_T^2(\mathbb R^{d_X})}^2
    =
        \eta^2\int_c^d \,
        \mathbb E[(B_t-B_c)^2]\,\ud t
    =
        \eta^2\int_c^d (t-c)\,\ud t
.
\end{equation}
Now, fix $\varepsilon>0$; and note that by~\eqref{eq:norme_me_up} we may choose $\eta>0$ small enough so that
$
    \|R_{\cdot}^{\eta}\|_{\mathcal{H}_T^2(\mathbb R^{d_X})}<\varepsilon.
$
We claim that $X_{\cdot}+R_{\cdot}^{\eta} \notin \mathfrak{S}_N$ for any $\eta>0$.  This would conclude our argument, since it would show that for any $\varepsilon>0$ the $\varepsilon$-ball about $X_{\cdot}$ in $\mathcal{H}_T^2(\mathbb R^{d_X})$ must contain a point outside of $\mathfrak{S}_N$; meaning that $\mathfrak{S}_N$ has \textit{empty interior} (since every point in an open subset of a metric space must contain all sufficiently small balls about that point).
\hfill\\
\noindent
Indeed, on the interval $[c,d)$, the process $X$ is equal to the constant random variable $\xi_j$, whereas $t\mapsto R_t^{\eta}(\omega)$ is continuous and non-constant on $[c,d)$ for $\mathbb P$-a.e.\ $\omega$. Therefore, for $\mathbb P$-a.e.\ $\omega$, the path
$$
t\mapsto (X_t+R_t^{\eta})(\omega)
$$
is continuous and non-constant on $[c,d)$. On the other hand, any element of $\mathfrak{S}_N$ has, for $\mathbb P$-a.e.\ $\omega$, a path which is piecewise constant on a deterministic partition with at most $N$ intervals. Such a path, if continuous on $[c,d)$, must be constant there; which cannot happen for $X_{\cdot}+R_{\cdot}^{\eta}$.
Hence $X_{\cdot}+R_{\cdot}^{\eta}$ does not belong to $\mathfrak{S}_N$ for any $\eta>0$.
Thus, $\mathfrak{B}$ is contained in a countable union of the nowhere dense sets $(\mathfrak{S}_N)_{N=1}^{\infty}$; whence $\mathfrak{B}$ is meagre in $\mathcal{H}_T^2(\mathbb R^{d_X})$.

\noindent
\textbf{Proof of (i) - Probabilistic non-representativeness:}  
Our approach is simple, using the inclusion $\mathfrak{B} \subseteq \bigcup_{N=1}^{\infty}\mathfrak{S}_N$ we have
\begin{equation}
\label{eq:measure_me_this}
    \mu_G(\mathfrak{B})
\le 
    \sum_{n=1}^{\infty}\,
        \mu_G(\mathfrak{S}_n)
.
\end{equation}
Our strategy is then simply to show that each $\mu_G(\mathfrak{S}_n)$ on the right-hand side of~\eqref{eq:measure_me_this} has $\mu_G$-measure zero; which we do by showing that each $\mathfrak{S}_n$ is contained in a countable union of proper closed linear subspaces; whence each of which must have $\mu_G$-measure $0$ since $Q$ has non-singular covariance operator, and thus so must $\mathfrak{S}_n$.  

More specifically, fix an arbitrary $N\in \mathbb{N}_+$. 
For each pair of rational numbers $a,b\in \mathbb Q\cap[0,T]$ with
\begin{equation}
\label{eq:big_bad_indices}
    0\le a<b\le T
\,\,\mbox{ and }\,\,
    \frac{T}{2N}<b-a
.
\end{equation}
Define the linear subspace
$
M_{a,b}
\eqdef
\big\{
    X\in \mathcal{H}_T^2(\mathbb R^{d_X}):
    \exists\ \xi\in L^2(\mathcal F_a;\mathbb R^{d_X})
    \text{ s.t.}
    X_t=\xi
    \quad
    \ud t\otimes \ud \mathbb P\text{-a.e. on }[a,b)\times\Omega
\big\}
$.  We claim $M_{a,b}$ is closed in $\mathcal{H}_T^2(\mathbb R^{d_X})$; indeed, if $(X_{\cdot}^{(n)})_{n=1}^{\infty}$ is a sequence in $M_{a,b}$ converging to some $X_{\cdot}\in \mathcal{H}_T^2(\mathbb R^{d_X})$ then, for each $n$, we may choose $\xi_n\in L^2(\mathcal{F}_a;\mathbb R^{d_X})$ such that
\[
X_t^n=\xi_n
\qquad
\ud t\otimes \ud \mathbb P\text{-a.e.\ on }[a,b)\times\Omega
.
\]
Then
$
(b-a)\,\mathbb E[\|\xi_n-\xi_m\|^2]
=
\mathbb E\Big[\int_a^b \|X_t^n-X_t^m\|^2\,\ud t\Big]
\le
\|X^n-X^m\|_{\mathcal{H}_T^2(\mathbb R^{d_X})}^2
$.
Thus $(\xi_n)_n$ is Cauchy in $L^2(\mathcal F_a;\mathbb R^{d_X})$, hence converges to some $\xi\in L^2(\mathcal F_a;\mathbb R^{d_X})$. Passing to the limit gives
$
X_t=\xi
$ $
\ud t\otimes \ud \mathbb P\text{-a.e. on }[a,b)\times\Omega
$; whence the limiting process $X_{\cdot}$ also belongs to $M_{a,b}$.
Next, we now show that each $M_{a,b}$ has positive co-dimension in $\mathcal{H}_T^2(\mathbb R^{d_X})$.  
Let $B_t\eqdef W_t^{(1)}$ be the first coordinate of the Brownian motion, let $e_1^{d_X}$ be the first canonical vector of $\mathbb R^{d_X}$, and define
$
Y_t\eqdef \mathbf 1_{[a,b)}(t)\,(B_t-B_a)e_1^{d_X}
$.
Then $Y_{\cdot}\in \mathcal{H}_T^2(\mathbb R^{d_X})$, but $Y_{\cdot}\notin M_{a,b}$; since, on $[a,b)$, the process $t\mapsto (B_t-B_a)e_1^{d_X}$ is not $\ud t\otimes \ud \mathbb P$-a.e.\ equal to any $\mathcal F_a$-measurable random variable. Hence $M_{a,b}\subsetneq \mathcal{H}_T^2(\mathbb R^{d_X})$.  Thus, $\mu_G(M_{a,b})=0$ for every $(a,b)$ satisfying~\eqref{eq:big_bad_indices}.  

In order to conclude that $\mathfrak{S}_N$ has $\mu_G$-measure zero, we will capture it in the union of countably many of these ``bad'' $M_{a,b}$.  Indeed, let  
$\Gamma_N$ denote the countable set of $(a,b)\in \mathbb{Q}^2$ satisfying~\eqref{eq:big_bad_indices}.
Clearly $\Gamma_N$ is countable; thus it remains to show that
\begin{equation}
\label{eq:capture_me_this}
    \mathfrak{S}_N 
\subseteq 
    \bigcup_{(a,b)\in \Gamma_N} M_{a,b}
.
\end{equation}
Indeed, let $X\in \mathfrak{S}_N$. Then, again, by~\eqref{eq:badguys_by_name} there exists a deterministic partition
$
0=t_0\le t_1\le \cdots \le t_N=T
$
such that $X$ is constant on each interval $[t_i,t_{i+1})$. Since
$
\sum_{i=0}^{N-1}(t_{i+1}-t_i)=T,
$
there must exist some $j\in [N-1]$ such that
$
t_{j+1}-t_j\ge \frac{T}{N}.
$
By density of the rationals in the reals, there exists some $a,b\in \mathbb{Q}$ such that
$
t_j<a<b<t_{j+1}
$ and $b-a>\frac{T}{2N}$.  
Since $X$ is constant on $[t_j,t_{j+1})$, it is in particular constant on $[a,b)$. Moreover the corresponding constant random variable is $\mathcal F_{t_j}$-measurable, hence also $\mathcal F_a$-measurable because $\mathcal F_{t_j}\subseteq \mathcal F_a$, thus $X\in M_{a,b}$, which proves~\eqref{eq:capture_me_this}. Hence,
\[
    \mu_G(\mathfrak{S}_N)
\le
    \sum_{(a,b)\in\Gamma_N}\mu_G(M_{a,b})
=
    0.
\]
Since $N\in\mathbb N_+$ was arbitrary, using~\eqref{eq:measure_me_this} gives
\[
    \mu_G(\mathfrak{B})
=
    0,
\]
which completes the proof.
\end{proof}

\section{Auxiliary SDE facts}\label{app:auxiliary-sde-facts}
This section collects a result which we were not able to find in the literature and that we use in Section \ref{s:MathFin}.
\begin{proof}[Proof of Lemma~\ref{lem:control}]
Let $L \in [0,+\infty)$ be a common Lipschitz constant for $b$ and  $\Sigma$ in the variables $(x,a)$, uniformly in $t$. Set $M_b\eqdef\sup_{(t,x,a)}
|b(t,x,a)|_{\mathbb R^{d_X}}$ and $M_\Sigma\eqdef\sup_{(t,x,a)}
\|\Sigma(t,x,a)\|_{\mathrm F}$. For a fixed $\alpha\in \mathcal{H}_T^2(\mathbb R^{d_\alpha})$, the coefficients $(t,\omega,x)\mapsto b(t,x,\alpha_t(\omega))$ and $(t,\omega,x)\mapsto \Sigma(t,x,\alpha_t(\omega))$ are progressively measurable, globally Lipschitz in $x$, and bounded. Hence the SDE admits a unique strong solution with continuous adapted paths. In particular, $X^\alpha$ is predictable. We next show that $X^\alpha\in \mathcal{H}_T^2(\mathbb R^{d_X})$. By the elementary inequality $|u+v+w|^2\le3(|u|^2+|v|^2+|w|^2)$, Cauchy--Schwarz, and It\^o isometry, for every $t\in[0,T]$,
\begin{equation*}
 \begin{aligned}
\mathbb E\left[|X_t^\alpha|_{\mathbb R^{d_X}}^2\right]
&\le
3\mathbb E\left[|x_0|_{\mathbb R^{d_X}}^2\right]
+
3\mathbb E\left[
\left|
\int_0^t b(s,X_s^\alpha,\alpha_s)\,ds
\right|_{\mathbb R^{d_X}}^2
\right] +
3\mathbb E\left[
\left|
\int_0^t \Sigma(s,X_s^\alpha,\alpha_s)\,\ud W_s
\right|_{\mathbb R^{d_X}}^2
\right]
\\
&\le
3\mathbb E\left[|x_0|_{\mathbb R^{d_X}}^2\right]
+
3T^2M_b^2
+
3TM_\Sigma^2.
\end{aligned}   
\end{equation*}
Integrating over $t\in[0,T]$ gives $X^\alpha\in \mathcal{H}_T^2(\mathbb R^{d_X})$. It remains to prove Lipschitz continuity of $\Gamma$. In order to do so, we let $\alpha,\beta\in \mathcal{H}_T^2(\mathbb R^{d_\alpha})$, and write
$\Delta X_t\eqdef X_t^\alpha-X_t^\beta$ and $\Delta\alpha_t\eqdef\alpha_t-\beta_t$. Then
\begin{equation*}
    \Delta X_t
=
\int_0^t
\left[
b(s,X_s^\alpha,\alpha_s)
-
b(s,X_s^\beta,\beta_s)
\right]\ud s
+
\int_0^t
\left[
\Sigma(s,X_s^\alpha,\alpha_s)
-
\Sigma(s,X_s^\beta,\beta_s)
\right]\ud W_s.
\end{equation*}
Using Cauchy--Schwarz, It\^o isometry, and the Lipschitz property, we obtain
\begin{equation*}
\begin{aligned}
\mathbb E\left[|\Delta X_t|_{\mathbb R^{d_X}}^2\right]
&\le
2T
\int_0^t
\mathbb E\left[
\left|
b(s,X_s^\alpha,\alpha_s)
-
b(s,X_s^\beta,\beta_s)
\right|_{\mathbb R^{d_X}}^2
\right]\ud s
\\
&\quad+
2
\int_0^t
\mathbb E\left[
\left\|
\Sigma(s,X_s^\alpha,\alpha_s)
-
\Sigma(s,X_s^\beta,\beta_s)
\right\|_{\mathrm F}^2
\right]\ud s
\\
&\le
2(T+1)L^2
\int_0^t
\mathbb E\left[
|\Delta X_s|_{\mathbb R^{d_X}}^2
+
|\Delta\alpha_s|_{\mathbb R^{d_\alpha}}^2
\right]\ud s.
\end{aligned}
\end{equation*}
Set $f(t)\eqdef\mathbb E\left[|\Delta X_t|_{\mathbb R^{d_X}}^2\right]$, $C_0\eqdef2(T+1)L^2$. Then $f(t)
\le
C_0\int_0^t f(s)\,\ud s
+
C_0\int_0^t
\mathbb E\left[
|\Delta\alpha_s|_{\mathbb R^{d_\alpha}}^2
\right]\ud s$. By Gr\"{o}nwall's lemma, $f(t)
\le
C_0e^{C_0T}
\int_0^T
\mathbb E\left[
|\Delta\alpha_s|_{\mathbb R^{d_\alpha}}^2
\right]\ud s$. Integrating again over $t\in[0,T]$, we obtain
\begin{equation*}
    \|X^\alpha-X^\beta\|_{\mathcal{H}_T^2(\mathbb R^{d_X})}^2
\le
T C_0 e^{C_0T}
\|\alpha-\beta\|_{\mathcal{H}_T^2(\mathbb R^{d_\alpha})}^2.
\end{equation*}
This concludes the proof. 
\end{proof}

\bibliographystyle{acm}
\bibliography{Bookeaping/References}

\end{document}